\documentclass[reqno]{amsart}
\usepackage{amsmath,amssymb,amsthm}
\usepackage{thmtools}
\usepackage{mathptmx}      % use Times fonts if available on your TeX system
\usepackage{mathrsfs}
\usepackage[scr=boondox]{mathalpha}
\usepackage{latexsym}
\usepackage{bm}
\usepackage[usenames,dvipsnames]{xcolor}
\usepackage{graphicx,pifont}
\usepackage{multicol,multirow}
\usepackage{ulem,cancel}
\usepackage{tikz}
\usetikzlibrary[patterns]
\usepackage{comment}

\usepackage{hyperref}
\usepackage{cleveref}
\crefname{equation}{}{}
\crefname{enumi}{}{}
\crefname{figure}{Figure}{Figure}
\crefname{subsection}{Subsection}{Subsections}
\crefname{lemma}{Lemma}{Lemma}
\crefname{theorem}{Theorem}{Theorem}
\crefname{proposition}{Proposition}{Proposition}
\crefname{section}{Section}{Section}
\crefname{appendix}{Appendix}{Appendix}
\crefname{definition}{Definition}{Definition}
\crefname{corollary}{Corollary}{Corollary}
\crefname{table}{Table}{Table}
\crefname{remark}{Remark}{Remark}

\usepackage{array}
\usepackage{caption}
\usepackage{subcaption}
\usepackage{txfonts}
\usepackage{mathtools}
\usepackage[shortlabels]{enumitem}
\usepackage{float}
\usepackage{booktabs}
\usepackage{dutchcal}
\usepackage[T1]{fontenc}
\usepackage{bigfoot}
\usepackage[numbered,framed]{matlab-prettifier}

\newtheorem{theorem}{Theorem}[section]
\newtheorem{lemma}[theorem]{Lemma}
\newtheorem{corollary}[theorem]{Corollary}

\theoremstyle{definition}
\newtheorem{definition}[theorem]{Definition}

\theoremstyle{remark}
\newtheorem{remark}[theorem]{Remark}
\numberwithin{equation}{section}
\numberwithin{figure}{section}

\newcommand{\bs}[1]{\boldsymbol{#1}}

\newcommand{\bb}[1]{{\mbox{\sffamily \bfseries #1}}}
\newcommand{\oname}[1]{\textrm{#1}}
\newcommand{\abs}[1]{\left|#1\right|}
\newcommand{\nrm}[1]{\left|\left|#1\right|\right|}
\newcommand{\eqdef}{\stackrel{\mathrm{def}}{=\joinrel=}}

\newcommand{\phf}[1]{{{#1}+1/2}}

\newcommand{\mhf}[1]{{{#1}-1/2}}
\newcommand{\powth}{{\textrm{\scriptsize th}}}

\newcommand{\per}{{\textrm{per}}}

\newcommand{\Pe}{{\textrm{Pe}}}
\newcommand{\hf}{{\frac{1}{2}}}

\lstMakeShortInline"

\begin{document}

\title[HV methods for advection-diffusion equations]{On a class of supraconvergent Hermite-type methods for advection-diffusion equations}
\author[X.~Zeng]{Xianyi Zeng}
\address{Department of Mathematics, Lehigh University, Bethlehem PA 18015, United States. \\
         Tel.: +1-610-758-3745}
%\address{Department of Mathematical Sciences, University of Texas at El Paso, El Paso TX 79902, United States.\\
%         Tel.: +1-915-747-6759}
%\email[X.~Zeng]{xzeng@utep.edu}
\email[X.~Zeng]{xyzeng@lehigh.edu}
\date{\today}

\subjclass[2010]{65M12 \and 35L65}

\keywords{
    Linear advection-diffusion equation;
    Hybrid-variable discretization;
    Hermite-type method;
    Supraconvergence;
    Arbitrary order of accuracy;
    Stability analysis.
}

\begin{abstract}
We develop and analyze a class of Hermite-type methods for linear advection-diffusion equations in the case of one-dimensional periodic domains.
The methods evolve both nodal values and cell averages and are therefore referred to as hybrid-variable (HV) methods.
Using Hermite interpolation theory, we construct the HV methods to arbitrary order of accuracy, and prove that the method is supraconvergent in the sense that the spatial order of accuracy can be up to two orders larger than the local truncation error. 
%Supraconvergence makes the HV methods attractive than existing ones, as they offer the same order of accuracy with much less computations.
In the second part, we prove all central HV methods for linear diffusion equations and linear advection-diffusion equations are stable at the semi-discretized level.
Next, we prove the convergence of all stable HV schemes and provide error estimate in the discrete $L^2$-norm of the semi-discretized solution, which extends the supraconvergence from a single mode to all sufficiently smooth initial data.
The theoretical predictions on the accuracy and stability of the methods are verified by extensive numerical tests.
\end{abstract}

\maketitle

\section{Introduction}
\label{sec:intro}
Advection-diffusion equations (ADE) model both transport and diffusion processes that are common in natural and industrial applications. 
Indeed, many numerical methods for the Navier-Stokes equations originate in the discretization of linear ADEs; yet numerical solutions to these model equations remain challenging because the relative importance of the advection and diffusion mechanisms can vary substantially with the physical parameters and the mesh size.
%Particularly, diffusion-dominated problems require faithful representation of the dissipative behavior and usually impose severe restrictions on the time step size.
General discussions of the analytical and numerical analysis of time-dependent ADEs can be found, for example, in~\cite{WHundsdorfer:2003a}.

This work focuses on the analysis of a class of Hermite-type discretization methods with arbitrary order of accuracy, with the overaching goal of constructing highly efficient and numerically stable methods for problems of multiple scales.
Particularly, high-order methods for advection-diffusion problems are attractive because they can resolve smooth solutions of multiple scales with substantially fewer degrees of freedom than low-order schemes.
Among many attempts, Hermite methods enrich the discrete data by tracking derivatives at grid points, in addition to approximating nodal solutions; thus they tend to achieve higher order of accuracy on a given stencil of mesh cells and facilitate parallel computations.
Efforts of this kind include the Hermite methods using staggered grids~\cite{JGoodrich:2006a}, their adaptive-order variants reported in~\cite{RChen:2012a}, and the Hermite WENO schemes~\cite{JQiu:2003a}.
These methods demonstrate that augmenting the nodal solutions with additional nodal derivatives can yield highly accurate discretizations while maintaining a small stencil.

In recent years, methods that evolve in addition cell averages gain increasing attention as they synergize extremely well with conservation laws.
One of the earliest example is probably van Leer's Scheme V~\cite{BvanLeer:1977b}, which combines cell averages with nodal solutions at cell faces to construct piecewise-quadratic representations.
Later methods in this category include the multi-moment constrained finite volume method~\cite{SIi:2009a} and the active flux framework~\cite{PRoe:2017a,RAbgrall:2023a}.

The hybrid-variable (HV) framework considered in this paper is a direct generalization of van Leer's Scheme V to arbitrary order of accuracy.
At the semi-discretized level, as assumed in this work, it evolves both nodal values and cell averages simultaneous in time; and the spatial discretization is constructed by approximating spatial derivatives using linear combinations of these two types of variables.
In an earlier work~\cite{XZeng:2019a}, the general HV framework was developed for linear advection equations, and it was proved that the HV methods are supraconvergent, meaning the order of the method is higher than that of the local truncation error.
%The supraconvergence property makes the HV schemes more attractive than other Hermite-type methods when efficiency is concerned.
Later developments of the HV discretization theory includes establishing a stability barrier for advection equations~\cite{XZeng:2026a} and its application in the construction of a fourth-order scheme on the smallest stencil for Euler flows~\cite{XZeng:2024a}.
The full classification of stable HV discretizations of linear advection equations, however, remains an open problem.

%Inspired by the supraconvergence property and its potential to construct more efficient schemes than those in the literature, 
In this work, we extend the HV discretization to the model linear advection-diffusion equations $w_t+cw_x-\nu w_{xx}=0, c, \nu > 0$ on a one-dimensional periodic domain and study its accuracy and stability properties.
In the HV framework, the cell-average equation is obtained by integrating the governing equation over each cell, where the advective flux at nodes are directly evaluated using nodal solutions and the diffusive flux is computed by approximating $\nu w_x$ by a central discrete operator dentoed $\nu[\mathcal{D}^c_xw]$.
The nodal equation instead is obtained by approximating $cw_x$ and $\nu w_{xx}$ with discrete operators $c[\mathcal{D}_x]$ and $\nu[\mathcal{D}_{xx}^c]$, respectively.
Here $[\mathcal{D}_x]$, $[\mathcal{D}^c_x]$, and $[\mathcal{D}^c_{xx}]$ are constructed as linear combinations of nearby nodal and cell-averaged solutions on a continuous stencil, and they are called HV discrete differential operators (HV-DDO).
A review of the general HV discretization framework and an explicit constfuction of these HV-DDOs are offered in~\cref{sec:prelim}.
Next, we will prove supraconvergence for a single-mode problem in~\cref{sec:acry} that if the formal orders of $[\mathcal{D}_x]$, $[\mathcal{D}_x^c]$, and $[\mathcal{D}_{xx}^c]$ are $P_1$, $P_2$, and $P_3$, respectively, the spatial order of the resulting HV scheme is $\min(P_1+2,P_2,P_3+2)$ (\cref{thm:acry}).
In~\cref{sec:stab}, we prove that all central HV methods are stable at the semi-discretized level.
Due to the complexity of this stability proof, the main result is splitted to two parts -- \cref{thm:stab_asym} focuses on the linear diffusion equations and the stability result is then extended to all ADEs in~\cref{thm:stab_ade}.
\cref{sec:conv} proves the convergence of the semi-discretized HV solution to the analytical one providing the underlying HV scheme is stable and consistent; particularly we prove the supraconvergence extends from a single mode to arbitrary initial data with sufficient regularity by carefully crafting a correction term in the classical Duhamel-type error estimate of the discrete $L^2$-norm of the numerical error. 
Both the accuracy and stability properties are verified numerically by extensive tests in~\cref{sec:num}, and~\cref{sec:concl} concludes the paper.

\section{Preliminaries}
\label{sec:prelim}
We consider the Cauchy problem of the equation:
\begin{equation}\label{eq:prelim_eqn}
  w_t + cw_x - \nu w_{xx} = 0\;,
\end{equation}
on a closed interval $x\in\Omega = [0,\;L]$ and $t\in[0,\;T]$ with periodic boundary conditions $w(0,t) = w(L,t)$ and $w_x(0,t) = w_x(L,t)$; here $c$ and $\nu>0$ are constants.
%$c\ne0$ is the constant advection velocity and $\nu>0$ is the constant diffusivity.

The numerical method for solving~\cref{eq:prelim_eqn} utilizes a uniform grid with the cells $\mathcal{C}_j^{j+1}=[x_j,\;x_{j+1}]$, $j=0,\cdots,N-1$, where $N$ is the number of cells; the grid nodes are denoted $x_j=jh$, $h=L/N$.
In the hybrid-variable discretization framework, we consider approximations to both nodal solutions and cell-averaged solutions at the semi-discretized level:
\begin{equation}\label{eq:prelim_semi}
  w_j(t)\approx w(x_j,\;t)\;,\quad
  \overline{w}_{j+1/2}(t)\approx \frac{1}{h}\int_{\mathcal{C}_j^{j+1}}w(x,\;t)dx\;.
\end{equation}
%Following the strategy of the method of lines, the HV semi-discretization will be paired with appropriate time-marching method and in this case, we seek the discrete solutions:
%\begin{equation}\label{eq:prelim_full}
%  w_j^n\approx w(x_j,\;t^n)\;,\quad
%  \overline{w}_{j+1/2}^n\approx \frac{1}{h}\int_{\mathcal{C}_j^{j+1}}w(x,\;t^n)dx\;,
%\end{equation}
%where $t^n = n\Delta t$, $0\le n\le M$, and $\Delta t=T/M$.

\begin{comment}
In the rest of this section, we briefly review the properties of the equation~\cref{eq:prelim_eqn} itself and the HV-discretization framework, which includes the general formula to construct arbitrary-order approximations to $w_x$ using both nodal and cell-averaged variables.
\end{comment}
In the rest of this section, we briefly review the HV-discretization framework to construct arbitrary-order approximations to $w_x$ and then provide similar formula for approximating $w_{xx}$, denoted by the HV discrete differential operators or HV-DDOs $[\mathcal{D}_x]$ and $[\mathcal{D}_{xx}]$, respectively.

\subsection{Hybrid-variable discretization}
\label{sec:prelim_hv}
%The essence of the hybrid-variable discretization is to use a linear combination of close-by nodal and cell-averaged variables to approximate the spatial derivatives.
%To this end, 
The semi-discretization of~\cref{eq:prelim_eqn} for the cell averages are obtained by integrating the equation over $\mathcal{C}_j^{j+1}$:
\begin{subequations}\label{eq:prelim_hv_semi}
  \begin{equation}\label{eq:prelim_hv_semi_cell}
    \overline{w}'_{\phf{j}} + \frac{c}{h}(w_{j+1}-w_j) - \frac{\nu}{h}\left([\mathcal{D}_x^cw]_{j+1}-[\mathcal{D}_x^cw]_j\right) = 0\;,
  \end{equation}
  where $[\mathcal{D}^c_xw]_j$ denotes an approximation to $w_x(x_j,\,t)$ that uses a centered stencil (denoted by the superscript $^c$);
  the semi-discretization formula for the nodal variables is given by:
  \begin{equation}\label{eq:prelim_hv_semi_node}
    w'_j+c[\mathcal{D}_xw]_j-\nu[\mathcal{D}_{xx}^cw]_j = 0\;,
  \end{equation}
  where $[\mathcal{D}_xw]_j$ approximates $w_x(x_j,\,t)$ with either an upwind-biased stencil or a centered one; and $[\mathcal{D}_{xx}^cw]_j$ is an approximation to $w_{xx}(x_j,\,t)$ with a centered stencil.
\end{subequations}

We consider HV-DDOs $[\mathcal{D}_x]$, $[\mathcal{D}_x^c]$, and $[\mathcal{D}_{xx}^c]$ constructed on continuous stencils:
\begin{align}
  \label{eq:prelim_hv_dx}
  [\mathcal{D}_xw]_j &= \frac{1}{h}\sum_{k=-l}^{r-1}\alpha_k\overline{w}_{\phf{j+k}}+\frac{1}{h}\sum_{k=-l'}^{r'}\beta_kw_{j+k}\;, \\
  \label{eq:prelim_hv_dx_c}
  [\mathcal{D}_x^cw]_j &= \frac{1}{h}\sum_{k=-p}^{p-1}\alpha_k^c\overline{w}_{\phf{j+k}}+\frac{1}{h}\sum_{k=-p'}^{p'}\beta_k^cw_{j+k}\;, \\
  \label{eq:prelim_hv_dxx_c}
  [\mathcal{D}_{xx}^cw]_j &= \frac{1}{h^2}\sum_{k=-q}^{q-1}\mathcal{a}_k\overline{w}_{\phf{j+k}}+\frac{1}{h^2}\sum_{k=-q'}^{q'}\mathcal{b}_kw_{j+k}\;, 
\end{align}
where the non-negative integers $r,l,p,q$ denote the stencil for cell-averaged variables and the primed-integers denote that for nodal variables.
The constants $\alpha_k,\,\beta_k,\,\alpha^c_k,\,\beta^c_k,\,\mathcal{a}_k,\,\mathcal{b}_k$ are determined by requiring certain orders of accuracy for these operators.
For the stencils to be continuous, these numbers satisfy: $\max(0,l-1)\le l'\le l\;,\ \max(0,r-1)\le r'\le r\;,\ l+r+l'+r'\ge1$, $\max(0,p-1)\le p'\le p\;,\ p+p'\ge1$, and $\max(0,q-1)\le q'\le q\;,\ q+q'\ge1$.
The formal order or the local truncation error of these operators is defined below.
\begin{definition}\label{def:prelim_hv_ddo}
  The operators $[\mathcal{D}_x]$, $[\mathcal{D}_x^c]$, and $[\mathcal{D}_{xx}^c]$ are $P_1^\powth\textrm{\sc-}$, $P_2^\powth\textrm{\sc-}$, and $P_3^\powth\textrm{\sc-}$order accurate for some integer $P_1,P_2,P_3\ge1$ if for any $C^\infty$ function $w(x)$:
  \begin{align}
    \label{eq:prelim_hv_dx_ddo}
    [\mathcal{D}_xw^\star]_j &= \partial_xw(x_j)+c_1\partial_x^{P_1+1}w(x_j)h^{P_1}+O(h^{P_1+1})\;, \\
    \label{eq:prelim_hv_dx_c_ddo}
    [\mathcal{D}_x^cw^\star]_j &= \partial_xw(x_j)+c_2\partial_x^{P_2+1}w(x_j)h^{P_2}+O(h^{P_2+2})\;, \\
    \label{eq:prelim_hv_dxx_c_ddo}
    [\mathcal{D}_{xx}^cw^\star]_j &= \partial_x^2w(x_j)+c_3\partial_x^{P_3+2}w(x_j)h^{P_3}+O(h^{P_3+2})\;, 
  \end{align}
  where the operators in the left-hand side are constructed using the exact cell-averaged and nodal quantities denoted by $\overline{w}^\star_{\phf{j+k}}$ and $w^\star_{j+k}$, and $c_1$, $c_2$, and $c_3$ are nonzero constants that are independent of the function $w$ and the cell size $h$.
\end{definition}
\begin{remark}\label{rm:prelim_hv_central}
  %The even and odd powers of $h$ in the Taylor series expansions cancel in $[\mathcal{D}_x^c]$ and $[\mathcal{D}_{xx}^c]$, respectively; thus the higher-order terms in~\cref{eq:prelim_hv_dx_c_ddo} and~\cref{eq:prelim_hv_dxx_c_ddo} are $O(h^{P_{2,3}+2})$ instead of $O(h^{P_{2,3}+1})$.
  The even and odd powers of $h$ cancel in the expansions of $[\mathcal{D}_x^c]$ and $[\mathcal{D}_{xx}^c]$, respectively; thus the higher-order terms are $O(h^{P_{2,3}+2})$ instead of $O(h^{P_{2,3}+1})$.
\end{remark}
\begin{comment}
\begin{remark}\label{rm:prelim_hv_poly}
Following~\cref{def:prelim_hv_ddo}, $[\mathcal{D}_x]$ and $[\mathcal{D}_x^c]$ are $P^\powth\textrm{\sc-}$order accurate if and only if for all polynomial $w(x)$ of degree no more than $P$, there is:
\begin{equation}\label{eq:prelim_hv_dx_poly}
  [\mathcal{D}_xw^\star]_j = [\mathcal{D}_x^cw^\star]_j = w'(x_j)\;,
\end{equation}
whereas there exists a polynomial of degree $P+1$ in the case of $[\mathcal{D}_x]$ or $P+2$ in the case of $[\mathcal{D}_x^c]$, such that~\cref{eq:prelim_hv_dx_poly} does not hold.
Similarly, $[\mathcal{D}_{xx}^c]$ is $P^\powth\textrm{\sc-}$order accurate if and only if for all polynomial $w(x)$ of degree no more than $P+1$, there is:
\begin{equation}\label{eq:prelim_hv_dxx_poly}
  [\mathcal{D}_{xx}^cw^\star]_j = w''(x_j)\;,
\end{equation}
whereas there exists a polynomial of degree $P+3$ such that~\cref{eq:prelim_hv_dxx_poly} does not hold.
These criteria are commonly known as the $P$-exactness conditions.
\end{remark}
\end{comment}

\subsection{Hybrid-variable discrete differential operators of arbitrary order}
\label{sec:prelim_ddo}
HV-DDOs with optimal accuracy can be constructed utilizing Hermite interpolation polynomials, for example, these coefficients for $[\mathcal{D}_x]$ and $[\mathcal{D}_x^c]$ are derived in an earlier work~\cite{XZeng:2019a} and summarized in the following results.
\begin{theorem}[{\cite[Theorem 1]{XZeng:2019a}}]\label{thm:prelim_ddo_dx}
  There exists a unique operator $[\mathcal{D}_x]$ given by~\cref{eq:prelim_hv_dx}, which is $(l+r+l'+r')^\powth\textrm{\sc-}$order accurate.
  The coefficients of this operator are:
  \begin{subequations}\label{eq:prelim_ddo_dx_coef}
    \small
    \begin{align}
      \label{eq:prelim_ddo_dx_alpha_neg}
      \alpha_\nu &= -(1-\delta_{ll'})\frac{2}{l^2}C_{-l}^{l,r}C_{-l}^{l,r'}-\sum_{k=-l'}^\nu\frac{2(1+k(\zeta_k^{l,r}+\zeta_k^{l',r'}))}{k^2}C_k^{l,r}C_k^{l',r'}\;,\quad\forall -l\le\nu<0\;;\\
      \label{eq:prelim_ddo_dx_alpha_pos}
      \alpha_\nu &= \sum_{k=\nu+1}^{r'}\frac{2(1+k(\zeta_k^{l,r}+\zeta_k^{l',r'}))}{k^2}C_k^{l,r}C_k^{l',r'}+(1-\delta_{rr'})\frac{2}{r^2}C_r^{l,r}C_r^{l',r}\;,\quad\forall 0\le\nu\le r-1\;;\\
      \label{eq:prelim_ddo_dx_beta}
      \beta_0 &= 2(\zeta_0^{l,r}+\zeta_0^{l',r'})\;;\quad
      %\label{eq:prelim_ddo_dx_beta_nz}
      \beta_\nu = -\frac{2}{\nu}C_\nu^{l,r}C_\nu^{l',r'}\;,\quad -l'\le\nu\le r',\;\ \nu\ne0\;.
    \end{align}
  \end{subequations}
  Here $\delta_{ll'}$ and $\delta_{rr'}$ are Kronecker deltas that equal $1$ when $l=l'$ and $r=r'$, respectively, and equal $0$ otherwise; the constants $\zeta$'s and $C$'s are defined by:
  \begin{equation}\label{eq:prelim_ddo_const}
    \zeta_{\nu}^{l,r} = \sum_{-l\le k\le r,\;k\ne\nu}\frac{1}{\nu-k} = H_{l+\nu}-H_{r-\nu}\;,\quad
    C_{\nu}^{l,r} = \frac{l!r!}{(l+\nu)!(r-\nu)!}\;,
  \end{equation}
  for all $-l \le \nu \le r$.% and $\nu\in\mathbb{Z}$.
  Here $H_n$ are Harmonic numbers with the convention $H_0=0$.
\end{theorem}
As a direct corollary of~\cref{thm:prelim_ddo_dx}, we may set $l=r=p$ and $l'=r'=p'$ to obtain:
\begin{corollary}\label{cor:prelim_ddo_dx_c}
  There exists a unique operator $[\mathcal{D}_x^c]$ given by~\cref{eq:prelim_hv_dx_c}, which is $2(p+p')^\powth\textrm{\sc-}$order accurate.
  The coefficients are:
  \begin{subequations}\label{eq:prelim_ddo_dx_c_coef}
    \small
    \begin{align}
      \label{eq:prelim_ddo_dx_c_alpha_neg}
      \alpha_\nu^c &= -\alpha_{-1-\nu}^c\;,\quad\forall -p\le\nu<0\;, \\
      \label{eq:prelim_ddo_dx_c_alpha_pos}
      \alpha_\nu^c &= \sum_{k=\nu+1}^{p'}\!\frac{2(1\!+\!k(\zeta_k^{p,p}\!+\!\zeta_k^{p'\!,p'}))}{k^2}C_k^{p,p}C_k^{p'\!,p'}+(1\!-\!\delta_{pp'})\frac{2}{p^2}C_p^{p,p}C_p^{p'\!,p},\ \forall 0\le\nu\le p-1; \\
      \label{eq:prelim_ddo_dx_c_beta}
      \beta_0^c   &= 0\;;\quad
      %\label{eq:prelim_ddo_dx_c_beta_nz}
      \beta_\nu^c = -\frac{2}{\nu}C_\nu^{p,p}C_\nu^{p'\!,p'} = -\beta_{-\nu}^c\;,\quad -p'\le\nu\le p'\;,\ \nu\ne0\;.
    \end{align}
  \end{subequations}
  These coefficients are {\it anti-symmetric} in the sense that $\alpha_\nu^c+\alpha_{-1-\nu}^c=0,\;-p\le\nu\le(p-1)$ and $\beta_{\nu}^c+\beta_{-\nu}^c=0,\;-p'\le\nu\le p'$.
\end{corollary}

The coefficients for $[\mathcal{D}_{xx}^c]$ can be derived in a similar way as given below, where the proof also provides the connection between HV discretization and Hermite interpolations.
\begin{theorem}\label{thm:prelim_ddo_dxx}
  There exists a unique operator $[\mathcal{D}_{xx}^c]$ given by~\cref{eq:prelim_hv_dxx_c}, which is at least $2(q+q')^\powth\textrm{\sc-}$order accurate.
  The corresponding coefficients are:
  \begin{subequations}\label{eq:prelim_ddo_dxx_c_coef}
    \small
    \begin{align}
      \label{eq:prelim_ddo_dxx_c_a_neg}
      \mathcal{a}_\nu &= \mathcal{a}_{-1-\nu}\;,\quad\forall -q\le\nu<0\;, \\
      \label{eq:prelim_ddo_dxx_c_a_pos}
      \mathcal{a}_\nu &= \sum_{k=\nu+1}^{q'}\!\frac{6(2\!+\!k(\zeta_k^{q,q}\!+\!\zeta_k^{q'\!,q'}))}{k^3}C_k^{q,q}C_k^{q'\!,q'}+(1\!-\!\delta_{qq'})\frac{6}{q^3}C_q^{q,q}C_q^{q'\!,q},\ \forall 0\le\nu\le q-1; \\
      \label{eq:prelim_ddo_dxx_c_b}
      \mathcal{b}_0 &= -\sum_{-q\le k\le q,\,k\ne0}\frac{3}{k^2} - \sum_{-q'\le k\le q',\,k\ne0}\frac{3}{k^2}\;,\quad
      %\label{eq:prelim_ddo_dxx_c_b_nz}
      \mathcal{b}_\nu = -\frac{6}{\nu^2}C_{\nu}^{q,q}C_{\nu}^{q',q'}\;,\quad -q'\le\nu\le q'\;,\ \nu\ne0\;.
    \end{align}
  \end{subequations}
  Furthermore, for this operator, $B_0\eqdef\sum_{k=-q'}^{q'}\mathcal{b}_k<0$.
\end{theorem}
\begin{proof}
  Assume first that $q'=q$ and we'll show there exists a unique operator $[\mathcal{D}_{xx}^c]$ that is at least $(4q-1)^\powth\textrm{\sc-}$order accurate; then a simple symmetry argument shows that this unique operator is actually at least $(4q)^\powth\textrm{\sc-}$order accurate.
  Following~\cref{eq:prelim_hv_dxx_c_ddo}, let\footnote{Here $\mathbb{P}^n$ denotes the space of all polynomials of degree no more than $n$.} $w\in\mathbb{P}^{4q}$ be arbitrary and $W\in\mathbb{P}^{4q+1}$ be a primitive of $w$, we have:
  \begin{displaymath}
    w_{j+k}^\star = W'(x_{j+k})\;,\quad
    \overline{w}_{\phf{j+k}}^\star = \frac{1}{h}(W(x_{j+k+1})-W(x_{j+k}))\;;
  \end{displaymath}
  hence~\cref{eq:prelim_hv_dxx_c_ddo} rewrites as:
  \begin{equation}
    \label{eq:prelim_ddo_dxx_exact}
    W'''(x_j) %&= \frac{1}{h^3}\sum_{k=-q}^{q-1}\mathcal{a}_k(W(x_{j+k+1})-W(x_{j+k}))+\frac{1}{h^2}\sum_{k=-q}^q\mathcal{b}_kW'(x_{j+k}) \\
              = -\frac{1}{h^3}\sum_{k=-q}^q\Delta\mathcal{a}_kW(x_{j+k}) + \frac{1}{h^2}\sum_{k=-q}^q\mathcal{b}_kW'(x_{j+k})\;,
  \end{equation}
  where $\Delta\mathcal{a}_k\eqdef\mathcal{a}_k-\mathcal{a}_{k-1},\,-q\le k\le q$ and $\mathcal{a}_{-q-1}=\mathcal{a}_q\equiv0$; by definition, $\sum_{k=-q}^q\Delta\mathcal{a}_k=0$.

  Following the Hermite interpolation theory, any $W\in\mathbb{P}^{4q+1}$ can be uniquely written as
  \begin{equation}\label{eq:prelim_ddo_herm}
    W(x) = \sum_{k=-q}^qW(x_{j+k})h_k(x) + \sum_{k=-q}^qW'(x_{j+k})g_k(x)\;,
  \end{equation}
  where $h$'s and $g$'s are the fundamental polynomials of the first and second kind of Hermite interpolation, see~\cite{GSzego:1975a,XZeng:2019a}.
  Then the required coefficients~\cref{eq:prelim_ddo_dxx_c_coef} can be obtained by:
  \begin{displaymath}
    \Delta\mathcal{a}_\nu = -h^3h'''_\nu(x_j)\;,\quad
    \mathcal{b}_\nu = h^2g'''_\nu(x_j)\;,\quad -q\le \nu\le q\;.
  \end{displaymath}
  The coefficients in the other case $q'=q-1$ are obtained similarly, now using fundamental polynomials of unbalanced Hermite interpolations.
\end{proof}

At last, we remark that the methodology applies equally to non-uniform grids. 
In this case, the existence and uniqueness results of~\cref{thm:prelim_ddo_dx} and~\cref{cor:prelim_ddo_dx_c} remains true but the coefficients take different values.
For~\cref{thm:prelim_ddo_dxx}, however, the least order of accuracy one may establish is $4q-1$ instead of $4q$ because one typically loses the symmetry of $\mathcal{a}$ and $\mathcal{b}$ coefficients when the grid is not uniform.

\section{Accuracy Analysis of a Single Mode}
\label{sec:acry}
The purpose of this section is to establish the relation between the local truncation errors ($P_1$, $P_2$, and $P_3$ as given in~\cref{def:prelim_hv_ddo}) and the actual spatial order of the semi-discretization~\cref{eq:prelim_hv_semi}.
%study how well the semi-discretized solutions to~\cref{eq:prelim_hv_semi} approximate the exact cell-averaged and nodal solutions, where the HV-DDOs are given by the general form~\cref{eq:prelim_hv_dx}--\cref{eq:prelim_hv_dxx_c}.
It is important to realize that for HV methods, the relation between the order of the DDO and the spatial order of the method is typically different from that of conventional finite difference or finite volume methods.
For example, in the case of advection equations (i.e., $\nu=0$) one has the following {\it supraconvergence} result:
\begin{theorem}[{\cite[Theorem 2]{XZeng:2019a}}]\label{thm:acry_dx}
  Suppose a $P^\powth\textrm{\sc-}$order HV-DDO $[\mathcal{D}_x]$ is used to discretize the advection equation $w_t+cw_x=0,\; c\ne0$, such that $P\ge1$ and $cb_0>0$ where $b_0\eqdef\beta_{-l'}+\cdots+\beta_{r'}$.
  Then the semi-discretized HV method is $(P+1)^\powth\textrm{\sc-}$order accurate when the initial data is given by a simple wave $w_0(x)=e^{i\kappa x}$.
\end{theorem}
The condition $cb_0>0$ is called the {\it generalized upwind condition} and it is satisfied if and only if there are more upwind variables than the downwind ones in~\cref{eq:prelim_hv_dx}, see~\cite[Proposition 5]{XZeng:2019a}.
In the rest of this paper, we always assume $c\ge0$ and thus upwind-baised stencil include more unknowns to the left of the focal node than those to the right.

%The situation is certainly more complicate with the advection-diffusion equations.
To extend the analysis to advection-diffusion equations,
\begin{comment}
we assume a simple wave~\cref{eq:prelim_prop_fr_simp} and compare it to the solutions to the ODE system~\cref{eq:prelim_hv_semi}.
To this end, the exact nodal and cell-averaged solutions corresponding to~\cref{eq:prelim_prop_fr_simp} are:
\end{comment}
we assume a simple wave with wavenumber $\kappa$ and compare it to the solutions to the ODE system~\cref{eq:prelim_hv_semi}.
To this end, the exact nodal and cell-averaged solutions are:
\begin{equation}\label{eq:acry_exact}
  w_j^\star(t) = e^{-ic\kappa t-\nu\kappa^2 t}e^{ij\theta}\;,\quad
  \overline{w}_{\phf{j}}^\star(t) = \frac{1}{i\theta}e^{-ic\kappa t-\nu\kappa^2 t}e^{ij\theta}\left(e^{i\theta}-1\right)\;,
\end{equation}
where $\theta\eqdef\kappa h$ is the numerical wavenumber.

To find the solution to the semi-discretized system, 
%Next, we focus on extending the analysis to general case; the method is the same but the calculation is considerably more complex. 
let us consider an HV-discretization with a $P_1^\powth\textrm{\sc-}$order $[\mathcal{D}_x]$, a $P_2^\powth\textrm{\sc-}$order $[\mathcal{D}_x^c]$, and a $P_3^\powth\textrm{\sc-}$order $[\mathcal{D}_{xx}^c]$, whose coefficients are given by~\cref{eq:prelim_hv_dx}, \cref{eq:prelim_hv_dx_c}, and~\cref{eq:prelim_hv_dxx_c}, respectively.
To this end, it is not difficult to see that the solutions to the ODE system~\cref{eq:prelim_hv_semi} are given by:
%~\cref{eq:acry_semi_sol}, where $A(t)$ and $N(t)$ solves the ODEs:
\begin{equation}\label{eq:acry_semi_sol}
  \overline{w}_{\phf{j}}(t) = \frac{1}{i\theta}A(t)e^{ij\theta}(e^{i\theta}-1)\;,\quad
  w_j(t) = N(t)e^{ij\theta}\;,
\end{equation}
where the scalar functions $A(t)$ and $N(t)$ solve the ODE system:
\begin{equation}\label{eq:acry_semi_ode}
  \left\{\begin{array}{l}
    A'(t) + \nu\kappa^2\alpha^c(\theta)A(t) + \left(ic\kappa+\nu\kappa^2\beta^c(\theta)\right)N(t) = 0\;, \\ \vspace*{-.1in} \\
    N'(t) + \left(ic\kappa\alpha(\theta)+\nu\kappa^2\mathcal{a}(\theta)\right)A(t) + 
            \left(ic\kappa\beta(\theta)+\nu\kappa^2\mathcal{b}(\theta)\right)N(t) = 0\;, \\ \vspace*{-.1in} \\
    A(0) = N(0) = 1\;.
  \end{array}\right.
\end{equation}
where the coefficient functions are:
\begin{align}
  \label{eq:acry_semi_fun_dx_c}
  &\alpha^c(\theta) = \frac{1}{(i\theta)^2}\sum_{k=-p}^{p-1}\alpha_k^ce^{ik\theta}(e^{i\theta}-1)\;,\quad
  &&\beta^c(\theta) = \frac{1}{i\theta}\sum_{k=-p'}^{p'}\beta_k^ce^{ik\theta}\;; \\
  \label{eq:acry_semi_fun_dx}
  &\alpha(\theta) = \frac{1}{(i\theta)^2}\sum_{k=-l}^{r-1}\alpha_ke^{ik\theta}(e^{i\theta}-1)\;,\quad
  &&\beta(\theta) = \frac{1}{i\theta}\sum_{k=-l'}^{r'}\beta_ke^{ik\theta}\;; \\
  \label{eq:acry_semi_fun_dxx_c}
  &\mathcal{a}(\theta) = \frac{1}{(i\theta)^3}\sum_{k=-q}^{q-1}\mathcal{a}_ke^{ik\theta}(e^{i\theta}-1)\;,\quad
  &&\mathcal{b}(\theta) = \frac{1}{(i\theta)^2}\sum_{k=-q'}^{q'}\mathcal{b}_ke^{ik\theta}\;.
\end{align}
These functions are tightly related to the order conditions.
For example, because $[\mathcal{D}_x]$ is $P_1^\powth\textrm{\sc-}$order accurate, plugging in the simple wave solution to~\cref{eq:prelim_hv_dx_ddo} we have:% for some constant $c_1\ne0$ that is independent of $\theta$:
\begin{displaymath}
  \frac{1}{h}\sum_{k=-l}^{r-1}\alpha_k\frac{1}{i\theta}e^{i(j+k)\theta}(e^{i\theta}-1)+\frac{1}{h}\sum_{k=-l'}^{r'}\beta_ke^{i(j+k)\theta} = i\kappa e^{ij\theta}(1+c_1\theta^{P_1} + O(\theta^{P_1+1}))\;,
\end{displaymath}
where $c_1\ne0$ is independent of $\theta$; or equivalently:
\begin{equation}\label{eq:acry_semi_approx_dx}
  \alpha(\theta) + \beta(\theta) = 1+c_1\theta^{P_1}+O(\theta^{P_1+1})\;.
\end{equation}
Similarly, there exist non-zero constants $c_2$ and $c_3$ that are independent of $\theta$, such that:
\begin{align}
  \label{eq:acry_semi_approx_dx_c}
  &\alpha^c(\theta) + \beta^c(\theta) = 1+c_2\theta^{P_2}+O(\theta^{P_2+2})\;, \\
  \label{eq:acry_semi_approx_dxx_c}
  \textrm{and}\qquad&\mathcal{a}(\theta) + \mathcal{b}(\theta) = 1+c_3\theta^{P_3}+O(\theta^{P_3+2})\;,
\end{align}
where again the $+2$ in high-order terms is due to central differencing.
To this end, we denote the coefficient matrix corresponding to~\cref{eq:acry_semi_ode} by $\bb{C}(\theta)$:
\begin{equation}\label{eq:acry_semi_mat}
  \bb{C}(\theta) = \left[\begin{array}{cc}
    \nu\kappa^2(1\!-\!\beta^c\!+\!c_2\theta^{P_2}\!+\!O(\theta^{P_2\!+\!2})) & ic\kappa+\nu\kappa^2\beta^c \\ \vspace*{-.1in} \\
    ic\kappa(1\!-\!\beta\!+\!c_1\theta^{P_1}\!+\!O(\theta^{P_1\!+\!1})) + \nu\kappa^2(1\!-\!\mathcal{b}\!+\!c_3\theta^{P_3}\!+\!O(\theta^{P_3\!+\!2})) &
    ic\kappa\beta+\nu\kappa^2\mathcal{b}
  \end{array}\right]\;,
\end{equation}
where we suppressed the dependence on $\theta$ of the coefficient functions for simplicity.
The two eigenvalues of~\cref{eq:acry_semi_mat} are:
\begin{equation}\label{eq:acry_semi_eigs}
  \lambda_{1,2} = \frac{1}{2}\nu\kappa^2(1-\beta^c+\mathcal{b}+c_2\theta^{P_2}) + \frac{1}{2}ic\kappa\beta + O(\theta^{P_2+2})\pm\sqrt{\Delta}\;,
\end{equation}
where:
\begin{align*}
  \Delta 
  =& \left(\frac{1}{2}\nu\kappa^2(1-\beta^c+\mathcal{b}+c_2\theta^{P_2}+O(\theta^{P_2+2}))+\frac{1}{2}ic\kappa\beta\right)^2 - \\
   & \nu\kappa^2(1-\beta^c+c_2\theta^{P_2}+O(\theta^{P_2+2}))(ic\kappa\beta+\nu\kappa^2\mathcal{b}) + \\
   & (ic\kappa+\nu\kappa^2\beta^c)\left(ic\kappa(1-\beta+c_1\theta^{P_1}+O(\theta^{P_1+1}))+\nu\kappa^2(1-\mathcal{b}+c_3\theta^{P_3}+O(\theta^{P_3+2}))\right)\;.
\end{align*}
To have a better understanding of the high-order terms, we need the asymptotic behavior of the coefficient functions as $\theta\to0$.
Because $b_0 = \lim_{\theta\to0}i\theta\beta(\theta) = \sum_{k=-l'}^{r'}\beta_k<\infty$, we have $\beta = O(\theta^{-1})$ (in fact $\beta=O(1)$ if $[\mathcal{D}_x]$ is central, but it does not affect the subsequent estimate).
Similarly, from the second part of~\cref{thm:prelim_ddo_dxx} we have $B_0=\lim_{\theta\to0}(i\theta)^2\mathcal{b}(\theta) \ne0$, hence $\mathcal{b} = -\frac{B_0}{\theta^2} + O(\theta^{-1}) = O(\theta^{-2})$.

The central difference $[\mathcal{D}_x^c]$ given by~\cref{cor:prelim_ddo_dx_c} always gives $\sum_{k=-p'}^{p'}\beta_k^c=0$ since $\beta_{-k}^c+\beta_k^c=0,\ \forall -p'\le k\le p'$; hence $\lim_{\theta\to0}i\theta\beta^c(\theta) = 0$.
Using the coefficients~\cref{eq:prelim_ddo_dx_c_beta} we can compute:
\begin{align*}
  \lim_{\theta\to0}\beta^c(\theta) 
  %&= \lim_{\theta\to0}\frac{1}{i\theta}\sum_{k=1}^{p'}(\beta_k^ce^{ik\theta}+\beta_{-k}^ce^{-ik\theta}) 
  % = \lim_{\theta\to0}\frac{1}{i\theta}\sum_{k=1}^{p'}\left(-\frac{2}{k}C_k^{p,p}C_k^{p',p'}\right)(e^{ik\theta}-e^{-ik\theta}) \\
  &= \sum_{k=1}^{p'}\left(-4\frac{p!p!}{(p+k)!(p-k)!}\frac{p'!p'!}{(p'+k)!(p'-k)!}\right) < 0\;,
\end{align*}
hence $\beta^c = b_0^c +O(\theta)$, where $b_0^c \eqdef \lim_{\theta\to0}\beta^c(\theta) < 0$.

Using these estimates, we may compute:
\begin{align*}
  \Delta 
  %=& \left[\left(\frac{1}{2}\nu\kappa^2(1\!+\!\beta^c\!-\!\mathcal{b}\!+\!c_2\theta^{P_2})+ic\kappa(1\!-\!\frac{1}{2}\beta)\right) +
  %         \left(\nu\kappa^2(-\!\beta^c\!+\!\mathcal{b})+ic\kappa(-\!1\!+\!\beta)\right)\right]^2 - \\
  % & \nu\kappa^2(1-\beta^c+c_2\theta^{P_2})(ic\kappa\beta+\nu\kappa^2\mathcal{b}) + (ic\kappa+\nu\kappa^2\beta^c)\left(ic\kappa(1-\beta+c_1\theta^{P_1})+\nu\kappa^2(1-\mathcal{b}+c_3\theta^{P_3})\right) + \\
  % & O(\theta^{P_2}+c\theta^{P_1+1}+\theta^{P_3+2}) \\
  %=& \left(\frac{1}{2}\nu\kappa^2(1\!+\!\beta^c\!-\!\mathcal{b}\!+\!c_2\theta^{P_2})+ic\kappa(1\!-\!\frac{1}{2}\beta)\right)^2 + (\nu\kappa^2)^2\left(-\beta^cc_2\theta^{P_2}+\beta^cc_3\theta^{P_3}\right) + (ic\kappa)^2c_1\theta^{P_1} + \\
  % & (ic\kappa)\nu\kappa^2\left(\beta^cc_1\theta^{P_1}-c_2\theta^{P_2}+c_3\theta^{P_3}\right) + O(\theta^{P_2}+c\theta^{P_1+1}+\theta^{P_3+2}) \\
  =& \left(\frac{1}{2}\nu\kappa^2(1\!+\!\beta^c\!-\!\mathcal{b}\!+\!c_2\theta^{P_2})+ic\kappa(1\!-\!\frac{1}{2}\beta)\right)^2 + O(\theta^{\min(P_2,P_3)}+c\theta^{P_1})\;.
\end{align*}
Thus we can choose a branch of the square root of $\Delta$, such that:
\begin{equation}\label{eq:acry_semi_branch}
  \sqrt{\Delta} = \frac{1}{2}\nu\kappa^2(1+\beta^c-\mathcal{b}+c_2\theta^{P_2})+ic\kappa(1-\frac{1}{2}\beta)+O(\theta^{\min(P_2,P_3)+2}+c\theta^{P_1+2})\;.
\end{equation}
%Indeed, realizing $\tilde{\lambda}\eqdef\frac{1}{2}\nu\kappa^2(1\!+\!\beta^c\!-\!\mathcal{b}\!+\!c_2\theta^{P_2})+ic\kappa(1\!-\!\frac{1}{2}\beta)=O(\theta^{-2})$ using previous calculations, we may write:
%\begin{displaymath}
%  \Delta = (\tilde{\lambda}+O(\theta^P))^2
%\end{displaymath}
%for some power $P$. 
%Comparing it to the previous formula for $\Delta$, we obtain:
%\begin{displaymath}
%  O(\theta^{\min(P_2,P_3)}+c\theta^{P_1}) = 2\tilde{\lambda}\,O(\theta^P) + O(\theta^{2P}) = O(\theta^{P-2}) + O(\theta^{2P})\;.
%\end{displaymath}
%Therefore one must have $\min(P-2,2P)\ge\min(P_2,P_3)$ if $c=0$ and $\min(P-2,2P)\ge\min(P_1,P_2,P_3)$ if $c\ne0$, which proves~\cref{eq:acry_semi_branch}.
With this choice, we have $\lambda_1\approx ic\kappa+\nu\kappa^2$ as $\theta\to0$; moreover:
\begin{align}
  \label{eq:acry_semi_eig_1}
  \lambda_1 &= ic\kappa + \nu\kappa^2(1+c_2\theta^{P_2}) + O(\theta^{\min(P_2,P_3)+2}+c\theta^{P_1+2})\;, \\
  \label{eq:acry_semi_eig_2}
  %\lambda_2 &= ic\kappa(-1+\beta) + \nu\kappa^2(-\beta^c+\mathcal{b}) + O(\theta^{\min(P_2,P_3)+2}+c\theta^{P_1+2}) = -\frac{B_0\nu\kappa^2}{\theta^2} + O(\theta^{-1})\;.
  \lambda_2 &= -\frac{B_0\nu\kappa^2}{\theta^2} + O(\theta^{-1})\;.
\end{align}
Then we show that both $A(t)$ and $N(t)$ are $O(\theta^{\min(P_2,P_3+2)}+c\theta^{P_1+2})$-approximations to the exact value $e^{-ic\kappa t-\nu\kappa^2t}$.
Using the eigenvalue decomposition of $\bb{C}(\theta)$:
\begin{align*}
  \bb{C}(\theta) %&= \left[\begin{array}{cc}
  %  \nu\kappa^2\alpha^c & ic\kappa+\nu\kappa\beta^c \\
  %  ic\kappa\alpha+\nu\kappa^2\mathcal{a} & ic\kappa\beta+\nu\kappa^2\mathcal{b}
  %\end{array}\right] \\
  &= \left[\begin{array}{cc}
    ic\kappa+\nu\kappa^2\beta^c & ic\kappa+\nu\kappa^2\beta^c \\
    \lambda_1-\nu\kappa^2\alpha^c & \lambda_2-\nu\kappa^2\alpha^c
  \end{array}\right]
  \left[\begin{array}{cc}
    \lambda_1 & 0 \\ 0 & \lambda_2
  \end{array}\right]
  \left[\begin{array}{cc}
    ic\kappa+\nu\kappa^2\beta^c & ic\kappa+\nu\kappa^2\beta^c \\
    \lambda_1-\nu\kappa^2\alpha^c & \lambda_2-\nu\kappa^2\alpha^c
  \end{array}\right]^{-1}\;,
\end{align*}
it follows that:
\begin{subequations}\label{eq:acry_semi_k_sol}
  \begin{align}
    \label{eq:acry_semi_k_sol_a}
    A(t) &= \frac{\lambda_2-ic\kappa-\nu\kappa^2(\alpha^c+\beta^c)}{\lambda_2-\lambda_1}e^{-\lambda_1t}-
            \frac{\lambda_1-ic\kappa-\nu\kappa^2(\alpha^c+\beta^c)}{\lambda_2-\lambda_1}e^{-\lambda_2t}\;, \\
    \label{eq:acry_semi_k_sol_n}
    N(t) &= \frac{(\lambda_1\!-\!\nu\kappa\alpha^c)(\lambda_2\!-\!ic\kappa\!-\!\nu\kappa^2(\alpha^c\!+\!\beta^c))}{(ic\kappa+\nu\kappa^2\beta^c)(\lambda_2-\lambda_1)}e^{-\lambda_1t}-
            \frac{(\lambda_2\!-\!\nu\kappa\alpha^c)(\lambda_1\!-\!ic\kappa\!-\!\nu\kappa^2(\alpha^c\!+\!\beta^c))}{(ic\kappa+\nu\kappa^2\beta^c)(\lambda_2-\lambda_1)}e^{-\lambda_2t}\;.
  \end{align}
\end{subequations}
From~\cref{eq:acry_semi_approx_dx_c} and~\cref{eq:acry_semi_eig_1} we obtain:
\begin{displaymath}
  \lambda_1-ic\kappa-\nu\kappa^2(\alpha^c+\beta^c) = O(\theta^{\min(P_2,P_3)+2}+c\theta^{P_1+2})\,;
\end{displaymath}
together with $\lambda_2=-B_0\nu\kappa^2/\theta^2+O(\theta^{-1})$ one has the estimates:%, there is $\lambda_2-\lambda_1 = -B_0\nu\kappa^2/\theta^2+O(\theta^{-1})$.
%Thus:
\begin{align*}
  A(t) 
  =& (1+O(\theta^{\min(P_2,P_3)+4}+c\theta^{P_1+4}))e^{-(ic\kappa+\nu\kappa^2+O(\theta^{\min(P_2,P_3+2)}+c\theta^{P_1+2}))t} - \\
   & O(\theta^{\min(P_2,P_3)+4}+c\theta^{P_1+4})e^{\frac{B_0\nu t}{h^2}+O(\theta^{-1})t}\;,
\end{align*}
and providing $B_0<0$, the second term decays exponentially fast to zero as $h\to0$ and:
\begin{equation}\label{eq:acry_semi_amp_a}
  A(t) = e^{-ic\kappa t-\nu\kappa^2t}(1+O(\theta^{\min(P_2,P_3+2)}+c\theta^{P_1+2}))\;.
\end{equation}
A similar conclusion is obtained for $N(t)$.
%\begin{equation}\label{eq:acry_semi_amp_n}
%  N(t) = e^{-ic\kappa t-\nu\kappa^2t}(1+O(\theta^{\min(P_2,P_3+2)}+c\theta^{P_1+2}))
%\end{equation}
%with the help of the estimate:
%\begin{displaymath}
%  \frac{\lambda_1-\nu\kappa\alpha^c}{ic\kappa+\nu\kappa^2\beta^c} = 1 + \frac{\lambda_1-ic\kappa-\nu\kappa^2(\alpha^c+\beta^c)}{ic\kappa+\nu\kappa^2\beta^c} = 1 + O(\theta^{\min(P_2,P_3)+2}+c\theta^{P_1+2})\;.
%\end{displaymath}
%
Summarizing the proof, we have the main theorem regarding the accuracy of the HV-method for advection-diffusion equations.
\begin{theorem}\label{thm:acry}
  Suppose the initial condition is given by a simple wave $w_0(x)=e^{i\kappa x}$.
  Given a $P_1^\powth\textrm{\sc-}$order $[\mathcal{D}_x]$, a $P_2^\powth\textrm{\sc-}$order $[\mathcal{D}_x^c]$ such that $\lim_{\theta\to0}i\theta\beta^c(\theta)=0$, and a $P_3^\powth\textrm{\sc-}$order $[\mathcal{D}_{xx}^c]$ such that $B_0=\lim_{\theta\to0}(i\theta)^2\mathcal{b}(\theta) = \sum_{k=-q'}^{q'}\mathcal{b}_k < 0$, the corresponding spatial order of accuracy of the HV method~\cref{eq:prelim_hv_semi} for the semi-discretization of~\cref{eq:prelim_eqn} with $\nu\ne0$ is $\min(P_1+2,P_2,P_3+2)$.
  Especially if $c=0$, the order of accuracy is $\min(P_2,P_3+2)$.
\end{theorem}
\begin{remark}\label{rm:acry_compact}
  The proof does not restrict to HV-DDOs with continuous stencils, as long as stated conditions on their orders and coefficients are satisfied. 
  Nevertheless, these conditions are automatically satisfied by continuous-stencil operators, see~\cref{sec:prelim_hv}.
\end{remark}
\begin{remark}\label{rm:acry_upw}
  In either~\cref{thm:acry_dx} or~\cref{thm:acry}, a condition ($cb_0>0$ and $B_0<0$, respectively) is required to ensure the exponentially fast decay of the second mode $e^{-\lambda_2t}$.% -- note that in the case of the advection-diffusion equation, the generalized upwind condition $cb_0>0$ for advection term is no longer required as the term $e^{-\frac{B_0}{h^2}}$ is now dominating. %($\lim_{\theta\to0}i\theta\beta(\theta)\ne0$ is equivalent to $b_0\ne0$).
\end{remark}

At the end of this section, we list a few HV-DDOs that are used in the numerical tests in~\cref{sec:num}.
Note that assuming a right-going wave, all $[\mathcal{D}_x]$ in~\cref{tb:acry_dx} has an upwind-biased stencil and leads to a stable semi-discretization of the linear advection equation $w_t+cw_x=0$, see~\cite{XZeng:2019a} and also the discussion in~\cref{sec:stab_ade}.
\begin{table}\centering
  \caption{$[\mathcal{D}_x]$ with up to fourth-order accuracy ($c>0$)~\cite{XZeng:2019a}.}
  \label{tb:acry_dx}
  \begin{tabular}{@{}llcll@{}}
    \toprule[.5mm]
      $\ $ & Symbol & $\quad$ & Formula of $[\mathcal{D}_x]_j$ & $\ $ \\ \cmidrule[.3mm](l){1-4}
      $\ $ & $[\mathcal{D}_x^{(1)}]$ & & $(-2\overline{w}_{\mhf{j}}+2w_j)/h$ \\
      $\ $ & $[\mathcal{D}_x^{(2)}]$ & & $(-6\overline{w}_{\mhf{j}}+2w_{j-1}+4w_j)/h$ \\
      $\ $ & $[\mathcal{D}_x^{(3)}]$ & & $(-7\overline{w}_{\mhf{j}}+\overline{w}_{\phf{j}}+2w_{j-1}+4w_j)/(2h)$ \\
      %$\ $ & $[\mathcal{D}_x^{(3b)}]$ & & $(-\overline{w}_{j-3/2}-17\overline{w}_{\mhf{j}}+8w_{j-1}+10w_j)/(2h)$ \\
      $\ $ & $[\mathcal{D}_x^{(4)}]$ & & $(-\overline{w}_{j-3/2}-31\overline{w}_{\mhf{j}}+2\overline{w}_{\phf{j}}+12w_{j-1}+18w_j)/(6h)$ \\
    \bottomrule[.5mm]
  \end{tabular}
\end{table}
\begin{table}\centering
  \caption{$[\mathcal{D}_x^c]$ with up to sixth-order accuracy.}
  \label{tb:acry_dxc}
  \begin{tabular}{@{}llcll@{}}
    \toprule[.5mm]
      $\ $ & Symbol & $\quad$ & Formula of $[\mathcal{D}_x^c]_j$ & $\ $ \\ \cmidrule[.3mm](l){1-4}
      $\ $ & $[\mathcal{D}_x^{(c-2)}]$ & & $(-\overline{w}_{\mhf{j}}+\overline{w}_{\phf{j}})/h$ \\
      $\ $ & $[\mathcal{D}_x^{(c-4)}]$ & & $(-4\overline{w}_{\mhf{j}}+4\overline{w}_{\phf{j}}+w_{j-1}-w_{j+1})/(2h)$ \\
      $\ $ & $[\mathcal{D}_x^{(c-6)}]$ & & $(-\overline{w}_{j-3/2}-81\overline{w}_{\mhf{j}}+81\overline{w}_{\phf{j}}+\overline{w}_{j+3/2}+24w_{j-1}-24w_{j+1})/(36h)$ \\
      %$\ $ & {\color{orange}$8$}   & 
      %     & {\color{orange}$(-14\overline{w}_{j-3/2}-230\overline{w}_{\mhf{j}}+230\overline{w}_{\phf{j}}+14\overline{w}_{j+3/2}+3w_{j-2}+96w_{j-1}-96w_{j+1}-3w_{j+2})/(108h)$} \\
      %$\ $ & {\color{orange}$10$}   & 
      %     & {\color{orange}$(-...+2350\overline{w}_{\phf{j}}+175\overline{w}_{j+3/2}+\overline{w}_{j+5/2}+...-900w_{j+1}-45w_{j+2})/(900h)$} \\
    \bottomrule[.5mm]
  \end{tabular}
\end{table}
\begin{table}\centering
  \caption{$[\mathcal{D}_{xx}^c]$ with up to fourth-order accuracy.}
  \label{tb:acry_dxx}
  \begin{tabular}{@{}llcll@{}}
    \toprule[.5mm]
      $\ $ & Symbol & $\quad$ & Formula of $[\mathcal{D}_{xx}^c]_j$ & $\ $ \\ \cmidrule[.3mm](l){1-4}
      $\ $ & $[\mathcal{D}_{xx}^{(c-2)}]$ & & $(3\overline{w}_{\mhf{j}}+3\overline{w}_{\phf{j}}-6w_j)/(h^2)$ \\
      $\ $ & $[\mathcal{D}_{xx}^{(c-4)}]$ & & $(15\overline{w}_{\mhf{j}}+15\overline{w}_{\phf{j}}-3w_{j-1}-24w_j-3w_{j+1})/(2h^2)$ \\
      %$\ $ & {\color{orange}$6$}   & 
      %     & {\color{orange}$(\overline{w}_{j-3/2}+209\overline{w}_{\mhf{j}}+209\overline{w}_{\phf{j}}+\overline{w}_{j+3/2}-48w_{j-1}-324w_j-48w_{j+1})/(24h)$} \\
      %$\ $ & {\color{orange}$8$}   & 
      %     & {\color{orange}$(31\overline{w}_{j-3/2}+1439\overline{w}_{\mhf{j}}+1439\overline{w}_{\phf{j}}+31\overline{w}_{j+3/2}-6w_{j-2}-384w_{j-1}-2160w_j-384w_{j+1}-6w_{j+2})/(144h)$} \\
    \bottomrule[.5mm]
  \end{tabular}
\end{table}
The digit in the operator symbol superscrit indicates the order of the DDO, for example, $[\mathcal{D}_x^{(3)}]$ is third order, whereas $[\mathcal{D}_x^{(4)}]$, $[\mathcal{D}_x^{(c-4)}]$, and $[\mathcal{D}_{xx}^{c-4}]$ are all fourth-order operators.
%The labels in the first columns indicate the order of the differential operator; we also use them to denote a particular HV scheme.
%For example, \hv{342} designates the HV method combinging a third-order $[\mathcal{D}_x]$, a fourth-order $[\mathcal{D}_x^c]$, and a second-order $[\mathcal{D}_{xx}^c]$ from~\cref{tb:acry_dx}--\cref{tb:acry_dxx}.
%Note that in the view of~\cref{thm:acry}, we list $[\mathcal{D}_x]$, $[\mathcal{D}_x^c]$, and $[\mathcal{D}_{xx}^c]$ with up to fourth, sixth, and fourth orders of accuracy, respectively. 
%{\color{red} (Review this part after adding the numerical tests)}

\section{Stability Analysis}
\label{sec:stab}
Stability analysis of the general ODE system~\cref{eq:prelim_hv_semi} remains out of reach as the stability barrier of the HV discretization for the advection equation remains an open problem, see recent progress in~\cite{XZeng:2026a}.
In this section, we focus on central HV methods for linear advection-diffusion equations:% and the remainder of this section is divided into three parts: 
\cref{sec:stab_gen} makes the notion of ``stability'' precise and derives inequalities that the coefficients need to satisfy in order for a general HV semi-discretization to be stable, \cref{sec:stab_diff} concentrates on central HV discretization of the linear diffusion equation, and~\cref{sec:stab_ade} generalizes the analysis to central HV methods for advection-diffusion equations.

\subsection{From stability to inequalities}
\label{sec:stab_gen}
%In this section we derive explicit inequalities to satisfy by the HV-DDO coefficients, so that the ODE system~\cref{eq:prelim_hv_semi} is stable.
%Particularly we assume the equation~\cref{eq:prelim_eqn} is discretized in space by~\cref{eq:prelim_hv_semi} with either upwind-biased or central $[\mathcal{D}_x]$, central $[\mathcal{D}_x^c]$, and central $[\mathcal{D}_{xx}^c]$.
%The main purpose of this section is to investigate the stability of the ODE system~\cref{eq:prelim_hv_semi}. %by extending a technique for advection equations proposed in previous work~\cite[Section 4]{XZeng:2019a}.
Let us denote the solution vector as
\begin{equation}\label{eq:stab_sol}
  \bs{w}(t) = [\overline{w}_{1/2},\;\overline{w}_{3/2},\;\cdots,\;\overline{w}_{N-1/2},\;w_0,\;w_1,\;\cdots,\;w_{N-1}]^T\;,
\end{equation}
and the ODE system~\cref{eq:prelim_hv_semi} is written in matrix form:
\begin{equation}\label{eq:stab_ode}
  \bs{w}' = -\frac{c}{h}\bs{D}\bs{w} + \frac{\nu}{h^2}\bs{K}\bs{w}\;.
\end{equation}
The coefficient matrix $\bs{D}$ corresponds to the advection term and it reads:
\begin{equation}\label{eq:stab_mat_dx}
  \bs{D} = \left[\begin{array}{cc}
    \bs{0}    & \bs{S}-\bs{I} \\
    G(\bs{S}) & H(\bs{S})
  \end{array}\right]\;,
\end{equation}
where $\bs{0}$ and $\bs{I}$ are the $N\times N$ zero matrix and identity matrix, respectively.
The cyclic matrix $\bs{S}$ is given by:
\begin{equation}\label{eq:stab_mat_s}
  \bs{S} = \left[\begin{array}{ccccc}
    0 & 1 & \cdots & 0 & 0 \\
    0 & 0 & \cdots & 0 & 0 \\ \vspace*{-.24in} \\
    \vdots & \vdots & \ddots & \vdots & \vdots \\
    0 & 0 & \cdots & 0 & 1 \\
    1 & 0 & \cdots & 0 & 0 
  \end{array}\right]\;,
\end{equation}
and the Laurent polynomials $G(s)$ and $H(s)$ are determined from the operator $[\mathcal{D}_x]$:
\begin{equation}\label{eq:stab_poly_dx}
  G(s) \eqdef \sum_{k=-l}^{r-1}\alpha_ks^k\;,\quad
  H(s) \eqdef \sum_{k=-l'}^{r'}\beta_ks^k\;.
\end{equation}
Similarly, the coefficient matrix $\bs{K}$ corresponds to the diffusion term and it reads:
\begin{equation}\label{eq:stab_mat_dxx}
  \bs{K} = \left[\begin{array}{cc}
    (\bs{S}-\bs{I})G^c(\bs{S}) & (\bs{S}-\bs{I})H^c(\bs{S}) \\
    A(\bs{S}) & B(\bs{S})
  \end{array}\right]
\end{equation}
where $G^c(s)$ and $H^c(s)$ are determined from $[\mathcal{D}_x^c]$:
\begin{equation}\label{eq:stab_poly_dx_c}
  G^c(s) \eqdef \sum_{k=-p}^{p-1}\alpha_k^cs^k\;,\quad
  H^c(s) \eqdef \sum_{k=-p'}^{p'}\beta_k^cs^k\;,
\end{equation}
and $A(s)$ and $B(s)$ are determined from $[\mathcal{D}_{xx}^c]$:
\begin{equation}\label{eq:stab_poly_dxx_c}
  A(s) \eqdef \sum_{k=-q}^{q-1}\mathcal{a}_ks^k\;,\quad
  B(s) \eqdef \sum_{k=-q'}^{q'}\mathcal{b}_ks^k\;.
\end{equation}
Defining the cell P\'{e}clet number $\Pe=ch/\nu$ and assuming $\nu/h^2=1$, \cref{eq:stab_ode} becomes: 
\begin{equation}\label{eq:stab_ode_simp}
  \bs{w}' = \bs{M}\bs{w}\;,\quad \bs{M} = -\Pe\bs{D}+\bs{K}\;.
\end{equation}
Note that since $c\ge0$, we always have $\Pe\ge0$.
The stability of a general square complex matrix $\bs{M}$ is defined as usual:
%For simplicity, we assume $\nu>0$ and define the cell P\'{e}clet number $\Pe=ch/\nu$ and focus on the stability associated with the coefficient matrix $\bs{M} = -\Pe\bs{D}+\bs{K}$:
\begin{definition}\label{def:stab_mat}
  The matrix $\bs{M}$ is stable provided that: if $\lambda$ is an eigenvalue of $\bs{M}$, then it either has negative real part or $\textup{\oname{Re}}\lambda=0$ and it is simple.
\end{definition}
%One can easily use the Jordan normal form of $\bs{M}$ to show that if $\bs{M}$ is stable, the solution to $\bs{w}'=\bs{M}\bs{w}$ is bounded for all $t\ge0$. %and the proof is to use the Jordan normal form of the matrix $\bs{M}$.
%The next lemma concerns the spectrum of $\bs{M}$. 
\begin{lemma}\label{lm:stab_eig}
  Define $\Delta(\Pe,s)$ as:
  \begin{equation}\label{eq:stab_lm_eig_disc}
    \Delta(\Pe,s) = \left[(s-1)G^c(s)-B(s)+\Pe H(s)\right]^2+4(s-1)(A(s)-\Pe G(s))(H^c(s)-\Pe)\;.
  \end{equation}
  Then the spectrum of $\bs{M}=-\Pe\bs{D}+\bs{K}$ is a subset of:
  \begin{equation}\label{eq:stab_lm_eig_spec}
    \mathcal{S}(\Pe) = \left\{\frac{1}{2}\left((s-1)G^c(s)+B(s)-\Pe H(s)\pm\sqrt{\Delta(\Pe,s)}\right)\,:\;\nrm{s}=1\right\}\;.
  \end{equation}
\end{lemma}
\begin{proof}
  We shall use the fact that all four $N\times N$ blocks of $\bs{M}$ are Laurent polynomials in $\bs{S}$ and that $\bs{S}$ has $N$ distinct eigenvalues $s_k=e^{i2k\pi/N}\,,\;k=1,\cdots,N$ and corresponding linearly independent eigenvectors $\bs{v}_1,\cdots\bs{v}_N\in\mathbb{C}^N\backslash\{\bs{0}\}$.
  For any $s_k$, we define a $2\times2$ matrix $\bs{M}_k$ as:
  \begin{equation}\label{eq:stab_eig_smallm}
    \bs{M}_k = -\Pe\left[\begin{array}{cc}
      0 & s_k-1 \\ G(s_k) & H(s_k)
    \end{array}\right] + \left[\begin{array}{cc}
      (s_k-1)G^c(s_k) & (s_k-1)H^c(s_k) \\ A(s_k) & B(s_k)
    \end{array}\right]\;.
  \end{equation}
  The two eigenvalues of $\bs{M}_k$ are:
  \begin{displaymath}
    \lambda_{1,2}(s_k) = \frac{1}{2}\left((s_k-1)G^c(s_k)+B(s_k)-\Pe H(s_k)\pm\sqrt{\Delta(\Pe,s_k)}\right)\in\mathcal{S}(\Pe)\;.
  \end{displaymath}
  In order to prove the first part of the lemma, it remains to show that the eigenvalues of $\bs{M}$ are precisely given by $\lambda_l(s_k),\;l=1,2,\;k=1,\cdots,N$. 
  To this end, we consider the Jordan normal form of $\bs{M}_k$, denoted by $\bs{J}_k$, and suppose $\bs{M}_k = \bs{U}_k\bs{J}_k\bs{U}_k^{-1}$.
  Then we can compute:
  %Denoting $\bs{M}_k=[m_{11}\ m_{12};\,m_{21}\ m_{22}]$ and $\bs{U}_k=[u_{11}\ u_{12};\,u_{21}\ u_{22}]$ for simplicity, and utilizing the fact that for any Laurent polynomial $P(s)$ there is $P(\bs{S})\bs{v}_k=P(s_k)\bs{v}_k$, we have:
  \begin{align*}
    \bs{M}(\bs{U}_k\otimes\bs{v}_k) %&= 
    %\left[\begin{array}{cc}
    %  (\bs{S}-\bs{I})G^c(\bs{S}) & (\bs{S}-\bs{I})(H^c(\bs{S})-\Pe\bs{I}) \\
    %  A(\bs{S})-\Pe G(\bs{S})       & B(\bs{S}) - \Pe H(\bs{S})
    %\end{array}\right]
    %\left[\begin{array}{cc}
    %  u_{11}\bs{v}_k & u_{12}\bs{v}_k \\ u_{21}\bs{v}_k & u_{22}\bs{v}_k
    %\end{array}\right] \\
    %&= \left[\begin{array}{cc}
    %  (m_{11}u_{11}+m_{12}u_{21})\bs{v}_k &
    %  (m_{11}u_{12}+m_{12}u_{22})\bs{v}_k \\
    %  (m_{21}u_{11}+m_{22}u_{21})\bs{v}_k &
    %  (m_{21}u_{12}+m_{22}u_{22})\bs{v}_k 
    %\end{array}\right] \\
    &= (\bs{M}_k\bs{U}_k)\otimes\bs{v}_k = (\bs{U}_k\bs{J}_k)\otimes\bs{v}_k
     = (\bs{U}_k\otimes\bs{v}_k)\bs{J}_k\;.
  \end{align*}
  Because $\bs{U}_k$ is invertible for all $k$, and $\bs{v}_1,\cdots,\bs{v}_N$ are linearly independent, it follows immediately that $\bs{M}$ is similar to $\oname{diag}(\bs{J}_1,\bs{J}_2,\cdots,\bs{J}_N)$, which completes the proof.
\end{proof}

Because the row sums of $\bs{M}$ are all nil, zero is always an eigenvalue.
From the previous lemma, we obtain the following criterion to check the stability of the matrix $\bs{M}$.%, i.e., all remaining eigenvalues have negative real parts.
\begin{lemma}\label{lm:stab_crit}
  The coefficient matrix $\bs{M}$ is stable provided that $\mathcal{S}'(\Pe)$ is contained in the open left half complex plane, where $\mathcal{S}'(\Pe)\subset\mathcal{S}(\Pe)$ is given by:
  \begin{equation}\label{eq:stab_lm_eig_specprime}
    \mathcal{S}'(\Pe) = \left\{\frac{1}{2}\left((s-1)G^c(s)+B(s)-\Pe H(s)\pm\sqrt{\Delta(\Pe,s)}\right)\,:\;\nrm{s}=1,\;s\ne1\right\}\;.
  \end{equation}
\end{lemma}
\begin{proof}
  %Due to our definition of stability, 
  We only need to investigate $\mathcal{S}(\Pe)\backslash\mathcal{S}'(\Pe)$ and thus set $s=1$.
  In this case:
  \begin{displaymath}
    \Delta(\Pe,1) = (-B(1)+\Pe H(1))^2 = \left(-\sum_{k=-q'}^{q'}\mathcal{b}_k+\Pe \sum_{k=-l'}^{r'}\beta_k\right)^2
                = \left(-B_0+\Pe b_0\right)^2\;.
  \end{displaymath}
  By~\cref{thm:prelim_ddo_dxx} there is $B_0<0$.
  Because $[\mathcal{D}_x]$ is either upwind-biased or central, which gives either $\Pe b_0 = (cb_0)(h/\nu)>0$ (as $cb_0>0$) or $b_0=0$.
  Hence there is always $\Delta(\Pe,1)>0$ and the corresponding elements in $\mathcal{S}(\Pe)$ are:
  \begin{displaymath}
    \frac{1}{2}\left(B(1)-\Pe H(1)\pm\sqrt{\Delta(\Pe,1)}\right) = 0,\ -\Pe b_0+B_0\;,
  \end{displaymath}
  i.e., a simple zero and a negative real number.
\end{proof}

\subsection{Stability proof for diffusion equations}
\label{sec:stab_diff}
The stability proof in the case when $\Pe=0$ or $\bs{M}=\bs{K}$ builds on analysis using positive trigonometric polynomials.
For this purpose we state two useful results below.
The first is a classical one by Vietoris~\cite{LVietoris:1959a}, see also~\cite{RAskey:1974a}.
\begin{theorem}\label{thm:stab_viet}
  If $c_0\ge \cdots \ge c_n > 0$ and $(2k)c_{2k}\le(2k-1)c_{2k-1}$ for all $k\ge1$, then:
  \begin{equation}\label{eq:stab_viet_trig}
    \sum_{k=1}^nc_k\sin k\theta > 0\;,\quad\forall\ 0<\theta<\pi\;;\quad
    \textrm{ and }\quad \sum_{k=0}^nc_k\cos k\theta > 0\;,\quad\forall\ 0<\theta<\pi\;.
  \end{equation}
\end{theorem}
For convenience, we shall call $\{c_0,\;c_1,\;\cdots,\;c_n\}$ a {\it special Vietoris sequence} if:
\begin{equation}\label{eq:stab_viet}
c_0\ge\cdots\ge c_n>0\quad\textrm{ and }\quad kc_k\le(k-1)c_{k-1}\ \textrm{ for all }k\ge2\;.
\end{equation}
It is clear that a special Vietoris sequence satisfies all conditions required in the theorem.

\smallskip

The second result is elementary:
\begin{theorem}\label{thm:stab_dec}
  If $c_1\ge c_2\ge\cdots\ge c_n\ge0$, then for all $0<\phi<\pi$:
  \begin{equation}\label{eq:stab_dec}
    \sum_{k=1}^nc_k\sin(2k-1)\phi\ge0\;.
  \end{equation}
  %and if we set $c_{n+1}=0$, then for all $0<\theta<\pi$:
  %\begin{equation}\label{eq:stab_dec2}
  %  \sum_{k=1}^n(c_k+c_{k+1})\sin k\theta \ge0\;.
  %\end{equation}
\end{theorem}
\begin{proof}
  %We first show~\cref{eq:stab_dec} and l
  Let $\theta=2\phi$, then $2\sin\phi\sin(2k-1)\phi = \cos(k-1)\theta - \cos\,k\theta$ and it follows:
  \begin{align*}
     &\ \sin\phi\sum_{k=1}^nc_k\sin(2k-1)\phi = \frac{1}{2}\sum_{k=1}^nc_k\left[\cos(k-1)\theta-\cos\,k\theta\right] \\
    =&\ \sum_{k=1}^nc_k\left[\sin^2k\phi-\sin^2(k-1)\phi\right] = \sum_{k=1}^n(c_k-c_{k+1})\sin^2k\phi \ge0\;,
  \end{align*}
  where $c_{n+1}=0$.
  Finally, \cref{eq:stab_dec} comes from the fact that $\sin\phi>0$ for $0<\phi<\pi$.
\end{proof}
  %Next we show~\cref{eq:stab_dec2} and let $\phi=\theta/2$; then using~\cref{eq:stab_dec} and $\cos\phi>0$ we obtain:
  %\begin{displaymath}
  %  2\cos\phi\sum_{k=1}^nc_k\sin(2k-1)\phi \ge 0\;,
  %\end{displaymath}
  %which is exactly~\cref{eq:stab_dec2}.

%Thirdly, a Fourier sine series is non-negative if the coefficients are positive, non-decreasing, and convex~\cite{INPak:1980a} or~\cite[3.5.18]{DSMitrinovic:1970a}:
%\begin{theorem}\label{thm:stab_cvx}
%  Let $\{c_k\}_{k=1}^n$ be a positive, non-decreasing, and convex series, i.e., $c_1\ge c_2\ge\cdots\ge c_n>0$ and $c_k-2c_{k+1}+c_{k+2}\ge0$ for $1\le k\le n$, where if a subscript is larger than $n$ we designate the $c$-coefficient by zero.
%  Then for all $0\le\theta\le\pi$:
%  \begin{equation}\label{eq:stab_cvx}
%    \sum_{k=1}^{n-1}c_k\sin k\theta + \frac{1}{2}c_n\sin n\theta \ge 0\;.
%  \end{equation}
%\end{theorem}

\noindent
The stability theorem regarding the HV discretizations of diffusion equations is given below:
\begin{theorem}\label{thm:stab_asym}
  When discretizing the diffusion equation ($\Pe=0$), the semi-discretized HV method constructed by $[\mathcal{D}_x^c]$ and $[\mathcal{D}_{xx}^c]$ given in~\cref{sec:prelim} is stable for all $h$.
\end{theorem}
In the proof below, we omit the verification of special Vietoris sequences since it is quite lengthy and tedious; instead, the details are offered in~\cref{app:seq}.
\begin{proof}
  Following~\cref{lm:stab_crit}, it suffices to show $\mathcal{S}'(0)$ is contained in the open left half of the complex plane -- in fact, we'll show it contains only negative real numbers.

  \smallskip
  
  {\bf Part 1}. 
  First we show $\Delta(0,s)=[(s-1)G^c(s)-B(s)]^2+4(s-1)A(s)H^c(s)\ge0$ where the inequality is strict if $s\ne-1$.
  Using~\cref{eq:prelim_ddo_dx_c_alpha_pos} and~\cref{eq:prelim_ddo_dxx_c_b}:
  %\begin{align*}
  %  (s-1)G^c(s) &= (s^{\frac{1}{2}}-s^{-\frac{1}{2}})\sum_{k=-p}^{p-1}\alpha_k^cs^{k+\frac{1}{2}}
  %               = (s^{\frac{1}{2}}-s^{-\frac{1}{2}})\sum_{k=0}^{p-1}\alpha_k^c(s^{k+\frac{1}{2}}-s^{-k-\frac{1}{2}}) \\
  %              &= -4\sin\phi\sum_{k=0}^{p-1}\alpha_k^c\sin(2k+1)\phi \;, \\
  %  B(s)        &= \sum_{k=-q'}^{q'}\mathcal{b}_ks^k = \mathcal{b}_0+\sum_{k=1}^{q'}\mathcal{b}_k(s^k+s^{-k})
  %               = \mathcal{b}_0+2\sum_{k=1}^{q'}\mathcal{b}_k\cos k\theta\;,
  %\end{align*}
  \begin{displaymath}
    (s-1)G^c(s) %&= (s^{\frac{1}{2}}-s^{-\frac{1}{2}})\sum_{k=-p}^{p-1}\alpha_k^cs^{k+\frac{1}{2}}
                % = (s^{\frac{1}{2}}-s^{-\frac{1}{2}})\sum_{k=0}^{p-1}\alpha_k^c(s^{k+\frac{1}{2}}-s^{-k-\frac{1}{2}}) \\
                = -4\sin\phi\sum_{k=0}^{p-1}\alpha_k^c\sin(2k+1)\phi \;,\quad
    B(s)        %&= \sum_{k=-q'}^{q'}\mathcal{b}_ks^k = \mathcal{b}_0+\sum_{k=1}^{q'}\mathcal{b}_k(s^k+s^{-k})
                 = \mathcal{b}_0+2\sum_{k=1}^{q'}\mathcal{b}_k\cos k\theta\;,
  \end{displaymath}
  where $s=e^{i\theta}$ and $\phi = \theta/2$; hence $(s-1)G^c(s)-B(s)\in\mathbb{R}$.
  Next we verify $(s-1)A(s)H^c(s)\ge0$ with the inequality being strict when $s\ne-1$.
  Using~\cref{eq:prelim_ddo_dxx_c_a_neg} and~\cref{eq:prelim_ddo_dx_c_beta}:
  %\begin{align*}
  %  (s-1)A(s) &= (s^{\frac{1}{2}}-s^{-\frac{1}{2}})\sum_{k=-q}^{q-1}\mathcal{a}_ks^{k+\frac{1}{2}} 
  %             = (s^{\frac{1}{2}}-s^{-\frac{1}{2}})\sum_{k=0}^{q-1}\mathcal{a}_k(s^{k+\frac{1}{2}}+s^{-k-\frac{1}{2}}) \\
  %            &= 4i\sin\phi\sum_{k=0}^{q-1}\mathcal{a}_k\cos(2k+1)\phi
  %             = 2i\sum_{k=0}^{q-1}\mathcal{a}_k(\sin(k+1)\theta-\sin k\theta)\;, \\
  %  H^c(s)    &= \sum_{k=-p'}^{p'}\beta_k^cs^k = \sum_{k=1}^{p'}\beta_k^c(s^k-s^{-k})
  %             = 2i\sum_{k=1}^{p'}\beta_k^c\sin k\theta\;.
  %\end{align*}
  \begin{displaymath}
    (s-1)A(s) %&= (s^{\frac{1}{2}}-s^{-\frac{1}{2}})\sum_{k=-q}^{q-1}\mathcal{a}_ks^{k+\frac{1}{2}} 
              % = (s^{\frac{1}{2}}-s^{-\frac{1}{2}})\sum_{k=0}^{q-1}\mathcal{a}_k(s^{k+\frac{1}{2}}+s^{-k-\frac{1}{2}}) \\
              %&= 4i\sin\phi\sum_{k=0}^{q-1}\mathcal{a}_k\cos(2k+1)\phi
               = 2i\sum_{k=0}^{q-1}\mathcal{a}_k(\sin(k+1)\theta-\sin k\theta)\;,\quad %\\
    H^c(s)    %&= \sum_{k=-p'}^{p'}\beta_k^cs^k = \sum_{k=1}^{p'}\beta_k^c(s^k-s^{-k})
               = 2i\sum_{k=1}^{p'}\beta_k^c\sin k\theta\;.
  \end{displaymath}
  Thus their product is real. 
  It is easy to see that $(s-1)A(s)H^c(s)=0$ if $s=-1$ or $\theta=\pi$; next we suppose $0<\theta<2\pi$ and $\theta\ne\pi$ and prove the following equivalent statement:
  \begin{equation}\label{eq:stab_asym_pos_equiv}
    \left(\sum_{k=1}^q(\mathcal{a}_{k-1}-\mathcal{a}_k)\sin\,k\theta\right)\left(\sum_{k=1}^{p'}(-\beta_k^c)\sin\,k\theta\right) > 0\;,\quad 0<\theta<\pi\ \textrm{ or }\ \pi<\theta<2\pi\;,
  \end{equation}
  where $\mathcal{a}_q=0$.
  We focus on the case $0<\theta<\pi$, since $\theta\to2\pi-\theta$ proves the other half.
 % since the other half can be obtained by $\theta\rightarrow2\pi-\theta$. 
  
  It is verified that both $\{\mathcal{a}_0-\mathcal{a}_1,\;\mathcal{a}_1-\mathcal{a}_2,\;\cdots,\;\mathcal{a}_{q-1}-\mathcal{a}_q\}$ and $\{-\beta_1^c,\;-\beta_2^c,\;\cdots,\;-\beta_{p'}^c\}$ are special Vietoris sequences in~\cref{app:seq}, then~\cref{eq:stab_asym_pos_equiv} follows from~\cref{thm:stab_viet}.

  \smallskip

  {\bf Part 2}. 
  Next we show the non-square-root part of $\mathcal{S}'(0)$ is negative:
  \begin{equation}\label{eq:stab_asym_comm_neg}
    (s-1)G^c(s)+B(s) < 0\;,\quad\forall\ s=e^{i\theta},\;0<\theta<2\pi\;.
  \end{equation}
  %Previously, it is already shown that:
  %\begin{equation}\label{eq:stab_asym_comm_neg_gc}
  %  (s-1)G^c(s) = -4\sin\phi\sum_{k=0}^{p-1}\alpha_k^c\sin(2k+1)\phi \in \mathbb{R}
  %\end{equation}
  %and
  %\begin{equation}\label{eq:stab_asym_comm_neg_b}
  %  B(s) = -\left((-\mathcal{b}_0)+\sum_{k=1}^{q'}(-2\mathcal{b}_k)\cos k\theta\right) \in \mathbb{R}\;;
  %\end{equation}
  Particularly we will prove $(s-1)G^c(s)\le0$ and $B(s)<0$.
  Between the two, the latter is a consequence of~\cref{thm:stab_viet} and the fact that $\{-\mathcal{b}_0,\;-2\mathcal{b}_1,\;\cdots,\;-2\mathcal{b}_{q'}\}$ is a special Vietoris sequence (\cref{app:seq}).
  In fact, if $s\ne-1$ the strict inequality comes from the theorem and a change of variable $\theta\to2\pi-\theta$ if needed.
  When $s=-1$, we'll see in the proof in the appendix that all inequalities for this special Vietoris sequence are strict; hence:
  \begin{displaymath}
    B(-1) = -\left[\left((-\mathcal{b}_0)-(-2\mathcal{b}_1)\right)+\left((-2\mathcal{b}_2)-(-2\mathcal{b}_3)\right)+\cdots\right] < 0\;.
  \end{displaymath}
  The first inequality $(s-1)G^c(s)\le0$ comes from~\cref{thm:stab_dec}, and we just need to verify $\alpha_0^c\ge\alpha_1^c\ge\cdots\ge\alpha_{p-1}^c\ge0$, which is obvious in the view of~\cref{eq:prelim_ddo_dx_c_alpha_pos}.
  
  \smallskip

  {\bf Part 3}. 
  Lastly, we show $\mathcal{S}'(0)$ contains only negative real numbers by proving:
  \begin{displaymath}
    \Delta(0,s) < \left[(s-1)G^c(s)+B(s)\right]^2\ \Longleftrightarrow\
    (s-1)A(s)H^c(s) < (s-1)G^c(s)B(s)
  \end{displaymath}
  or equivalently
  \begin{equation}\label{eq:stab_asym_eigval_equiv}
    \sum_{k=1}^q(\mathcal{a}_{k\!-\!1}\!-\!\mathcal{a}_k)\sin k\theta\,
    \sum_{k=1}^{p'}(-\beta_k^c)\sin k\theta < 
    \sin\phi\sum_{k=0}^{p-1}\alpha_k^c\sin(2k\!+\!1)\phi\,
    \sum_{k=-q'}^{q'}(-\mathcal{b}_k)\cos k\theta\!.
  \end{equation}
  By symmetry, we only need to consider $0<\theta<\pi$ or $0<\phi<\pi/2$.
  To proceed, we will split the $\mathcal{a}$, $\mathcal{b}$ series and the $\alpha^c$, $\beta^c$ series and prove separately:
  \begin{align}
    \label{eq:stab_asym_eigval_dxx}
    \cos\phi\sum_{k=1}^q(\mathcal{a}_{k-1}-\mathcal{a}_k)\sin k\theta &< \sin\phi\left(-\mathcal{b}_0+\sum_{k=1}^{q'}(-2\mathcal{b}_k)\cos k\theta\right)\;, \\
    \label{eq:stab_asym_eigval_dxc}
    \left(\sum_{k=1}^{p'}(-\beta_k^c)\sin k\theta\right) &< \cos\phi\sum_{k=0}^{p-1}\alpha_k^c\sin(2k+1)\phi\;,
  \end{align}
  It is easy to rearrange~\cref{eq:stab_asym_eigval_dxx} as:
  %\begin{align*}
  %  & \sum_{k=1}^q(\mathcal{a}_{k-1}-\mathcal{a}_k)\left[\sin(2k\!+\!1)\phi+\sin(2k\!-\!1)\phi\right] \\
  %  <& (-2\mathcal{b}_0)\sin\phi+\sum_{k=1}^{q'}(-2\mathcal{b}_k)\left[\sin(2k\!+\!1)\phi-\sin(2k\!-\!1)\phi\right]
  %\end{align*}
  %or equivalently:
  \begin{equation}\label{eq:stab_asym_eigval_dxx_equiv}
    (-\mathcal{a}_0+\mathcal{a}_1-2\mathcal{b}_0+2\mathcal{b}_1)\sin\phi + \sum_{k=1}^q\left(-\mathcal{a}_{k-1}+\mathcal{a}_{k+1}-2\mathcal{b}_k+2\mathcal{b}_{k+1}\right)\sin(2k+1)\phi > 0\;,
  \end{equation}
  where $\mathcal{a}_q=\mathcal{a}_{q+1}=\mathcal{b}_{q+1}=0$ and if $q'=q-1$ there is also $\mathcal{b}_q=0$.
  Denoting the left hand side by $\sum_{k=0}^qc_k\sin(2k+1)\phi>0$ for short, we shall show in~\cref{app:seq} that:
  \begin{equation}\label{eq:stab_asym_eigval_dxx_seq}
    \sum_{k=0}^q(2k+1)c_k = 0\;;\quad c_0>0\;;\quad c_k<0\;,\quad 1\le k\le q\;;
  \end{equation}
  then one has for all $0<\phi<\pi/2$:
  \begin{displaymath}
    \sum_{k=0}^qc_k\sin(2k+1)\phi \ge c_0\sin\phi - \sum_{k=1}^q\abs{c_k}\abs{\sin(2k+1)\phi} 
    > c_0\sin\phi - \sum_{k=1}^q\abs{c_k}(2k+1)\sin\phi = 0\;.
  \end{displaymath}
  %where we used the fact that $\abs{\sin m\phi} < m\sin\phi$ for all $m\ge2$ and $0<\phi<\pi$.
  For~\cref{eq:stab_asym_eigval_dxc}, we rearrange it as:
  %\begin{displaymath}
  %  \sum_{k=0}^{p-1}\alpha_k^c\left[\sin(k\!+\!1)\theta+\sin k\theta\right]+\sum_{k=1}^{p'}2\beta_k^c\sin k\theta > 0
  %\end{displaymath}
  %or equivalently:
  \begin{equation}\label{eq:stab_asym_eigval_dxc_equiv}
    \sum_{k=1}^p\left(\alpha_{k-1}^c+\alpha_k^c+2\beta_k^c\right)\sin k\theta > 0\;,
  \end{equation}
  where as before it is understood that $\alpha_p^c=0$ and if $p'=p-1$ there is also $\beta_p^c=0$.
  Rewriting the left hand side of~\cref{eq:stab_asym_eigval_dxc_equiv} as $\sum_{k=1}^pd_k\sin k\theta$ for short, we'll prove in~\cref{app:seq} that:
  \begin{equation}\label{eq:stab_asym_eigval_dxc_seq}
    \sum_{k=1}^pkd_k = 1\;;\quad
    d_1>d_2>\cdots>d_p>0\;;\quad
    d_1 > \frac{1}{2}\;.
  \end{equation}
  Then one can show that:
  \begin{displaymath}
    \sum_{k=1}^pk\sin k\theta \ge d_1\sin\theta - \sum_{k=2}^pd_k\abs{\sin k\theta} > 
    d_1\sin\theta - \sum_{k=2}^pkd_k\sin\theta = (2d_1-1)\sin\theta > 0
  \end{displaymath}
  for all $0<\theta<\pi$.
  The proof is completed.
\end{proof}

\subsection{Stability of central HV methods for advection-diffusion equations}
\label{sec:stab_ade}
%Recognizing that the stability theory for general HV discretization of advection equations is not completed, 
We consider in this section central HV methods for advection-diffusion equations -- that is, the discrete differential operator $[\mathcal{D}_x]$ in~\cref{eq:prelim_hv_semi_cell} uses a central stencil:
\begin{equation}
  [\mathcal{D}_xw]_j = \frac{1}{h}\sum_{k=-r}^{r-1}\alpha_k\overline{w}_{\phf{j+k}}+\frac{1}{h}\sum_{k=-r'}^{r'}\beta_kw_{j+k}\;,
\end{equation}
where we abused the notation and keep the symbols $\alpha_k$ and $\beta_k$, which are defined by~\cref{cor:prelim_ddo_dx_c}, with $p$ and $p'$ replaced by $r$ and $r'$, respectively.

To this end, the Laurent polynomials~\cref{eq:stab_poly_dx} are modified to $G(s) = \sum_{k=-r}^{r-1}\alpha_ks^k$ and $H(s) = \sum_{k=-r'}^{r'}\beta_ks^k$, and stability of the semi-discretized method reduces to show $\mathcal{S}'(\oname{Pe})$ is contained in the left complex plane, where $\mathcal{S}'(\oname{Pe})$ is given by~\cref{eq:stab_lm_eig_specprime}.
Suppressing the dependence in $s$, defining $F^c=(s-1)G^c$, $F=(s-1)G$ and $C=(s-1)A$, and following the analysis in~\cref{sec:stab_diff}, we have for all $s=e^{i\theta}$, $0<\theta<2\pi$:
\begin{equation}\label{eq:stab_ade_ineq_1}
  F^c, F < 0\;,\quad
  H^c, H \in i\mathbb{R}^-\;,\quad
  C \in i\mathbb{R}^+\;,\quad
  B < 0\;,
\end{equation}
and in addition:
\begin{equation}\label{eq:stab_ade_ineq_2}
  CH^c < F^cB\;,\quad
  CH < FB\;.
\end{equation}
Writing $F^c(s)=-{f}^c(\theta)$, $F(s)=-{f}(\theta)$, $H^c(s)=-i{h}^c(\theta)$, $H(s)=-i{h}(\theta)$, $C(s)=i{c}(\theta)$, and $B(s)=-{b}(\theta)$. 
Assume then ${f}^c, {f}, {h}^c, {h}, {c}, {b}$ are all positive functions, and~\cref{eq:stab_ade_ineq_2} becomes:
\begin{equation}\label{eq:stab_ade_ineq_scr}
  \delta^c = f^cb-ch^c > 0\;,\quad
  \delta = fb-ch > 0\;.
\end{equation}
To this end, writing $x=\oname{Pe}\ge0$ for short:
\begin{align*}
  \Delta(\oname{Pe},s) = \Delta(x,s) %= (-f^c+b-ihx)^2 + 4(ic+fx)(-ih^c-x) \\
  = -(h^2+4f)x^2 - 2i(2fh^c+2c+hb-hf^c)x + (b-f^c)^2+4ch^c\;.
\end{align*}
The condition $\mathcal{S}'(x) = \mathcal{S}'(\oname{Pe}) \subset\mathbb{C}^-$ thus reduces to:
%\begin{align}
%  \notag
%  & \pm\oname{Re}\sqrt{\Delta(x,s)} < f^c+b \\
%  \notag
%  \Leftrightarrow& \left(\sqrt{c_0+ic_1x+c_2x^2}+\sqrt{c_0-ic_1x+c_2x^2}\right)^2 < 4(f^c+b)^2 \\
%  \label{eq:stab_ade_equiv_1}
%  \Leftrightarrow& \sqrt{(c_0+c_2x^2)^2+c_1^2x^2} < 2(f^c+b)^2-c_0-c_2x^2\;.
%\end{align}
\begin{equation}\label{eq:stab_ade_equiv_1}
  \pm\oname{Re}\sqrt{\Delta(x,s)} < f^c+b \Leftrightarrow \sqrt{(c_0+c_2x^2)^2+c_1^2x^2} < 2(f^c+b)^2-c_0-c_2x^2\;.
\end{equation}
Here $c_0=(b-f^c)^2+4ch^c$, $c_1=-2(2fh^c+2c+hb-hf^c)$, and $c_2=-(h^2+4f)$.
The right hand side of~\cref{eq:stab_ade_equiv_1} is:
\begin{align*}
  2(f^c+b)^2-c_0-c_2x^2 %&= 2(f^c+b)^2-(b-f^c)^2-4ch^c+(h^2+4f)x^2 \\
  &= (f^c+b)^2+4\delta^c + (h^2+4f)x^2 > 0\;,
\end{align*}
therefore,~\cref{eq:stab_ade_equiv_1} is equivalent to:
\begin{align}
  \notag
  &\quad(c_0+c_2x^2)^2 + c_1^2x^2 < [2(f^c+b)^2-c_0]^2-2c_2[2(f^c+b)^2-c_0]x^2+c_2^2x^4 \\
  %\notag
  %\Leftrightarrow&\quad0 < 4(f^c+b)^2[(f^c+b)^2-c_0] - \left\{2c_2[2(f^c+b)^2-c_0]+c_1^2+2c_0c_2\right\}x^2 \\
  \label{eq:stab_ade_equiv_2}
  \Leftrightarrow&\quad0 < 16(f^c+b)^2\delta^c + 16\left[f(f^c+b)^2+(fh^c+c+hb)(hf^c-c-fh^c)\right]x^2\;.
\end{align}
%The $x^2$-coefficient of~\cref{eq:stab_ade_equiv_2} is actually positive for all $0<\theta<2\pi$; and for this purpose two lemmas are needed:
We will prove that the $X^2$-coefficient is positive; for this purpose two lemmas are needed:
\begin{lemma}\label{lm:stab_ade_dx}
  Let $f(\theta)=-(s-1)G(s)$ and $h(\theta)=iH(s)$ be decided by the central operator $[\mathcal{D}_x]$~\cref{eq:prelim_hv_dx} with $r=l$ and $r'=l'$, then for all $0<\theta\le\pi$:
  \begin{equation}\label{eq:stab_ade_dx}
    \theta\,h < f < \theta(\theta+h)\;,
  \end{equation}
  and
  \begin{equation}\label{eq:stab_ade_dx_sin}
    f-2\sin\frac{\theta}{2}h \ge 4\sin^2\frac{\theta}{2}\;.
  \end{equation}
  Furthermore, if $r+r'\ge2$, there is also:
  \begin{equation}\label{eq:stab_ade_dx_omega}
    f > \omega(\omega+h)\;,
  \end{equation}
  %where $\omega = \theta\left(1-\frac{\theta^2}{18\pi^2}\right)$.
  where $\omega = \theta\left(1-\frac{\theta^2}{50}\right)$.
\end{lemma}
Note that $[\mathcal{D}_x]$ when $r=l$ and $r'=l'$ is precisely $[\mathcal{D}_x^c]$~\cref{eq:prelim_hv_dx_c} with $p=l$ and $p'=l'$; we use $f$ and $h$ in this lemma as~\cref{eq:stab_ade_dx} is what is needed in the proof of~\cref{thm:stab_ade}.
%The requirement that $r+r'\ge2$ is mainly to include the lower bound involving $\omega$ and make the dicussion easier, it only excludes one situation when $r=l=1$ and $r'=l'=0$, in which case $h\equiv0$ and $f=4\sin^2(\theta/2)$ -- one still has $\theta\,h<f<\theta(\theta+h)$.
This lemma is proved in~\cref{app:ade}. %, which however uses the notations of $[\mathcal{D}_x^c]$ as a continuation of the notations in the previous appendix.
Then we need two inequalities related to $A(s)$ and $B(s)$:
\begin{lemma}\label{lm:stab_ade_dxx}
  Let $c(\theta)=-i(s-1)A(s)$ and $b(\theta)=-B(s)$ be determined by the central operator $[\mathcal{D}^c_{xx}]$~\cref{eq:prelim_hv_dxx_c}, then for all $0<\theta\le\pi$:
  \begin{equation}\label{eq:stab_ade_dxx}
    c < \omega b < \theta b\;,
  \end{equation}
  where $\omega$ is defined the same as before, and
  \begin{equation}\label{eq:stab_ade_dxx_sin}
    \theta b - c \ge 6(\theta-\sin\theta)\;,\quad
    b \le 4\pi^2-6\pi\theta+3\theta^2\;.
  \end{equation}
\end{lemma}
Note that the choice of $\omega$ is delicate -- it separates the quantities $c/b$ and $\sqrt{f+h^2/4}-h/2$ for all almost all possible choices of $[\mathcal{D}_x]$ and $[\mathcal{D}^c_{xx}]$, except when $[\mathcal{D}_x]$ is given by the stencil $r=l=1$ and $r'=l'=0$, which explains the requirement $r+r'\ge2$ in~\cref{eq:stab_ade_dx_omega}.
For this excluded $[\mathcal{D}_x]$, there is $h\equiv0$ and $f=4\sin^2\frac{\theta}{2}$, and the proof of~\cref{thm:stab_ade} in this case rests on~\cref{eq:stab_ade_dx_sin} and~\cref{eq:stab_ade_dxx_sin}.
The main result of this section is:
\begin{theorem}\label{thm:stab_ade}
  Let the semi-discretization~\cref{eq:prelim_hv_semi} be given by a central $[\mathcal{D}_x]$, that is~\cref{eq:prelim_hv_dx} with $l=r$ and $l'=r'$, and $[\mathcal{D}_x^c]$ and $[\mathcal{D}_{xx}^c]$ as before. 
  Then the resulting ODE system is stable for all $h>0$ and $\Pe\ge0$.
\end{theorem}
\begin{proof}
  By previous derivation, we just need to show:
  \begin{equation}\label{eq:stab_ade_coef2}
    C_2 \eqdef f(f^c+b)^2+(fh^c+c+hb)(hf^c-c-fh^c) > 0
  \end{equation}
  for all $0<\theta<2\pi$; because the left hand side is even, we just need to focus on $0<\theta\le\pi$.

  %It is easy to see $f, f^c, b$ are cosine series and therefore even, and $h, h^c, c$ are sine series and odd, hence the left hand side of~\cref{eq:stab_ade_coef2} is even and by symmetry we only need to consider $0<\theta\le\pi$. 

  %We re-arrange the terms of~\cref{eq:stab_ade_coef2} next. which is somehow inspired by the original coefficients of $\Delta(x,s)$.
  Letting $t=\sqrt{f+\frac{1}{4}h^2}$ (this is $-c_2/4$ defined earlier), one has:
  %\begin{align}
  %  \notag
  %   &\ f(f^c+b)^2+(fh^c+c+hb)(hf^c-c-fh^c) \\
  %  \notag
  %  =&\ (f+\frac{1}{4}h^2)(f^c+b)^2-(fh^c+c)^2+(hf^c-hb)(fh^c+c)-\frac{1}{4}h^2(f^c-b)^2 \\
  %  \notag
  %  =&\ t^2(f^c+b)^2-\left[fh^c+c-\frac{1}{2}h(f^c-b)\right]^2 \\
  %  \label{eq:stab_ade_coef2_equiv}
  %  =&\ t^2(f^c+b)^2-\left[f^c\left(\frac{fh^c}{f^c}-\frac{1}{2}h\right)+b\left(\frac{c}{b}+\frac{1}{2}h\right)\right]^2\;.
  %\end{align}
  \begin{equation}\label{eq:stab_ade_coef2_equiv}
    C_2 = t^2(f^c+b)^2-\left[f^c\left(\frac{fh^c}{f^c}-\frac{1}{2}h\right)+b\left(\frac{c}{b}+\frac{1}{2}h\right)\right]^2\;.
  \end{equation}
  To estimate the last term, we first assume $r+r'\ge2$ and define $\rho=t-\frac{1}{2}h>0$, then by~\cref{lm:stab_ade_dx}:
  \begin{displaymath}
    \rho(\rho+h) = f > \omega(\omega+h)\quad\Longrightarrow\quad\rho> \omega
  \end{displaymath}
  where we used~\cref{lm:stab_ade_dx}.
  By~\cref{eq:stab_ade_dxx}, we have:
  \begin{displaymath}
    t-\frac{1}{2}h = \rho > \omega > \frac{c}{b}\quad\Longrightarrow\quad\frac{c}{b}+\frac{1}{2}h < t\;.
  \end{displaymath}
  For the other term in the square bracket of~\cref{eq:stab_ade_coef2_equiv}, we note again by~\cref{lm:stab_ade_dx}:
  \begin{displaymath}
    \theta(\theta+h) > f = \rho(\rho+h) \quad\Longrightarrow\quad \theta > \rho\;,
  \end{displaymath}
  thus using~\cref{lm:stab_ade_dxx} one has:
  \begin{displaymath}
    \frac{fh^c}{f^c} < \frac{f}{\theta} < \frac{f}{\rho} = \frac{f}{\sqrt{f+\frac{1}{4}h^2}-\frac{1}{2}h} = \sqrt{f+\frac{1}{4}h^2}+\frac{1}{2}h = t + \frac{1}{2}h\;.
  \end{displaymath}
  Therefore:
  \begin{displaymath}
    -t < -\frac{1}{2}h < \frac{fh^c}{f^c}-\frac{1}{2}h < t\quad\Longrightarrow\quad \abs{\frac{fh^c}{f^c}-\frac{1}{2}h} < t\;.
  \end{displaymath}
  To this end, we can finally continue from~\cref{eq:stab_ade_coef2_equiv} and derive:
  %\begin{align*}
  %   &\ f(f^c+b)^2+(fh^c+c+hb)(hf^c-c-fh^c) \\
  %  \ge&\ t^2(f^c+b)^2 - \left[f^c\abs{\frac{fh^c}{f^c}-\frac{1}{2}h}+b\left(\frac{c}{b}+\frac{1}{2}h\right)\right]^2 \\
  %  >& t^2(f^c+b)^2-(f^ct+bt)^2 = 0\,
  %\end{align*}
  \begin{displaymath}
    C_2 \ge t^2(f^c+b)^2 - \left[f^c\abs{\frac{fh^c}{f^c}-\frac{1}{2}h}+b\left(\frac{c}{b}+\frac{1}{2}h\right)\right]^2 > t^2(f^c+b)^2-(f^ct+bt)^2 = 0\,.
  \end{displaymath}

  Lastly we consider $[\mathcal{D}_x]$ with $r=l=1$ and $r'=l'=0$, in which case $h\equiv0$ and $f=4\sin^2\frac{\theta}{2}$. 
  by the discussion below the lemma, we see $h\equiv0$ and $0<f<\theta^2$.
  Furthermore, \cref{eq:stab_ade_dx} applies to $f_c$ and $h_c$, therefore:
  \begin{displaymath}
    C_2 = f(f^c+b)^2-(fh^c+c)^2 = \left[\sqrt{f}(f^c+b)+fh^c+c\right]\left[\sqrt{f}(f^c-\sqrt{f}h^c)+(\sqrt{f}b-c)\right]\;,
  \end{displaymath}
  and an estimate of the last term can be obtained from~\cref{eq:stab_ade_dx_sin} and~\cref{eq:stab_ade_dxx_sin}:
  \begin{align*}
    &\ \sqrt{f}(f^c-\sqrt{f}h^c)+(\sqrt{f}b-c) = 2\sin\frac{\theta}{2}\left(f^c-2\sin\frac{\theta}{2}\,h^c\right)+\left(2\sin\frac{\theta}{2}\,b-c\right) \\
    \ge&\ 8\sin^3\frac{\theta}{2}+6(\theta-\sin\theta)-\left(\theta-2\sin\frac{\theta}{2}\right)(4\pi^2-6\pi\theta+3\theta^2)\;,
  \end{align*}
  which can be easility verified to be positive for $0<\theta<\pi$, hence the proof is completed.
\end{proof}

\section{Convergence and Error Estimate}
\label{sec:conv}

In this section we extend the supraconvergence to arbitrary initial data with sufficient regularity, provided that the semi-discretized HV method is stable, in the sense that $\mathcal{S}'(\Pe)$ is contained in the open left complex plane, see~\cref{lm:stab_crit}.
Then we have the following theorem on supraconvergence and error estimate in discrete $L^2$-norm.
\begin{theorem}\label{thm:conv}
  Assume~\cref{eq:prelim_eqn} with periodic boundary condition and initial data $w_0$ is solved by a stable HV method as given in~\cref{thm:acry}, and let $P=\min(P_1+2,P_2,P_3+2)$ if $c\ne0$ and $P=\min(P_2,P_3+2)$ if $c=0$.
  Then fixing $T>0$, there exists $h_0>0$ and $C_1, C_2>0$ that are independent of $h$ and $w_0$, such that for all $h\le h_0$ and $w_0\in H^{P+2}_{\per}([0, L])$:
  \begin{equation}\label{eq:conv}
    \sup_{0\le t\le T}\nrm{\bs{w}_h(t)-\bs{w}^\star_h(t)}_h \le C_1\nrm{\bs{w}_h(0)-\bs{w}^\star_h(0)}_h + C_2h^P\nrm{w_0}_{H^{P+2}}\;,
  \end{equation}
  where $\bs{w}_h(t)$ is the vector of discrete solution as before:
  \begin{equation}\label{eq:conv_solvec}
    \bs{w}_h(t) = [\overline{w}_{1/2}(t),\;\overline{w}_{3/2}(t),\;\cdots,\;\overline{w}_{N-1/2}(t),\;w_0(t),\;w_1(t),\;\cdots,\;w_{N-1}(t)]^T\;,
  \end{equation}
  and $\bs{w}^\star_h(t)$ is the exact solution vector on the same uniform grid. 
  The discrete vector norm $\nrm{\cdot}_h$ is defined by $\nrm{\bs{w}}_h^2 = h\sum_{j=0}^{N-1}(\abs{\overline{w}_{\phf{j}}}^2+\abs{w_j}^2)$; and $H^{P+2}_{\per}([0, L])$ is the standard Sobolev space of functions whose $L$-periodization has up to $(P+2)$-th order derivatives that are square integrable on $[0, L]$.
\end{theorem}
\begin{proof}
  The proof focuses on non-zero advection, and the modification in the case of $c=0$ is straightforward.
  The solution vector solves the linear system:
  \begin{equation}\label{eq:conv_ode_semi}
    \bs{w}' = \frac{\nu}{h^2}\bs{M}\bs{w}\;,\quad
    \bs{M} = \bs{K} - \Pe\bs{D}\;,
  \end{equation}
  see~\cref{eq:stab_sol}, \cref{eq:stab_mat_dx}, and~\cref{eq:stab_mat_dxx}.
  First, we show that there exists a constant $C_1>0$ that is independent of $h$ (or $N$) and a positive number $h_0>0$, such that for all $h\le h_0$:
  \begin{equation}\label{eq:conv_exp_bound}
    \sup_{t\ge0}\nrm{\exp\left(\frac{\nu t}{h^2}\bs{M}\right)} \le C_1\;,
  \end{equation}
  Following the proof of~\cref{lm:stab_eig}, one may pick the eigenvectors to $\bs{S}$ as the orthonormal set of unit vectors $\bs{v}_k = \frac{1}{\sqrt{N}}[1, s_k, s_k^2, \cdots, s_k^{N-1}]^T, k=1,\cdots,N$, which shows that:
  \begin{displaymath}
    \nrm{\exp\left(\frac{\nu t}{h^2}\bs{M}\right)} \le \max_{1\le k\le N}\nrm{\exp\left(\frac{\nu t}{h^2}\bs{M}_k\right)} \le \sup_{0\le\theta\le2\pi}\nrm{\exp\left(\frac{\nu t}{h^2}\bs{M}(\theta,\Pe)\right)}\;,
  \end{displaymath}
  where
  \begin{equation}\label{eq:conv_exp_smallm}
    \bs{M}(\theta,\Pe)=\begin{bmatrix} (s-1)G^c(s) & (s-1)H^c(s) \\ A(s) & B(s)\end{bmatrix}-\Pe\begin{bmatrix} 0 & s-1 \\ G(s) & H(s) \end{bmatrix},\quad s=e^{i\theta}\;,
  \end{equation}
  and $\bs{M}_k=\bs{M}(\theta_k,\Pe)$, $\theta_k = 2k\pi/N$ identify with~\cref{eq:stab_eig_smallm}.
  Comparing~\cref{eq:acry_semi_fun_dx_c}--\cref{eq:acry_semi_fun_dxx_c} and~\cref{eq:stab_poly_dx}, ~\cref{eq:stab_poly_dx_c} and~\cref{eq:stab_poly_dxx_c}, there is:
  \begin{align*}
    &(s-1)G(s) = (i\theta)^2\alpha(\theta)\;, 
    &&H(s) = (i\theta)\beta(\theta)\;, 
    &&(s-1)G^c(s) = (i\theta)^2\alpha^c(\theta)\;,\\ 
    &H^c(s) = (i\theta)\beta^c(\theta)\;, 
    &&(s-1)A(s) = (i\theta)^3\mathcal{a}(\theta)\;,
    &&B(s) = (i\theta)^2\mathcal{b}(\theta)\;;
  \end{align*}
  which establishes the connection between $\bs{M}$ and $\bb{C}$ that is given by~\cref{eq:acry_semi_mat}:
  \begin{displaymath}
    \frac{\nu}{h^2}\bs{M}(\theta,\Pe) = -\begin{bmatrix} 1 & 0 \\ 0 & \frac{i\theta}{e^{i\theta}-1} \end{bmatrix}
      \bb{C}(\theta)
      \begin{bmatrix} 1 & 0 \\ 0 & \frac{e^{i\theta}-1}{i\theta} \end{bmatrix}
  \end{displaymath}
  Therefore, the two eigenvalues of $\bs{M}(\theta,\Pe)$, denoted $\mu_{1,2}(\theta,\Pe)$, are given by $\mu_{1,2}=-\frac{h^2}{\nu}\lambda_{1,2}$ in~\cref{eq:acry_semi_eigs}; we will discuss the spectral norm of $\bs{M}(\theta,\Pe)$ for small $h$ and large $h$ separately.

  First we consider small $\theta$. 
  By~\cref{eq:acry_semi_eig_1} and~ \cref{eq:acry_semi_eig_2}, we see $\mu_1(0,\Pe)=0$ and $\mu_2(0,\Pe)=B_0 - \Pe b_0\le B_0 <0$.
  Furthermore, fixing $\Pe_0>0$ one has:
  \begin{displaymath}
    \mu_1 = - i\Pe\theta -\theta^2 + O(\abs{\theta}^3)\;,\quad
    \mu_2 = B_0 - \Pe b_0 + O(\abs{\theta})\;.
  \end{displaymath}
  uniformly for all $0\le\Pe\le \Pe_0$.
  Indeed, these estimates are valid at fixed $\Pe$ using~\cref{eq:acry_semi_eigs}; and its uniformity in $\Pe$ can be see from~\cref{eq:conv_exp_smallm}.
  Therefore we may find a $\theta_0>0$ such that:
  \begin{equation}\label{eq:conv_meig_smalltheta}
    -2\theta^2 < \oname{Re}\mu_1 < -\frac{1}{2}\theta^2\;,\quad
    \oname{Re}\mu_2 < \frac{1}{2}B_0\;,\quad\forall 0<\abs{\theta}<\theta_0, 0\le\Pe\le \Pe_0\;.
  \end{equation}
  Shrink it if need so that $8\theta_0^2<\abs{B_0}$, we have for all these $\theta$:
  \begin{displaymath}
    \oname{Re}(\mu_1-\mu_2) > \frac{1}{2}\abs{B_0}-2\theta_0^2 > \frac{1}{4}\abs{B_0}\;.
  \end{displaymath}
  Define the projections onto the two eigenspaces:
  \begin{displaymath}
    \bs{\Pi}_1 = \frac{\bs{M}-\mu_2\bs{I}_2}{\mu_1-\mu_2}\;,\quad
    \bs{\Pi}_2 = \frac{\bs{M}-\mu_1\bs{I}_2}{\mu_2-\mu_1}\;,
  \end{displaymath}
  one thus has:
  \begin{displaymath}
    \exp\left(\frac{\nu t}{h^2}\bs{M}\right) = \exp\left(\frac{\nu t\mu_1}{h^2}\right)\bs{\Pi}_1+\exp\left(\frac{\nu t\mu_2}{h^2}\right)\bs{\Pi}_2\Rightarrow
    \nrm{\exp\left(\frac{\nu t}{h^2}\bs{M}\right)}\le\nrm{\bs{\Pi}_1}+\nrm{\bs{\Pi_2}}\le C_3\;,
  \end{displaymath}
  where:
  \begin{displaymath}
    C_3 = \frac{4}{\abs{B}_0}\max_{\abs{\theta}\le\theta_0,0\le\Pe\le\Pe_0}\left(2\nrm{\bs{M}(\theta,\Pe)}+\abs{\mu_1(\theta,\Pe)}+\abs{\mu_2(\theta,\Pe)}\right)
  \end{displaymath}
  is a positive constant that only depends on the choice of the HV-DDOs, $\Pe_0$, and $\theta_0$.

  Next we consider the other values of $\theta$. 
  The compact set $\left\{\mu_{1,2}(\theta,\Pe):\,\theta_0\le\abs{\theta}\le\pi, 0\le\Pe\le\Pe_0\right\}$ is a subset of $\mathcal{S}'(\Pe)$, which by assumption is contained in the open left complex plane; therefore there exists a $\mu_0>0$, such that:
  \begin{equation}\label{eq:conv_meig_largetheta}
    \oname{Re}\mu_{1,2} < -\mu_0\;,\quad\forall \theta_0\le\abs{\theta}\le\pi,\ 0\le\Pe\le\Pe_0\;.
  \end{equation}
  Let $\bs{M}$ be similar to the upper triangular matrix $\bs{U}\eqdef\begin{bmatrix} \mu_1 & m_{12} \\ 0 & \mu_2 \end{bmatrix}$ after a unitary Schur decomposition, where $m_{12}$ has the estimate:
  \begin{displaymath}
    \abs{m_{12}}\le\nrm{\bs{U}}=\nrm{\bs{M}}\le C_4\eqdef\max_{\theta_0\le\abs{\theta}\le\pi, 0\le\Pe\le\Pe_0}\nrm{\bs{M}(\theta,\Pe)}\;,
  \end{displaymath}
  Hence one has the following estimate for $\exp\left(\frac{\nu t}{h^2}\bs{M}\right)$:
  \begin{align*}
    \nrm{\exp\left(\frac{\nu t}{h^2}\bs{M}\right)} 
    &= \nrm{\exp\left(\frac{\nu t}{h^2}\bs{U}\right)} 
    = \nrm{\begin{bmatrix} e^{\frac{\nu\mu_1t}{h^2}} & m_{12}\int_0^{\frac{\nu t}{h^2}}e^{\mu_1(\frac{\nu t}{h^2}-s)}\,e^{\mu_2s}\,ds \\ 0 & e^{\frac{\nu\mu_2t}{h^2}} \end{bmatrix}} \\
    &\le \left(1+C_4\frac{\nu t}{h^2}\right)e^{-\frac{\mu_0\nu t}{h^2}} \le 1+\frac{C_4}{\mu_0e}\;,\quad \forall t\ge0\;.
  \end{align*}

  Combining the two cases we obtain~\cref{eq:conv_exp_bound} with $C_1=\max\left(C_3,1+\frac{C_4}{\mu_0 e}\right)$ for all $h\le h_0\eqdef \frac{\nu\Pe_0}{c}$.
  Note that if the problem is pure duffisive, i.e., $c=0$, the previous arguments still hold with $\Pe_0=0$ and $h_0=\infty$.

  \medskip

  While~\cref{eq:conv_exp_bound} allows us to estimate the propagation of the initial discretization error, the plan is to prove the convergence of the HV scheme using the Duhamel estimate, for example, as outlined in \cite{BGustafsson:2013a}.
  However, the standard Duhamel's principle is insufficient to prove the supraconvergence.
  %Indeed, denoting $\bs{w}^\star_h(t)$ the solution vector computed from the exact solution, c.f.,~\cref{eq:conv_solvec}, then it satisfies the following differential equation:
  Indeed, the exact solution vector $\bs{w}^\star_h(t)$ satisfies the following differential equation:
  \begin{equation}\label{eq:conv_ode_star}
    (\bs{w}^\star_h)' = \frac{\nu}{h^2}\bs{M}\bs{w}^\star_h + \bs{\tau}_h\;,
  \end{equation}
  where $\bs{\tau}_h(t)$ represents the local truncation errors.
  Let the numerical error be denoted $\bs{e}_h=\bs{w}_h-\bs{w}^\star_h$, then subtracting~\cref{eq:conv_ode_star} from~\cref{eq:conv_ode_semi} gives:
  \begin{displaymath}
    \bs{e}_h' = \frac{\nu}{h^2}\bs{M}\bs{e}_h - \bs{\tau}_h
  \end{displaymath}
  and by Duhamel's principle:
  \begin{equation}\label{eq:conv_err}
    \bs{e}_h(t) = \exp\left(\frac{\nu t}{h^2}\bs{M}\right)\bs{e}_h(0) - \int_0^t\exp\left(\frac{\nu(t-s)}{h^2}\bs{M}\right)\bs{\tau}_h(s)\,ds\;.
  \end{equation}
  By the construction of the HV-DDOs, the local truncation error has the estimate $\nrm{\bs{\tau}_h}_h = O(h^Q)$, where $Q=\min(P_1,P_2,P_3)$, see~\cref{eq:prelim_hv_dx_ddo}--\cref{eq:prelim_hv_dxx_c_ddo}; thus a direct norm-vector estimate can give at most an $O(h^Q)$ estimation of the second term of~\cref{eq:conv_err}, which is less than the required $O(h^P)=O(h^{\min(P_1+2,P_2,P_3+2)})$.
  To get around this limitation, we need to explore more detailed structures of $\bs{\tau}_h$. 
  Using a correction technique motivated by~\cite{DBouche:2005a}, we compare $\bs{w}_h$ with $\tilde{\bs{w}}_h\eqdef\bs{w}^\star_h+\bs{u}_h$ for some correction term $\bs{u}_h$; the residual of $\tilde{\bs{w}}_h$ is:
  \begin{displaymath}
    \tilde{\bs{\tau}}_h \eqdef \tilde{\bs{w}}'_h-\frac{\nu}{h^2}\bs{M}\tilde{\bs{w}}_h = \bs{\tau}_h + \bs{u}_h' - \frac{\nu}{h^2}\bs{M}\bs{u}_h\;.
  \end{displaymath}
  The difference $\tilde{\bs{e}}_h=\bs{w}_h-\tilde{\bs{w}}_h=\bs{e}_h-\bs{u}_h$ satisfies $\tilde{\bs{e}}_h'=\frac{\nu}{h^2}\bs{M}\tilde{\bs{e}}_h-\tilde{\bs{\tau}}_h$ and by Duhamel's principle:
  \begin{displaymath}
    \tilde{\bs{e}}(t) = \exp\left(\frac{\nu t}{h^2}\bs{M}\right)\tilde{\bs{e}}_h(0) - \int_0^t\exp\left(\frac{\nu(t-s)}{h^2}\bs{M}\right)\tilde{\bs{\tau}}_h(s)\,ds\;,
  \end{displaymath}
  or equivalently:
  \begin{equation}\label{eq:conv_err_bound_moderr}
    \bs{e}_h(t) = \bs{u}_h(t) + \exp\left(\frac{\nu t}{h^2}\bs{M}\right)(\bs{e}_h(0)-\bs{u}_h(0)) - \int_0^t\exp\left(\frac{\nu(t-s)}{h^2}\bs{M}\right)\tilde{\bs{\tau}}_h(s)\,ds\;.
  \end{equation}
  Therefore we can obtain the required bound if the correction term $\bs{u}_h(t)$ is chosen such that:
  \begin{displaymath}
    \sup_{0\le t\le T}\nrm{\bs{u}_h(t)}_h + \int_0^T\nrm{\tilde{\tau}_h(t)}_hdt\le \tilde{C}h^P\nrm{w_0}_{H^{P+2}}\;,
  \end{displaymath}
  for some constant $\tilde{C}>0$.
  To this end, we write each vector in the form~\cref{eq:conv_solvec} as $\bs{v}_h=\begin{bmatrix} \bs{v}_{ch} \\ \bs{v}_{nh}\end{bmatrix}$, with the two halves accounting for the cell averages and nodal values, respectively; then we define:
  \begin{equation}\label{eq:conv_cor}
    \bs{u}_h = \begin{bmatrix}
      \bs{0} \\
      \frac{h^2}{\nu B_0-cb_0h}\bs{\tau}_{nh}
    \end{bmatrix}\;,
  \end{equation}
  which is well-defined because by the construction of HV-DDOs $B_0<0$ and $cb_0\ge0$ for central or upwind-biased $[\mathcal{D}_x]$.

  To bound the discrete $L^2$-norm by a Sobolev norm, we define a function $\tau_h(x,t)$ that equals the nodal truncation error defined at arbitrary $x$ by translating the stencil: Defining the shifting and averaging operators for a scalar function $v(x)$ by $Ev(x) = v(x+h)$ and $Av(x) = \frac{1}{h}\int_x^{x+h}v(y)dy$, then
  %\begin{equation}\label{eq:conv_nshift}
  \begin{align*}
    \tau_h(x,t) =&\ c\left[\frac{1}{h}\left(\sum_{k=-l}^{r-1}\alpha_kAE^k+\sum_{k=-l'}^{r'}\beta_kE^k\right)w^\star(x,t)-\partial_xw^\star(x,t)\right] \\
    &\ - \nu\left[\frac{1}{h^2}\left(\sum_{k=-q}^{q-1}\mathcal{a}_kAE^k+\sum_{k=-q'}^{q'}\mathcal{b}_kE^k\right)w^\star(x,t)-\partial_x^2w^\star(x,t)\right]\;,
  \end{align*}
  %\end{equation}
  which satisfies $\tau_{nh,j}(t) = \tau_h(x_j,t)$ for all $j$.
  %Because the nodal local truncation error has order $\min(P_1,P_3)\ge P_2$, using the one-dimensional Sobolev embedding theorem, we have the following useful estimate on pointwise values for some constant $C_5',C_5''>0$:
  Then there exist constants $C_5',C_5''>0$ such that:
  \begin{align}
    \label{eq:conv_pointbound}
    &\nrm{\tau_h}_{\infty} \le C_5'\nrm{\tau_h}_{H^1} \le C_5''h^{P-2}\nrm{w^\star(\cdot,t)}_{H^{P+1}}\;, \\
    \label{eq:conv_pointbound_der}
    &\nrm{\tau_h}_{W^{2,\infty}}+\nrm{\partial_t\tau_h}_{\infty}
    \le (1+c+\nu)C_5'\nrm{\tau_h}_{H^3}
    \le (1+c+\nu)C_5''h^{P-2}\nrm{w^\star(\cdot,t)}_{H^{P+3}}\;,
  \end{align}
  where we used the fact that $\partial_t\tau_h=-c\partial_x\tau_h+\nu\partial_x^2\tau_h$; and the first inequalities come from the one-dimensional Sobolev embedding theorem and the second ones are a result of the Taylor theorem and the fact that the nodal local truncation error has order $\min(P_1,P_3)\ge P-2$.
  Therefore: 
  \begin{align}
    \notag
    \nrm{\bs{u}_h(t)}_h &\le \frac{h^2}{\nu \abs{B_0}}\left(h\sum_{j=0}^{N-1}\abs{\tau_h(x_j,t)}^2\right)^{\frac{1}{2}} \le C_5h^{P}\nrm{w^\star(\cdot,t)}_{H^{P+1}} \\ 
    \label{eq:conv_bound_cor}
    &\le  C_5h^{P}\nrm{w_0}_{H^{P+1}} \le C_5h^P\nrm{w_0}_{H^{P+2}}\;, 
    %C_5'h^{Q+2}\nrm{w^\star(\cdot,t)}_{H^{Q+2}} 
    %  &\le C_5h^P\nrm{w^{\star}(\cdot,t)}_{H^{P+2}} \le C_5h^P\nrm{w^{\star}_0}_{H^{P+2}}\;,
  \end{align}
  where $C_5=\sqrt{L}C_5''/(\nu\abs{B_0})$ and we used the standard estimate on solutions to the periodic advection-diffusion equations by their initial conditions. 
  %for some constants $C_5'$ and $C_5$, as $h\le h_0$ and $Q\le P\le Q+2$.
  %Here we also used a standard estimate of solutions to periodic linear advection-diffusion equations~\cite{APazy:1983a}.

  To estimate the corrected error $\tilde{\bs{\tau}}_h$, we write $\bs{M}=\begin{bmatrix}\bs{M}_{cc} & \bs{M}_{cn} \\ \bs{M}_{nc} & \bs{M}_{nn} \end{bmatrix}$ and compute:
  \begin{align*}
    \tilde{\bs{\tau}}_h = \begin{bmatrix} \tilde{\bs{\tau}}_{ch} \\ \tilde{\bs{\tau}}_{nh} \end{bmatrix} = \begin{bmatrix}
      \bs{\tau}_{ch} - \frac{\nu}{\nu B_0-cb_0h}\bs{M}_{cn}\bs{\tau}_{nh} \\
      \bs{\tau}_{nh} + \frac{h^2}{\nu B_0-cb_0h}\bs{\tau}_{nh}' - \frac{\nu}{\nu B_0-cb_0h}\bs{M}_{nn}\bs{\tau}_{nh}
    \end{bmatrix}\;.
  \end{align*}
  Noting that:
  \begin{displaymath}
    \tau_{ch,\phf{j}} = -\frac{\nu}{h^2}\left[\sum_{k=-p}^{p-1}\alpha_k^cAE^k+\sum_{k=-p'}^{p'}\beta_k^cE^k\right](E-I)w^\star(x_j,t)+\frac{\nu}{h}(E-I)\partial_xw^\star(x_j,t)\;,
  \end{displaymath}
  where $I$ is the identity map; and a similar argument using Sobolev embedding and the Taylor theorem (the order of $\tau_{ch,\phf{j}}$ is $P_2\ge P$) yields:
  \begin{equation}\label{eq:conv_est_err_c}
    \abs{\tau_{ch,\phf{j}}} \le C_6h^P\nrm{w^\star(\cdot,t)}_{H^{P+3}}\;,
  \end{equation}
  for some $C_6>0$.
  Therefore we obtain the following estimate on $\tilde{\tau}_{ch,\phf{j}}$:
  \begin{align}
    \notag
    &\ \tilde{\tau}_{ch,\phf{j}}(t) = \tau_{ch,\phf{j}} - \frac{\nu}{\nu B_0-cb_0h}\left[\sum_{k=-p'}^{p'}\beta_k^cE^k-\frac{ch}{\nu}\right](E-I)\tau_h(x_j,t) \\
    \notag
    \Rightarrow&\ \abs{\tilde{\tau}_{ch,\phf{j}}} \le C_6h^P\nrm{w^\star(\cdot,t)}_{H^{P+3}} + \frac{1}{\abs{B_0}}\left(h\abs{c_\beta}+\frac{ch}{\nu}\right)\nrm{(E-I)\tau_h}_{W^{1,\infty}} \\
    \label{eq:conv_est_moderr_c}
    \Rightarrow&\ \abs{\tilde{\tau}_{ch,\phf{j}}} \le C_7h^P\nrm{w^\star(\cdot,t)}_{H^{P+3}}\;.
  \end{align}
  where $c_\beta=\sum_{k=-p'}^{p'}k\beta_k^c\le0$ and $C_7=C_6+\frac{\abs{c_\beta}+c/\nu}{\abs{B_0}}(1+c+\nu)C_5''$.

  Similarly for the corrected nodal error there is:
  \begin{align*}
    \tilde{\tau}_{nh,j}(t) = \tau_h(x_j,t) + \frac{h^2}{\nu B_0-cb_0h}\partial_t\tau_h(x_j,t) - \frac{\nu}{\nu B_0-cb_0h}\left(\sum_{k=-q'}^{q'}\mathcal{b}_kE^k-\frac{ch}{\nu}\sum_{k=-l'}^{r'}\beta_kE^k\right)\tau_h(x_j,t)\;.
  \end{align*}
  The leading terms in the two sums cancel exactly with $\tau_h(x_j,t)$, along with the estimate:
  \begin{displaymath}
    \abs{\sum_{k=-q'}^{q'}\mathcal{b}_kE^k\tau_h-B_0\tau_h} \le C_8'h^2\nrm{\tau_h}_{W^{2,\infty}}\;,\quad
    \abs{\sum_{k=-l'}^{r'}\beta_kE^k\tau_h-b_0\tau_h} \le C_8''h\nrm{\tau_h}_{W^{1,\infty}}\;,
  \end{displaymath}
  therefore we obtain the estimate
  \begin{equation}\label{eq:conv_est_moderr_n}
    \abs{\tilde{\tau}_{nh,j}(t)} \le \frac{h^2}{\nu \abs{B_0}}\nrm{\partial_t\tau_h}_{\infty}+\frac{1}{\abs{B_0}}\left(C_8'h^2+\frac{ch^2}{\nu}C_8''\right)\nrm{\tau_h}_{W^{2,\infty}}
    \le C_8h^P\nrm{w^\star(\cdot,t)}_{H^{P+3}}\;.
  \end{equation}
  Here $C_8=\left(\frac{1}{\nu \abs{B_0}}+\frac{C_8'}{\abs{B_0}}+\frac{cC_8''}{\nu\abs{B_0}}\right)(1+c+\nu)C_5''$.

  Putting~\cref{eq:conv_est_moderr_c} and~\cref{eq:conv_est_moderr_n} together, we see:
  \begin{equation}\label{eq:conv_est_moderr}
    \nrm{\tilde{\bs{\tau}}(t)}_h \le C_9h^P\nrm{w^\star(\cdot,t)}_{H^{P+3}}\;,
  \end{equation}
  where $C_9=\sqrt{L}\sqrt(C_7^2+C_8^2)$.
  By a standard estimate of solutions to periodic linear advection-diffusion equations~\cite{APazy:1983a}, one has:
  \begin{equation}\label{eq:conv_est_ade}
    \int_0^T\nrm{w^\star(\cdot,t)}_{H^{P+3}}dt\le \left(T+\sqrt{\frac{2T}{e\nu}}\right)\nrm{w_0}_{H^{P+2}}\;,
  \end{equation}
  Therefore using~\cref{eq:conv_exp_bound} and~\cref{eq:conv_bound_cor}, we obtain an estimate of $\nrm{\bs{e}_h(t)}_h$ for all $0\le t\le T$ from~\cref{eq:conv_err_bound_moderr}:
  \begin{align*}
    \nrm{\bs{e}_h(t)}_h &\le \nrm{\bs{u}_h(t)}_h + C_1\left(\nrm{\bs{e}_h(0)}_h+\nrm{\bs{u}_h(0)}_h\right) + C_1C_9\left(T+\sqrt{\frac{2T}{e\nu}}\right)h^P\nrm{w_0}_{H^{P+2}} \\
    &\le C_1\nrm{\bs{e}_h(0)}_h + \left[C_5+C_1C_5+C_1C_9\left(T+\sqrt{\frac{2T}{e\nu}}\right)\right]h^P\nrm{w_0}_{H^{P+2}}\;,
  \end{align*}
  %\begin{align*}
  %  \nrm{\bs{e}_h(t)}_h &\le \nrm{\bs{u}_h(t)}_h + C_1\left(\nrm{\bs{e}_h(0)}_h+\nrm{\bs{u}_h(0)}_h\right) + C_1C_{10}Th^P\nrm{w_0}_{H^{P+2}} \\
  %  &\le C_1\nrm{\bs{e}_h(0)}_h + \left[C_5+C_1C_5+C_1C_{10}T\right]h^P\nrm{w_0}_{H^{P+2}}\;,
  %\end{align*}
  which completes the proof of~\cref{eq:conv}.
\end{proof}
By the theorem and the stability result for all central HV schemes in~\cref{thm:stab_ade}, we see that all these methods compute solutions that converge to the analytical one supraconvergently as $h\to0$.

\section{Numerical Examples}
\label{sec:num}
In this section we present numerical examples verifying the theoretical order of accuracy and stability.
%Particularly, we shall verify the spatial order of accuracy (\cref{thm:acry}) by pairing a selection of HV-DDOs with an explicit time integrator.
Particularly, we shall verify \cref{thm:acry} by pairing a selection of HV-DDOs with an explicit time integrator in~\cref{sec:num_full}; and stability of selected central HV methods is verified by plotting the trajectory $\mathcal{S}(\Pe)$ in~\cref{sec:num_cent}.
%Next, the stability of sample central HV methods for both diffusion and advection-diffusion equations will be verified by plotting the trajectory $\mathcal{S}(\Pe)$, which all eigenvalues associated with the ODE system~\cref{eq:prelim_hv_semi_cell} lie on. 

\subsection{Accuracy verification of fully-discretized HV methods}
\label{sec:num_full}
Let the HV method be constructed by a $P_1^\powth\textrm{\sc-}$order $[\mathcal{D}_x]$, a $P_2^\powth\textrm{\sc-}$order $[\mathcal{D}_x^c]$, and a $P_3^\powth\textrm{\sc-}$order $[\mathcal{D}_{xx}^c]$, then~\cref{thm:acry} predicts that the spatial order of accuracy is $\min(P_1+2,P_2,P_3+2)$.
To this end, we consider below four categories of tests: (1) $P_1+2<\min(P_2,P_3+2)$, (2) $P_2<\min(P_1+2,P_3+2)$, (3) $P_3+2<\min(P_1+2,P_2)$, and (4) $P_1+2=P_2=P_3+2$.
%Note that the last one includes combinations that are optimal to achieve a certain spatial order.
%In all categories, we consider tests with both upwind-baised $[\mathcal{D}_x]$ and central $[\mathcal{D}_x]$.
%The operator $[\mathcal{D}_x]$ is given by either~\cref{tb:acry_dx} or~\cref{tb:acry_dxc}, $[\mathcal{D}_x^c]$ is chosen from~\cref{tb:acry_dxc}, and $[\mathcal{D}_{xx}^c]$ is selected from~\cref{tb:acry_dxx}.
In all tests we solve the equation:
\begin{equation}\label{eq:num_full_eqn}
  w_t + w_x - 0.01\,w_{xx} = 0\;,\quad (x,t)\in[0,\ 1]^2
\end{equation}
with periodic boundary conditions and an Gaussian pulse for the initial data:
\begin{equation}\label{eq:num_full_ic}
  w(x,0) = \exp\left(-100(x-\hf)^2\right)\;.
\end{equation}
This article focuses on the spatial order of accuracy, thus we pair all HV discretizations with the standard explicit second-order Runge-Kutta scheme and a sufficiently small time step size. % $\Delta t=10^{-5}$.

Since these fully discrete solutions also contain temporal discretization error, we compare numerical solutions on three successively refined grids to estimate the spatial order using Richardson's extrapolation strategy (see for example~\cite{EIsaacson:1994a}).
The time step is chosen sufficiently small so that the influence of temporal discretization on the cross-grid differences is negligible.
%the temporal discretization error is negligible relative to the cross-grid differences used to estimate the spatial convergence order.
%Due to the unknown temporal discretization error, there is no exact solution to compute the numerical errors; thus we adopt the Richardson extrapolation strategy  and estimate the order of convergence using three consecutive grids. % as described below.
In general, let $w(h)$ be the approximation to some scalar quantity computed using the parameter $h$, and $w(0)$ be the limit of $w(h)$ as $h\to0$, then for a $P^\powth\textrm{\sc-}$order method one has:
\begin{displaymath}
  w(h) = w(0) + Ch^P + O(h^{P+1})
\end{displaymath}
for some non-zero constant $C$ that is independent of $h$.
Computing the numerical solutions on a sequence of grids with cell size $h$, $h/2$, and $h/4$, one has:
\begin{displaymath}
  w(h) \approx w(0) + Ch^P\;,\quad
  w(h/2) \approx w(0) + \frac{C}{2^P}h^P\;,\quad
  w(h/4) \approx w(0) + \frac{C}{2^{2P}}h^P\;;
\end{displaymath}
therefore one can estimate $P$ as:
\begin{equation}\label{eq:num_full_richard}
  \frac{w(h)-w(h/2)}{w(h/2)-w(h/4)} \approx \frac{h^P-2^{-P}h^P}{2^{-P}h^P-2^{-2P}h^P} = 2^P\quad\Rightarrow\quad
  P \approx \log_2\frac{w(h)-w(h/2)}{w(h/2)-w(h/4)}\;.
\end{equation}
%In our case, instead of computing $w(h)$ we compute directly $w(h)-w(h/2)$ using two consecutive grids, and then the order of convergence can be estimated using~\cref{eq:num_full_richard}$_2$.
In our case, let the numerical solutions to~\cref{eq:num_full_eqn} at $T=1$ computed on a grid with cell size $h$ be denoted:
\begin{displaymath}
  w^h_0,\cdots,w^h_{N_h}\;;\quad
  \overline{w}^h_{1/2},\cdots,\overline{w}^h_{\mhf{N_h}}\;,
\end{displaymath}
where $N_h=\frac{1}{h}$ is the number of cells, then we can estimate the $L^\infty$-norm of the cross-grid difference in nodal and cell-averaged solutions using grids with cell sizes $h$ and $h/2$ as:
\begin{equation}\label{eq:num_full_err_linf}
  E^\infty_{h,h/2} = \max_{0\le j<N_h}\abs{w^h_j-w^{h/2}_{2j}}\;,\quad
  \overline{E}^\infty_{h,h/2} = \max_{0\le j<N_h}\abs{\overline{w}^h_{\phf{j}}-\frac{\overline{w}^{h/2}_{\phf{2j}}+\overline{w}^{h/2}_{2j+3/2}}{2}}\;, 
\end{equation}
and the $L^2$-norm of the difference is computed as:
\begin{equation}\label{eq:num_full_err_lone}
  E^2_{h,h/2} = \left(\sum_{j=0}^{N_h-1}h\abs{w^h_j-w^{h/2}_{2j}}^2\right)^{\frac{1}{2}}\;,\quad
  \overline{E}^2_{h,h/2} = \left(\sum_{j=0}^{N_h-1}h\abs{\overline{w}^h_{\phf{j}}-\frac{\overline{w}^{h/2}_{\phf{2j}}+\overline{w}^{h/2}_{2j+3/2}}{2}}^2\right)^{\frac{1}{2}}\;.
\end{equation}
In all tests below, we will report these cross-grid differences as well as the convergence orders computed using~\cref{eq:num_full_richard}; for example, the $L^2$ convergence rate in nodal solutions computed using three grids with cell sizes $4h$, $2h$, and $h$ is given by $\log_2\left(E^2_{4h,2h}/E^2_{2h,h}\right)$, etc.

For each test, we fix the time step size at $\Delta t=10^{-5}$, which is verified to be sufficiently small for both stable computations and negligible temporal error in cross-grid differences (numerical results not shown), and compute the numerical solutions on a sequence of $6$ uniform grids with number of cells ranging from $32$ to $1024$. %are used to assess the numerical order of accuracy.
Both cross-grid differences and convergence rates computed as described earlier are reported.
Using the second row of~\cref{tb:num_full_p1_2} as an example, $h=1/128$ indicates the error is computed using the two grids with $h=1/64$ and $h=1/128$, and the convergence rate is estimated using the three grids with $h=1/32$, $h=1/64$, and $h=1/128$.
Thus the row starting with $h=1/64$ only has cross-grid difference but no convergence order, as the coarsest grid has $h=1/32$.

%\begin{remark}\label{rm:num_ic}
Lastly, in all tests the initial values for cell-averaged variables are computed using Gauss-Legendre rules with order of accuracy that is at least two orders higher than the one predicted by~\cref{thm:acry}.
%\end{remark}

\subsubsection{Category 1: $P_1+2<\min(P_2,P_3+2)$}
\label{sec:num_full_p1}
%Three tests are considered in this category:
Two tests are considered in this category:
\begin{align}
  %\label{eq:num_full_p1_142}
  %&[\mathcal{D}_x] = [\mathcal{D}_x^{(1)}]\;, &&[\mathcal{D}_x^c] = [\mathcal{D}_x^{(c-4)}]\;, &&[\mathcal{D}_{xx}^c] = [\mathcal{D}_{xx}^{(c-2)}]\;; \\
  \label{eq:num_full_p1_c264}
  &[\mathcal{D}_x] = [\mathcal{D}_x^{(c-2)}]\;, &&[\mathcal{D}_x^c] = [\mathcal{D}_x^{(c-6)}]\;, &&[\mathcal{D}_{xx}^c] = [\mathcal{D}_{xx}^{(c-4)}]\;; \\
  \label{eq:num_full_p1_364}
  &[\mathcal{D}_x] = [\mathcal{D}_x^{(3)}]\;, &&[\mathcal{D}_x^c] = [\mathcal{D}_x^{(c-6)}]\;, &&[\mathcal{D}_{xx}^c] = [\mathcal{D}_{xx}^{(c-4)}]\;.
\end{align}
%The estimated errors and convergence rates are summarized in~\cref{tb:num_full_p1_1}--\cref{tb:num_full_p1_3}.
The estimated cross-grid differences and convergence rates are summarized in~\cref{tb:num_full_p1_2} and~\cref{tb:num_full_p1_3}, which confirm the convergence rate of $P_1+2$ in both $L^2$-norm and $L^\infty$-norm.
%We can see that the convergence orders in both nodal solutions and cell-averaged solutions computed using both $L^2$-errors and $L^\infty$-errors are $P_1+2$, as predicted by~\cref{thm:acry}.
%\begin{table}\centering
%  \caption{Estimated $L^2$- and $L^\infty$-errors and convergence rates for the HV method~\cref{eq:num_full_p1_142}: $P_1=1$, $P_2=4$, $P_3=2$, thus $\min(P_1+2,P_2,P_3+2)=3$.}
%  \label{tb:num_full_p1_1}
%  \small
%  \begin{tabular}{@{}llcclcclcclc@{}}
%    \toprule[.5mm]
%             & \multicolumn{2}{c}{Nodal solution ($L^2$)} &
%             & \multicolumn{2}{c}{Cell averages ($L^2$)} &
%             & \multicolumn{2}{c}{Nodal solution ($L^\infty$)} &
%             & \multicolumn{2}{c}{Cell averages ($L^\infty$)} \\ \cmidrule[.2mm]{2-3} \cmidrule[.2mm]{5-6} \cmidrule[.2mm]{8-9} \cmidrule[.2mm]{11-12}
%      $h$    & Error & Order & $\ $ & Error & Order & $\ $ & Error & Order & $\ $ & Error & Order \\ \cmidrule[.3mm]{2-12}
%      1/64   & 5.31e-3 &      & & 5.38e-3 &      & & 1.03e-2 &      & & 1.03e-2 &      \\
%      1/128  & 9.56e-4 & 2.47 & & 9.70e-4 & 2.47 & & 1.84e-3 & 2.48 & & 1.92e-3 & 2.43 \\
%      1/256  & 1.48e-4 & 2.70 & & 1.50e-4 & 2.70 & & 2.84e-4 & 2.70 & & 2.98e-4 & 2.69 \\
%      1/512  & 2.08e-5 & 2.83 & & 2.11e-5 & 2.83 & & 3.99e-5 & 2.83 & & 4.20e-5 & 2.83 \\
%      1/1024 & 2.77e-6 & 2.91 & & 2.81e-6 & 2.91 & & 5.33e-6 & 2.91 & & 5.60e-6 & 2.91 \\
%    \bottomrule[.5mm]
%  \end{tabular}
%\end{table}
\begin{table}\centering
  \caption{Estimated $L^2$ and $L^\infty$ cross-grid differences and convergence rates for the HV method~\cref{eq:num_full_p1_c264}, where $\min(P_1+2,P_2,P_3+2)=4$.}
  \label{tb:num_full_p1_2}
  \small
  \begin{tabular}{@{}llcclcclcclc@{}}
    \toprule[.5mm]
             & \multicolumn{2}{c}{Nodal solution ($L^2$)} &
             & \multicolumn{2}{c}{Cell averages ($L^2$)} &
             & \multicolumn{2}{c}{Nodal solution ($L^\infty$)} &
             & \multicolumn{2}{c}{Cell averages ($L^\infty$)} \\ \cmidrule[.2mm]{2-3} \cmidrule[.2mm]{5-6} \cmidrule[.2mm]{8-9} \cmidrule[.2mm]{11-12}
      %$h$    & Error & ~Order & $\ $ & Error & ~Order & $\ $ & Error & ~Order & $\ $ & Error & ~Order \\ \cmidrule[.3mm]{2-12}
      $h$    & Difference & ~Order & $\ $ & Difference & ~Order & $\ $ & Difference & ~Order & $\ $ & Difference & ~Order \\ \cmidrule[.3mm]{2-12}
      1/64   & 4.86e-4 &      & & 4.97e-4 &      & & 1.06e-3 &      & & 1.10e-3 &      \\
      1/128  & 3.07e-5 & 3.99 & & 3.12e-5 & 3.99 & & 6.73e-5 & 3.98 & & 6.86e-5 & 4.00 \\
      1/256  & 1.92e-6 & 4.00 & & 1.95e-6 & 4.00 & & 4.23e-6 & 3.99 & & 4.30e-6 & 4.00 \\
      1/512  & 1.20e-7 & 4.00 & & 1.22e-7 & 4.00 & & 2.65e-7 & 4.00 & & 2.69e-7 & 4.00 \\
      1/1024 & 7.51e-9 & 4.00 & & 7.63e-9 & 4.00 & & 1.65e-8 & 4.00 & & 1.68e-8 & 4.00 \\
    \bottomrule[.5mm]
  \end{tabular}
\end{table}
\begin{table}\centering
  \caption{Estimated $L^2$ and $L^\infty$ cross-grid differences and convergence rates for the HV method~\cref{eq:num_full_p1_364}, where $\min(P_1+2,P_2,P_3+2)=5$.}
  \label{tb:num_full_p1_3}
  \small
  \begin{tabular}{@{}llcclcclcclc@{}}
    \toprule[.5mm]
             & \multicolumn{2}{c}{Nodal solution ($L^2$)} &
             & \multicolumn{2}{c}{Cell averages ($L^2$)} &
             & \multicolumn{2}{c}{Nodal solution ($L^\infty$)} &
             & \multicolumn{2}{c}{Cell averages ($L^\infty$)} \\ \cmidrule[.2mm]{2-3} \cmidrule[.2mm]{5-6} \cmidrule[.2mm]{8-9} \cmidrule[.2mm]{11-12}
      %$h$    & Error & ~Order & $\ $ & Error & ~Order & $\ $ & Error & ~Order & $\ $ & Error & ~Order \\ \cmidrule[.3mm]{2-12}
      $h$    & Difference & ~Order & $\ $ & Difference & ~Order & $\ $ & Difference & ~Order & $\ $ & Difference & ~Order \\ \cmidrule[.3mm]{2-12}
      1/64   & 2.56e-5  &      & & 2.63e-5  &      & & 5.14e-5  &      & & 5.27e-5  &      \\
      1/128  & 1.01e-6  & 4.67 & & 1.03e-6  & 4.67 & & 2.02e-6  & 4.67 & & 2.10e-6  & 4.65 \\
      1/256  & 3.60e-8  & 4.81 & & 3.66e-8  & 4.81 & & 7.21e-8  & 4.81 & & 7.46e-8  & 4.81 \\
      1/512  & 1.21e-9  & 4.90 & & 1.23e-9  & 4.90 & & 2.42e-9  & 4.90 & & 2.51e-9  & 4.90 \\
      1/1024 & 3.93e-11 & 4.95 & & 3.99e-11 & 4.95 & & 7.86e-11 & 4.95 & & 8.13e-11 & 4.95 \\
    \bottomrule[.5mm]
  \end{tabular}
\end{table}

\subsubsection{Category 2: $P_2<\min(P_1+2,P_3+2)$}
\label{sec:num_full_p2}
%Three tests are considered in this category:
Two tests are considered in this category:
\begin{align}
  %\label{eq:num_full_p2_122}
  %&[\mathcal{D}_x] = [\mathcal{D}_x^{(1)}]\;, &&[\mathcal{D}_x^c] = [\mathcal{D}_x^{(c-2)}]\;, &&[\mathcal{D}_{xx}^c] = [\mathcal{D}_{xx}^{(c-2)}]\;; \\
  \label{eq:num_full_p2_c222}
  &[\mathcal{D}_x] = [\mathcal{D}_x^{(c-2)}]\;, &&[\mathcal{D}_x^c] = [\mathcal{D}_x^{(c-2)}]\;, &&[\mathcal{D}_{xx}^c] = [\mathcal{D}_{xx}^{(c-2)}]\;; \\
  \label{eq:num_full_p2_344}
  &[\mathcal{D}_x] = [\mathcal{D}_x^{(3)}]\;, &&[\mathcal{D}_x^c] = [\mathcal{D}_x^{(c-4)}]\;, &&[\mathcal{D}_{xx}^c] = [\mathcal{D}_{xx}^{(c-4)}]\;.
\end{align}
The estimated cross-grid differences and convergence rates are summarized in~\cref{tb:num_full_p2_2} and \cref{tb:num_full_p2_3}; and the convergence rates are all $P_2$, confirming~\cref{thm:acry}.
%\begin{table}\centering
%  \caption{Estimated $L^2$- and $L^\infty$-errors and convergence rates for the HV method~\cref{eq:num_full_p2_122}: $P_1=1$, $P_2=2$, $P_3=2$, thus $\min(P_1+2,P_2,P_3+2)=2$.}
%  \label{tb:num_full_p2_1}
%  \small
%  \begin{tabular}{@{}llcclcclcclc@{}}
%    \toprule[.5mm]
%             & \multicolumn{2}{c}{Nodal solution ($L^2$)} &
%             & \multicolumn{2}{c}{Cell averages ($L^2$)} &
%             & \multicolumn{2}{c}{Nodal solution ($L^\infty$)} &
%             & \multicolumn{2}{c}{Cell averages ($L^\infty$)} \\ \cmidrule[.2mm]{2-3} \cmidrule[.2mm]{5-6} \cmidrule[.2mm]{8-9} \cmidrule[.2mm]{11-12}
%      $h$    & Error & ~Order & $\ $ & Error & ~Order & $\ $ & Error & ~Order & $\ $ & Error & ~Order \\ \cmidrule[.3mm]{2-12}
%      1/64   & 5.54e-3 &      & & 5.48e-3 &      & & 1.10e-2 &      & & 1.06e-2 &      \\
%      1/128  & 9.84e-4 & 2.49 & & 9.71e-4 & 2.50 & & 1.95e-3 & 2.49 & & 1.88e-3 & 2.49 \\
%      1/256  & 1.56e-4 & 2.66 & & 1.52e-4 & 2.67 & & 3.13e-4 & 2.64 & & 2.95e-4 & 2.67 \\
%      1/512  & 2.44e-5 & 2.68 & & 2.34e-5 & 2.70 & & 5.03e-5 & 2.64 & & 4.72e-5 & 2.65 \\
%      1/1024 & 4.22e-6 & 2.53 & & 4.03e-6 & 2.54 & & 9.10e-6 & 2.47 & & 8.59e-6 & 2.46 \\
%    \bottomrule[.5mm]
%  \end{tabular}
%\end{table}
\begin{table}\centering
  \caption{Estimated $L^2$ and $L^\infty$ cross-grid differences and convergence rates for the HV method~\cref{eq:num_full_p2_c222}, where $\min(P_1+2,P_2,P_3+2)=2$.}
  \label{tb:num_full_p2_2}
  \small
  \begin{tabular}{@{}llcclcclcclc@{}}
    \toprule[.5mm]
             & \multicolumn{2}{c}{Nodal solution ($L^2$)} &
             & \multicolumn{2}{c}{Cell averages ($L^2$)} &
             & \multicolumn{2}{c}{Nodal solution ($L^\infty$)} &
             & \multicolumn{2}{c}{Cell averages ($L^\infty$)} \\ \cmidrule[.2mm]{2-3} \cmidrule[.2mm]{5-6} \cmidrule[.2mm]{8-9} \cmidrule[.2mm]{11-12}
      %$h$    & Error & ~Order & $\ $ & Error & ~Order & $\ $ & Error & ~Order & $\ $ & Error & ~Order \\ \cmidrule[.3mm]{2-12}
      $h$    & Difference & ~Order & $\ $ & Difference & ~Order & $\ $ & Difference & ~Order & $\ $ & Difference & ~Order \\ \cmidrule[.3mm]{2-12}
      1/64   & 1.93e-3 &      & & 1.93e-3 &      & & 4.32e-3 &      & & 4.35e-3 &      \\
      1/128  & 2.79e-4 & 2.79 & & 2.78e-4 & 2.79 & & 6.31e-4 & 2.77 & & 6.30e-4 & 2.79 \\
      1/256  & 5.56e-5 & 2.32 & & 5.56e-5 & 2.32 & & 1.26e-4 & 2.32 & & 1.27e-4 & 2.32 \\
      1/512  & 1.30e-5 & 2.09 & & 1.30e-5 & 2.09 & & 2.97e-5 & 2.09 & & 2.97e-5 & 2.09 \\
      1/1024 & 3.20e-6 & 2.02 & & 3.20e-6 & 2.02 & & 7.29e-6 & 2.02 & & 7.29e-6 & 2.02 \\
    \bottomrule[.5mm]
  \end{tabular}
\end{table}
\begin{table}\centering
  \caption{Estimated $L^2$ and $L^\infty$ cross-grid differences and convergence rates for the HV method~\cref{eq:num_full_p2_344}, where $\min(P_1+2,P_2,P_3+2)=4$.}
  \label{tb:num_full_p2_3}
  \small
  \begin{tabular}{@{}llcclcclcclc@{}}
    \toprule[.5mm]
             & \multicolumn{2}{c}{Nodal solution ($L^2$)} &
             & \multicolumn{2}{c}{Cell averages ($L^2$)} &
             & \multicolumn{2}{c}{Nodal solution ($L^\infty$)} &
             & \multicolumn{2}{c}{Cell averages ($L^\infty$)} \\ \cmidrule[.2mm]{2-3} \cmidrule[.2mm]{5-6} \cmidrule[.2mm]{8-9} \cmidrule[.2mm]{11-12}
      %$h$    & Error & ~Order & $\ $ & Error & ~Order & $\ $ & Error & ~Order & $\ $ & Error & ~Order \\ \cmidrule[.3mm]{2-12}
      $h$    & Difference & ~Order & $\ $ & Difference & ~Order & $\ $ & Difference & ~Order & $\ $ & Difference & ~Order \\ \cmidrule[.3mm]{2-12}
      1/64   & 2.51e-5  &      & & 2.54e-5  &      & & 4.98e-5  &      & & 5.15e-5  &      \\
      1/128  & 1.00e-6  & 4.65 & & 9.99e-7  & 4.67 & & 2.01e-6  & 4.63 & & 2.03e-6  & 4.67 \\
      1/256  & 3.77e-8  & 4.73 & & 3.70e-8  & 4.75 & & 7.60e-8  & 4.72 & & 7.60e-8  & 4.74 \\
      1/512  & 1.52e-9  & 4.63 & & 1.48e-9  & 4.65 & & 3.15e-9  & 4.59 & & 3.01e-9  & 4.66 \\
      1/1024 & 7.52e-11 & 4.34 & & 7.30e-11 & 4.34 & & 1.62e-10 & 4.28 & & 1.56e-10 & 4.27 \\
    \bottomrule[.5mm]
  \end{tabular}
\end{table}

\subsubsection{Category 3: $P_3+2<\min(P_1+2,P_2)$}
\label{sec:num_full_p3}
Two tests are considered in this category:
\begin{align}
  \label{eq:num_full_p3_362}
  &[\mathcal{D}_x] = [\mathcal{D}_x^{(3)}]\;, &&[\mathcal{D}_x^c] = [\mathcal{D}_x^{(c-6)}]\;, &&[\mathcal{D}_{xx}^c] = [\mathcal{D}_{xx}^{(c-2)}]\;; \\
  \label{eq:num_full_p3_c462}
  &[\mathcal{D}_x] = [\mathcal{D}_x^{(c-4)}]\;, &&[\mathcal{D}_x^c] = [\mathcal{D}_x^{(c-6)}]\;, &&[\mathcal{D}_{xx}^c] = [\mathcal{D}_{xx}^{(c-2)}]\;.
\end{align}
The estimated cross-grid differences and convergence rates are provided in~\cref{tb:num_full_p3_1} and~\cref{tb:num_full_p3_2}; and all convergence rates predicted by~\cref{thm:acry} are confirmed to be $P_3+2$.
\begin{table}\centering
  \caption{Estimated $L^2$ and $L^\infty$ cross-grid differences and convergence rates for the HV method~\cref{eq:num_full_p3_362}, where $\min(P_1+2,P_2,P_3+2)=4$.}
  \label{tb:num_full_p3_1}
  \small
  \begin{tabular}{@{}llcclcclcclc@{}}
    \toprule[.5mm]
             & \multicolumn{2}{c}{Nodal solution ($L^2$)} &
             & \multicolumn{2}{c}{Cell averages ($L^2$)} &
             & \multicolumn{2}{c}{Nodal solution ($L^\infty$)} &
             & \multicolumn{2}{c}{Cell averages ($L^\infty$)} \\ \cmidrule[.2mm]{2-3} \cmidrule[.2mm]{5-6} \cmidrule[.2mm]{8-9} \cmidrule[.2mm]{11-12}
      %$h$    & Error & ~Order & $\ $ & Error & ~Order & $\ $ & Error & ~Order & $\ $ & Error & ~Order \\ \cmidrule[.3mm]{2-12}
      $h$    & Difference & ~Order & $\ $ & Difference & ~Order & $\ $ & Difference & ~Order & $\ $ & Difference & ~Order \\ \cmidrule[.3mm]{2-12}
      1/64   & 4.60e-6 &      & & 4.72e-6 &      & & 9.47e-6 &      & & 9.41e-6 &      \\
      1/128  & 1.55e-6 & 1.57 & & 1.58e-6 & 1.58 & & 3.10e-6 & 1.61 & & 3.19e-6 & 1.56 \\
      1/256  & 2.01e-7 & 2.95 & & 2.04e-7 & 2.95 & & 4.02e-7 & 2.95 & & 4.16e-7 & 2.94 \\
      1/512  & 1.74e-8 & 3.53 & & 1.77e-8 & 3.53 & & 3.49e-8 & 3.53 & & 3.61e-8 & 3.53 \\
      1/1024 & 1.28e-9 & 3.76 & & 1.30e-9 & 3.76 & & 2.57e-9 & 3.76 & & 2.66e-9 & 3.76 \\
    \bottomrule[.5mm]
  \end{tabular}
\end{table}
\begin{table}\centering
  \caption{Estimated $L^2$ and $L^\infty$ cross-grid differences and convergence rates for the HV method~\cref{eq:num_full_p3_c462}, where $\min(P_1+2,P_2,P_3+2)=4$.}
  \label{tb:num_full_p3_2}
  \small
  \begin{tabular}{@{}llcclcclcclc@{}}
    \toprule[.5mm]
             & \multicolumn{2}{c}{Nodal solution ($L^2$)} &
             & \multicolumn{2}{c}{Cell averages ($L^2$)} &
             & \multicolumn{2}{c}{Nodal solution ($L^\infty$)} &
             & \multicolumn{2}{c}{Cell averages ($L^\infty$)} \\ \cmidrule[.2mm]{2-3} \cmidrule[.2mm]{5-6} \cmidrule[.2mm]{8-9} \cmidrule[.2mm]{11-12}
      %$h$    & Error & ~Order & $\ $ & Error & ~Order & $\ $ & Error & ~Order & $\ $ & Error & ~Order \\ \cmidrule[.3mm]{2-12}
      $h$    & Difference & ~Order & $\ $ & Difference & ~Order & $\ $ & Difference & ~Order & $\ $ & Difference & ~Order \\ \cmidrule[.3mm]{2-12}
      1/64   & 7.92e-5 &      & & 8.11e-5 &      & & 1.59e-4 &      & & 1.66e-4 &      \\
      1/128  & 6.05e-6 & 3.71 & & 6.15e-6 & 3.72 & & 1.21e-5 & 3.71 & & 1.25e-5 & 3.73 \\
      1/256  & 3.87e-7 & 3.97 & & 3.93e-7 & 3.97 & & 7.74e-7 & 3.97 & & 8.01e-7 & 3.97 \\
      1/512  & 2.43e-8 & 4.00 & & 2.46e-8 & 4.00 & & 4.86e-8 & 4.00 & & 5.02e-8 & 4.00 \\
      1/1024 & 1.52e-9 & 4.00 & & 1.54e-9 & 4.00 & & 3.04e-9 & 4.00 & & 3.14e-9 & 4.00 \\
    \bottomrule[.5mm]
  \end{tabular}
\end{table}

\subsubsection{Category 4: $P_1+2=P_2=P_3+2$}
\label{sec:num_full_opt}
%Three tests are considered in this category:
Lastly, we consider two tests that have optimal order of accuracy:
\begin{align}
  %\label{eq:num_full_opt_242}
  %&[\mathcal{D}_x] = [\mathcal{D}_x^{(2)}]\;, &&[\mathcal{D}_x^c] = [\mathcal{D}_x^{(c-4)}]\;, &&[\mathcal{D}_{xx}^c] = [\mathcal{D}_{xx}^{(c-2)}]\;; \\
  \label{eq:num_full_opt_c242}
  &[\mathcal{D}_x] = [\mathcal{D}_x^{(c-2)}]\;, &&[\mathcal{D}_x^c] = [\mathcal{D}_x^{(c-4)}]\;, &&[\mathcal{D}_{xx}^c] = [\mathcal{D}_{xx}^{(c-2)}]\;; \\
  \label{eq:num_full_opt_464}
  &[\mathcal{D}_x] = [\mathcal{D}_x^{(4)}]\;, &&[\mathcal{D}_x^c] = [\mathcal{D}_x^{(c-6)}]\;, &&[\mathcal{D}_{xx}^c] = [\mathcal{D}_{xx}^{(c-4)}]\;.
\end{align}
The estimated cross-grid differences and convergence rates are summarized in~\cref{tb:num_full_opt_2} and \cref{tb:num_full_opt_3}; and the convergence rates are all $P_1+2=P_2=P_3+2$, which confirm~\cref{thm:acry}.
%\begin{table}\centering
%  \caption{Estimated $L^2$- and $L^\infty$-errors and convergence rates for the HV method~\cref{eq:num_full_opt_242}: $P_1=1$, $P_2=2$, $P_3=2$, thus $\min(P_1+2,P_2,P_3+2)=2$.}
%  \label{tb:num_full_opt_1}
%  \small
%  \begin{tabular}{@{}llcclcclcclc@{}}
%    \toprule[.5mm]
%             & \multicolumn{2}{c}{Nodal solution ($L^2$)} &
%             & \multicolumn{2}{c}{Cell averages ($L^2$)} &
%             & \multicolumn{2}{c}{Nodal solution ($L^\infty$)} &
%             & \multicolumn{2}{c}{Cell averages ($L^\infty$)} \\ \cmidrule[.2mm]{2-3} \cmidrule[.2mm]{5-6} \cmidrule[.2mm]{8-9} \cmidrule[.2mm]{11-12}
%      $h$    & Error & ~Order & $\ $ & Error & ~Order & $\ $ & Error & ~Order & $\ $ & Error & ~Order \\ \cmidrule[.3mm]{2-12}
%      1/64   & 2.71e-4 &      & & 2.77e-4 &      & & 5.93e-4 &      & & 6.02e-4 &      \\
%      1/128  & 2.84e-5 & 3.25 & & 2.89e-5 & 3.26 & & 6.22e-5 & 3.25 & & 6.33e-5 & 3.25 \\
%      1/256  & 2.61e-6 & 3.45 & & 2.65e-6 & 3.45 & & 5.74e-6 & 3.44 & & 5.81e-6 & 3.44 \\
%      1/512  & 2.11e-7 & 3.63 & & 2.14e-7 & 3.63 & & 4.66e-7 & 3.62 & & 4.71e-7 & 3.62 \\
%      1/1024 & 1.55e-8 & 3.77 & & 1.57e-8 & 3.77 & & 3.41e-8 & 3.77 & & 3.45e-8 & 3.77 \\
%    \bottomrule[.5mm]
%  \end{tabular}
%\end{table}
\begin{table}\centering
  \caption{Estimated $L^2$ and $L^\infty$ cross-grid differences and convergence rates for the HV method~\cref{eq:num_full_opt_c242}, where $\min(P_1+2,P_2,P_3+2)=4$.}
  \label{tb:num_full_opt_2}
  \small
  \begin{tabular}{@{}llcclcclcclc@{}}
    \toprule[.5mm]
             & \multicolumn{2}{c}{Nodal solution ($L^2$)} &
             & \multicolumn{2}{c}{Cell averages ($L^2$)} &
             & \multicolumn{2}{c}{Nodal solution ($L^\infty$)} &
             & \multicolumn{2}{c}{Cell averages ($L^\infty$)} \\ \cmidrule[.2mm]{2-3} \cmidrule[.2mm]{5-6} \cmidrule[.2mm]{8-9} \cmidrule[.2mm]{11-12}
      %$h$    & Error & ~Order & $\ $ & Error & ~Order & $\ $ & Error & ~Order & $\ $ & Error & ~Order \\ \cmidrule[.3mm]{2-12}
      $h$    & Difference & ~Order & $\ $ & Difference & ~Order & $\ $ & Difference & ~Order & $\ $ & Difference & ~Order \\ \cmidrule[.3mm]{2-12}
      1/64   & 1.12e-3 &      & & 1.14e-3 &      & & 2.43e-3 &      & & 2.48e-3 &      \\
      1/128  & 7.60e-5 & 3.88 & & 7.71e-5 & 3.89 & & 1.67e-4 & 3.87 & & 1.70e-4 & 3.87 \\
      1/256  & 4.79e-6 & 3.99 & & 4.86e-6 & 3.99 & & 1.06e-5 & 3.98 & & 1.07e-5 & 3.99 \\
      1/512  & 3.00e-7 & 4.00 & & 3.04e-7 & 4.00 & & 6.62e-7 & 4.00 & & 6.70e-7 & 4.00 \\
      1/1024 & 1.88e-8 & 4.00 & & 1.90e-8 & 4.00 & & 4.14e-8 & 4.00 & & 4.19e-8 & 4.00 \\
    \bottomrule[.5mm]
  \end{tabular}
\end{table}
\begin{table}\centering
  \caption{Estimated $L^2$ and $L^\infty$ cross-grid differences and convergence rates for the HV method~\cref{eq:num_full_opt_464}, where $\min(P_1+2,P_2,P_3+2)=6$.}
  \label{tb:num_full_opt_3}
  \small
  \begin{tabular}{@{}llcclcclcclc@{}}
    \toprule[.5mm]
             & \multicolumn{2}{c}{Nodal solution ($L^2$)} &
             & \multicolumn{2}{c}{Cell averages ($L^2$)} &
             & \multicolumn{2}{c}{Nodal solution ($L^\infty$)} &
             & \multicolumn{2}{c}{Cell averages ($L^\infty$)} \\ \cmidrule[.2mm]{2-3} \cmidrule[.2mm]{5-6} \cmidrule[.2mm]{8-9} \cmidrule[.2mm]{11-12}
      %$h$    & Error & ~Order & $\ $ & Error & ~Order & $\ $ & Error & ~Order & $\ $ & Error & ~Order \\ \cmidrule[.3mm]{2-12}
      $h$    & Difference & ~Order & $\ $ & Difference & ~Order & $\ $ & Difference & ~Order & $\ $ & Difference & ~Order \\ \cmidrule[.3mm]{2-12}
      1/64   & 3.10e-6  &      & & 3.18e-6  &      & & 6.90e-6  &      & & 6.91e-6  &      \\
      1/128  & 6.67e-8  & 5.54 & & 6.79e-8  & 5.55 & & 1.48e-7  & 5.54 & & 1.51e-7  & 5.51 \\
      1/256  & 1.27e-9  & 5.72 & & 1.29e-9  & 5.72 & & 2.83e-9  & 5.71 & & 2.87e-9  & 5.72 \\
      1/512  & 2.22e-11 & 5.84 & & 2.25e-11 & 5.84 & & 4.95e-11 & 5.84 & & 5.02e-11 & 5.84 \\
      1/1024 & 3.68e-13 & 5.91 & & 3.74e-13 & 5.91 & & 8.24e-13 & 5.91 & & 8.35e-13 & 5.91 \\
    \bottomrule[.5mm]
  \end{tabular}
\end{table}

\subsection{Stability plots of central HV methods}
\label{sec:num_cent}
%In the second set of numerical examples,
Next we demonstrate the stability of selected central HV schemes to solve linear diffusion equations and linear advection-diffusion equations by sketching their eigenvalue trajectories $\mathcal{S}(\Pe)$.

In the first set of plots (\cref{fg:num_cent_stab_bal}), we consider operators $[\mathcal{D}_x]$ with stencil $l'=r'=l=r$, denoted $[\mathcal{D}_x]=[\mathcal{D}^{(c-4r)}_x]$ following the convention of previous section, and $[\mathcal{D}_x^c]$ with stencil $p'=p$ and denoted $[\mathcal{D}^{(c-4p)}_x]$, and $[\mathcal{D}_{xx}^c]$ with stencil $q'=q$ and denoted $[\mathcal{D}_{xx}^{(c-4q)}]$.
We call these ``balanced'' HV schemes as all operators use the same number of cells and nodes on either side.
The details of the stability proof of the central HV method constructed using these operators are given in~\cref{sec:stab_ade} and~\cref{app:ade}.
For each combination that we picked, the trajectory $\mathcal{S}(\Pe)$ with $\Pe=0$ (diffusion-only case), $\Pe=1$, $\Pe=5$, and $\Pe=20$ are plotted.
\begin{figure}\centering
  %\includegraphics[width=.48\textwidth]{./figs/num_cent_2_2_2.pdf} ~~
  %\includegraphics[width=.48\textwidth]{./figs/num_cent_4_6_4.pdf} \\
  %\hspace{.4cm} $[\mathcal{D}_x^{(c-4)}], [\mathcal{D}_x^{(c-4)}], [\mathcal{D}_{xx}^{(c-4)}]$ \hspace{2.4cm}
  %$[\mathcal{D}_x^{(c-8)}], [\mathcal{D}_x^{(c-12)}], [\mathcal{D}_{xx}^{(c-8)}]$ \\
  %\includegraphics[width=.48\textwidth]{./figs/num_cent_18_20_18.pdf} ~~
  %\includegraphics[width=.48\textwidth]{./figs/num_cent_18_6_4.pdf} \\
  %\hspace{.4cm} $[\mathcal{D}_x^{(c-36)}], [\mathcal{D}_x^{(c-40)}], [\mathcal{D}_{xx}^{(c-36)}]$ \hspace{2.4cm}
  %$[\mathcal{D}_x^{(c-36)}], [\mathcal{D}_x^{(c-12)}], [\mathcal{D}_{xx}^{(c-8)}]$ \\
  %\includegraphics[width=.48\textwidth]{./figs/num_cent_4_20_4.pdf} ~~
  %\includegraphics[width=.48\textwidth]{./figs/num_cent_4_6_18.pdf} \\
  %\hspace{.4cm} $[\mathcal{D}_x^{(c-8)}], [\mathcal{D}_x^{(c-40)}], [\mathcal{D}_{xx}^{(c-8)}]$ \hspace{2.4cm}
  %$[\mathcal{D}_x^{(c-8)}], [\mathcal{D}_x^{(c-12)}], [\mathcal{D}_{xx}^{(c-36)}]$ 
  \includegraphics[width=.48\textwidth]{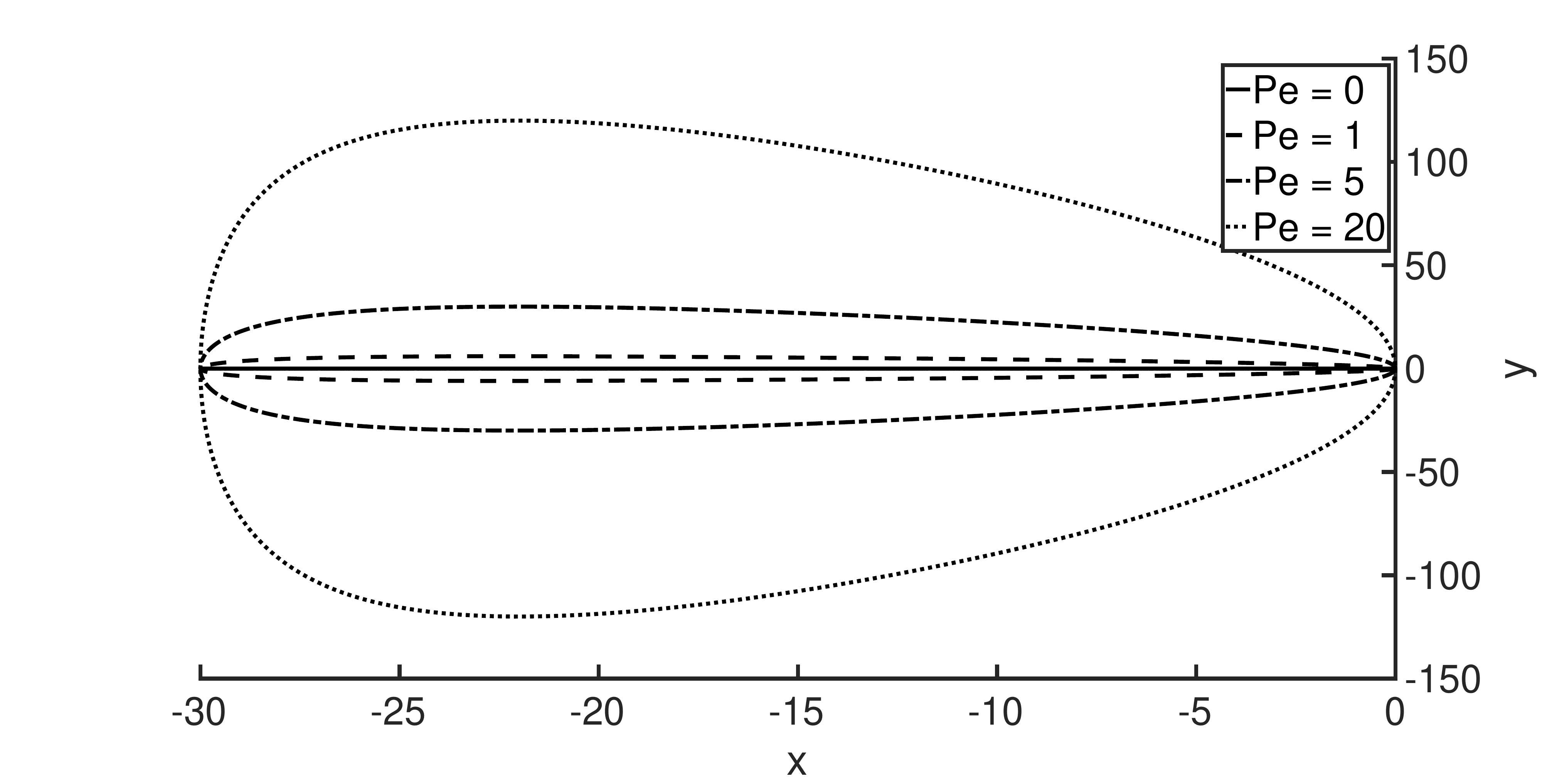} ~~
  \includegraphics[width=.48\textwidth]{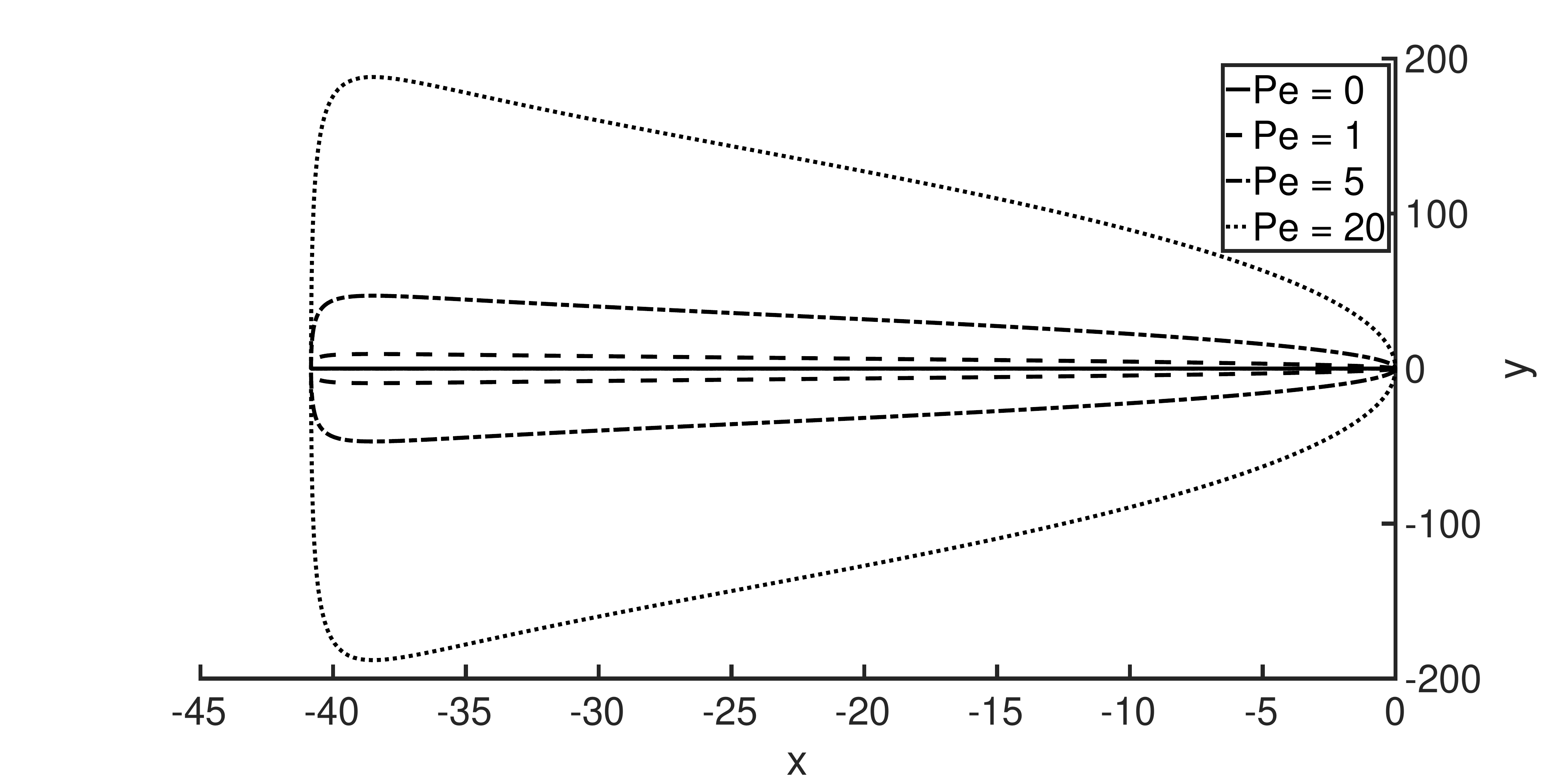} \\
  \hspace{.4cm} $[\mathcal{D}_x^{(c-4)}], [\mathcal{D}_x^{(c-4)}], [\mathcal{D}_{xx}^{(c-4)}]$ \hspace{2.4cm}
  $[\mathcal{D}_x^{(c-36)}], [\mathcal{D}_x^{(c-12)}], [\mathcal{D}_{xx}^{(c-8)}]$ \\
  \includegraphics[width=.48\textwidth]{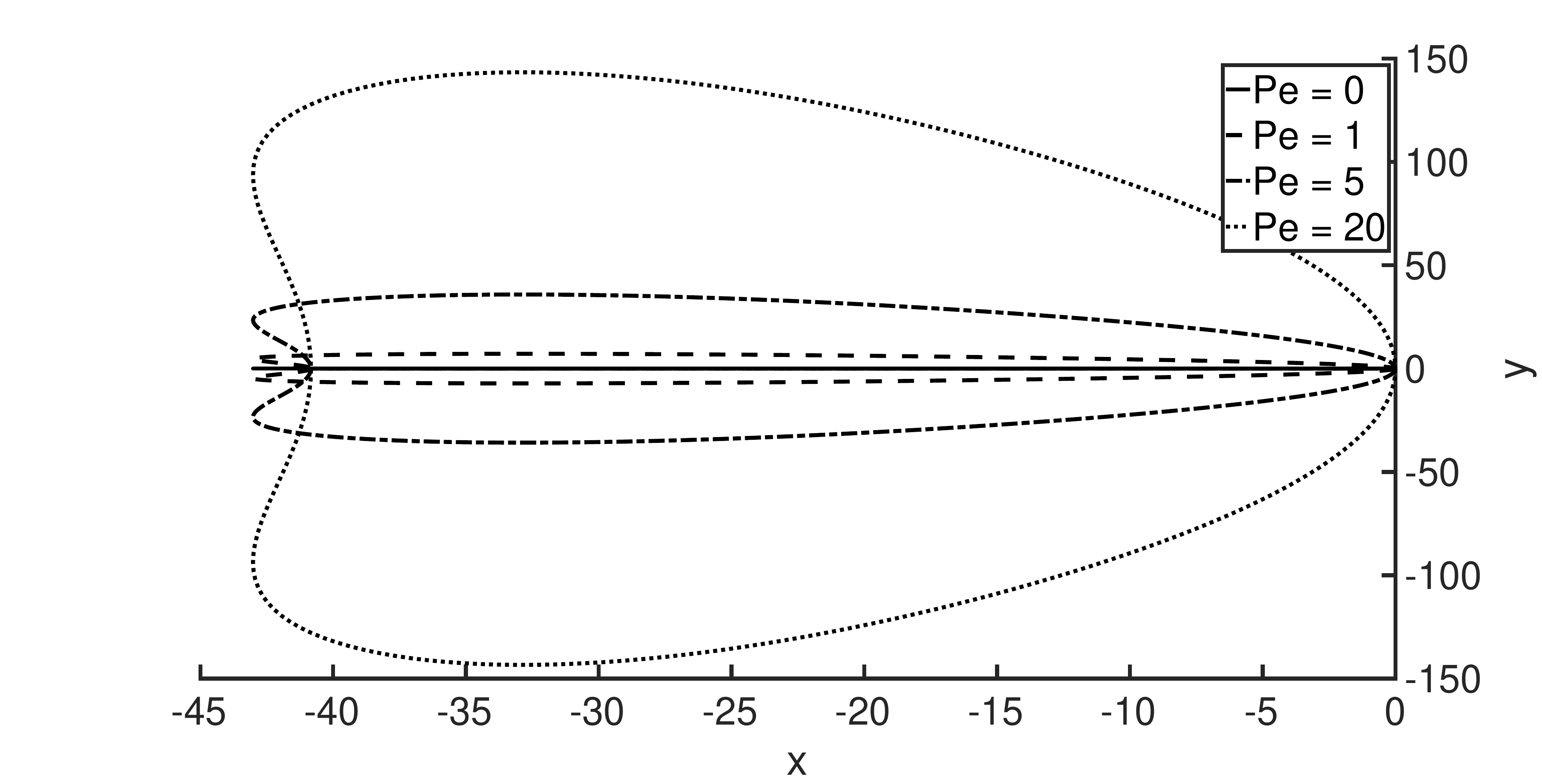} ~~
  \includegraphics[width=.48\textwidth]{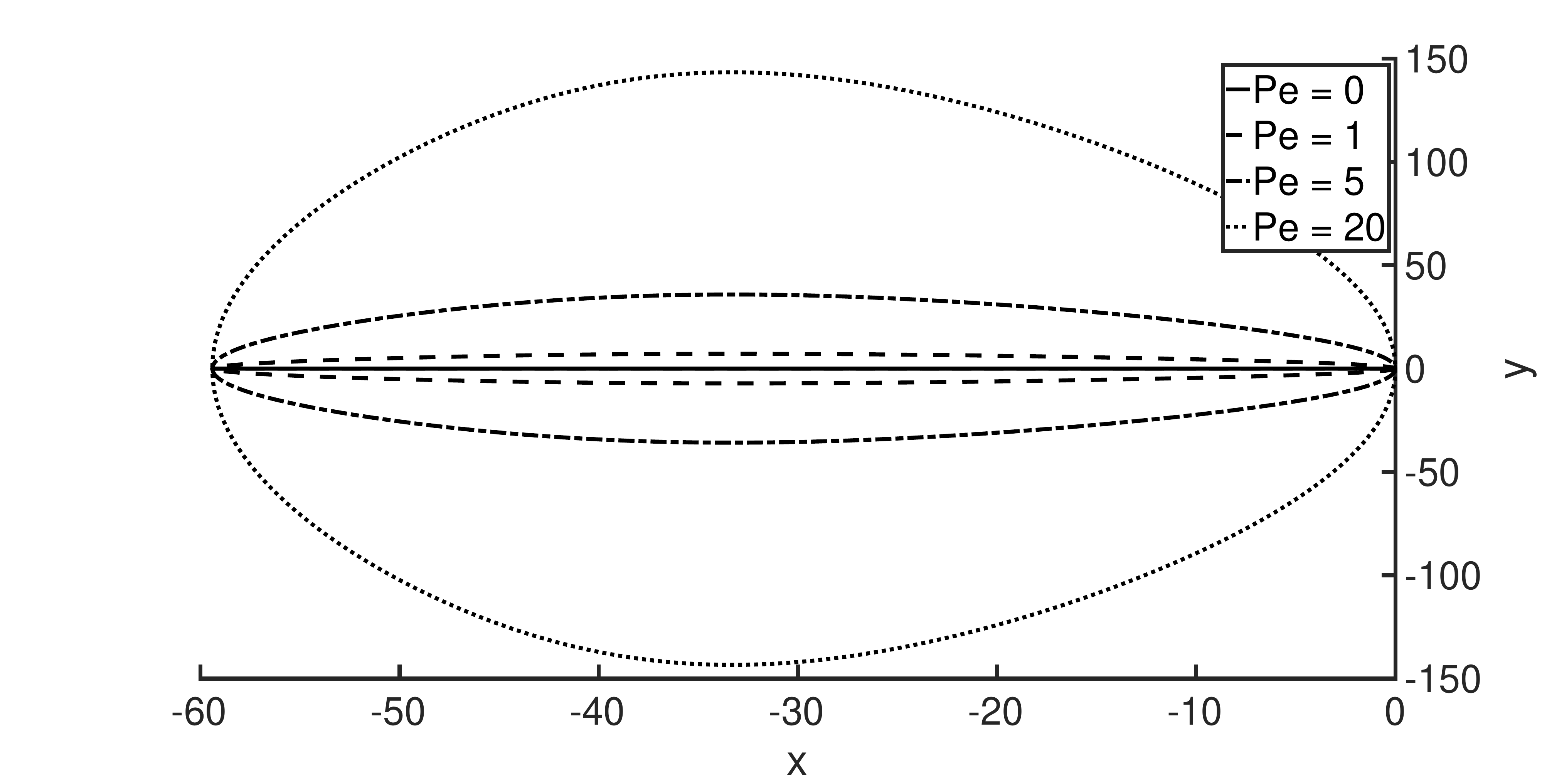} \\
  \hspace{.4cm} $[\mathcal{D}_x^{(c-8)}], [\mathcal{D}_x^{(c-40)}], [\mathcal{D}_{xx}^{(c-8)}]$ \hspace{2.4cm}
  $[\mathcal{D}_x^{(c-8)}], [\mathcal{D}_x^{(c-12)}], [\mathcal{D}_{xx}^{(c-36)}]$ 
  \caption{Trajectories $\mathcal{S}(\Pe)$ corresponding to selected balanced central HV schemes for $\Pe=0, 1, 5, 20$.
    Each central HV method is designated by the three operators in the order $[\mathcal{D}_x]$, $[\mathcal{D}_x^c]$, $[\mathcal{D}_{xx}^c]$.}
  \label{fg:num_cent_stab_bal}
\end{figure}
We observe from the plots that the entire trajectories all lie in the left complex plane, confirming the stability that we proved.

In the second set of plots we blend in some ``unblanced'' operators such as $[\mathcal{D}_x]=[\mathcal{D}_x^{(c-10)}]$, which is defined by $l=r=3$ and $l'=r'=2$.
Although the details of the stability proof of such central HV schemes were not provided (due to their similarity with the case of balanced HV methods), we see that their $\mathcal{S}(\Pe)$ again lies to the left of the imaginary axis and thus the semi-discretized method is stable.
\begin{figure}\centering
  %\includegraphics[width=.48\textwidth]{./figs/num_cent_2_3_2.pdf} ~~
  %\includegraphics[width=.48\textwidth]{./figs/num_cent_5_7_5.pdf} \\
  %\hspace{.4cm} $[\mathcal{D}_x^{(c-4)}], [\mathcal{D}_x^{(c-6)}], [\mathcal{D}_{xx}^{(c-4)}]$ \hspace{2.4cm}
  %$[\mathcal{D}_x^{(c-10)}], [\mathcal{D}_x^{(c-14)}], [\mathcal{D}_{xx}^{(c-10)}]$ \\
  %\includegraphics[width=.48\textwidth]{./figs/num_cent_17_19_17.pdf} ~~
  %\includegraphics[width=.48\textwidth]{./figs/num_cent_18_7_5.pdf} \\
  %\hspace{.4cm} $[\mathcal{D}_x^{(c-34)}], [\mathcal{D}_x^{(c-38)}], [\mathcal{D}_{xx}^{(c-34)}]$ \hspace{2.4cm}
  %$[\mathcal{D}_x^{(c-36)}], [\mathcal{D}_x^{(c-14)}], [\mathcal{D}_{xx}^{(c-10)}]$ \\
  %\includegraphics[width=.48\textwidth]{./figs/num_cent_5_20_5.pdf} ~~
  %\includegraphics[width=.48\textwidth]{./figs/num_cent_5_7_18.pdf} \\
  %\hspace{.4cm} $[\mathcal{D}_x^{(c-10)}], [\mathcal{D}_x^{(c-40)}], [\mathcal{D}_{xx}^{(c-10)}]$ \hspace{2.4cm}
  %$[\mathcal{D}_x^{(c-10)}], [\mathcal{D}_x^{(c-14)}], [\mathcal{D}_{xx}^{(c-36)}]$ 
  \includegraphics[width=.48\textwidth]{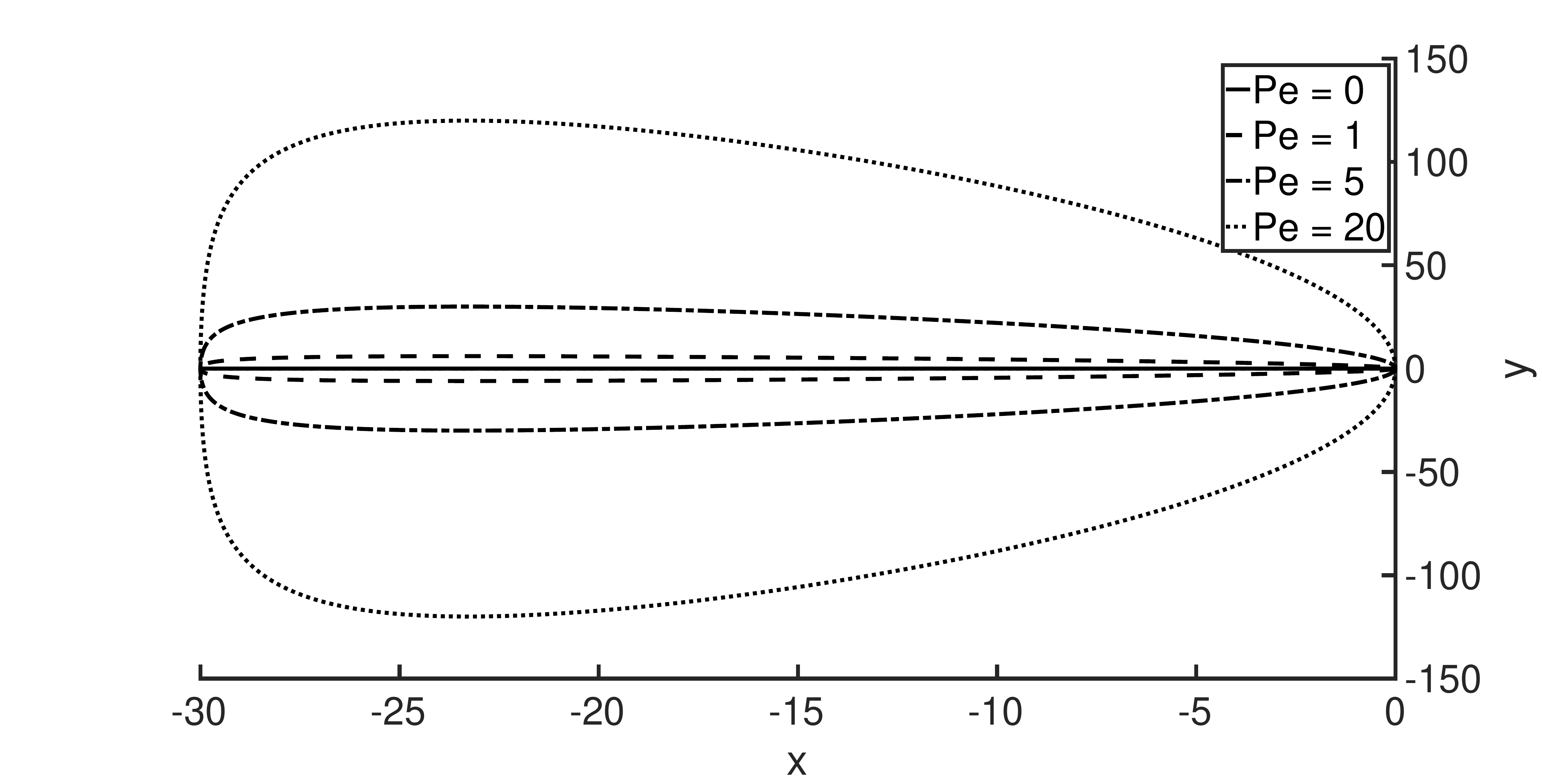} ~~
  \includegraphics[width=.48\textwidth]{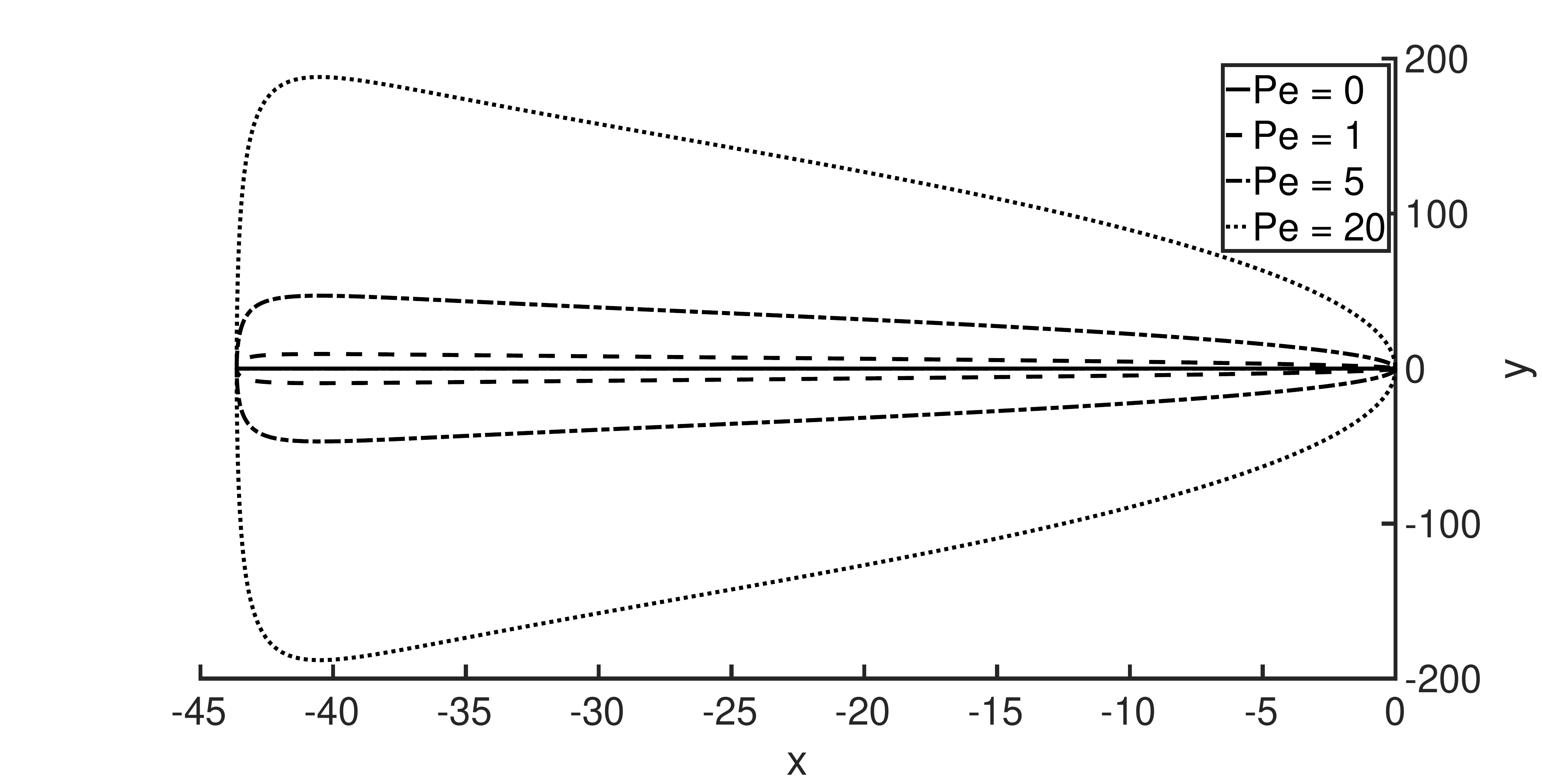} \\
  \hspace{.4cm} $[\mathcal{D}_x^{(c-4)}], [\mathcal{D}_x^{(c-6)}], [\mathcal{D}_{xx}^{(c-4)}]$ \hspace{2.4cm}
  $[\mathcal{D}_x^{(c-36)}], [\mathcal{D}_x^{(c-14)}], [\mathcal{D}_{xx}^{(c-10)}]$ \\
  \includegraphics[width=.48\textwidth]{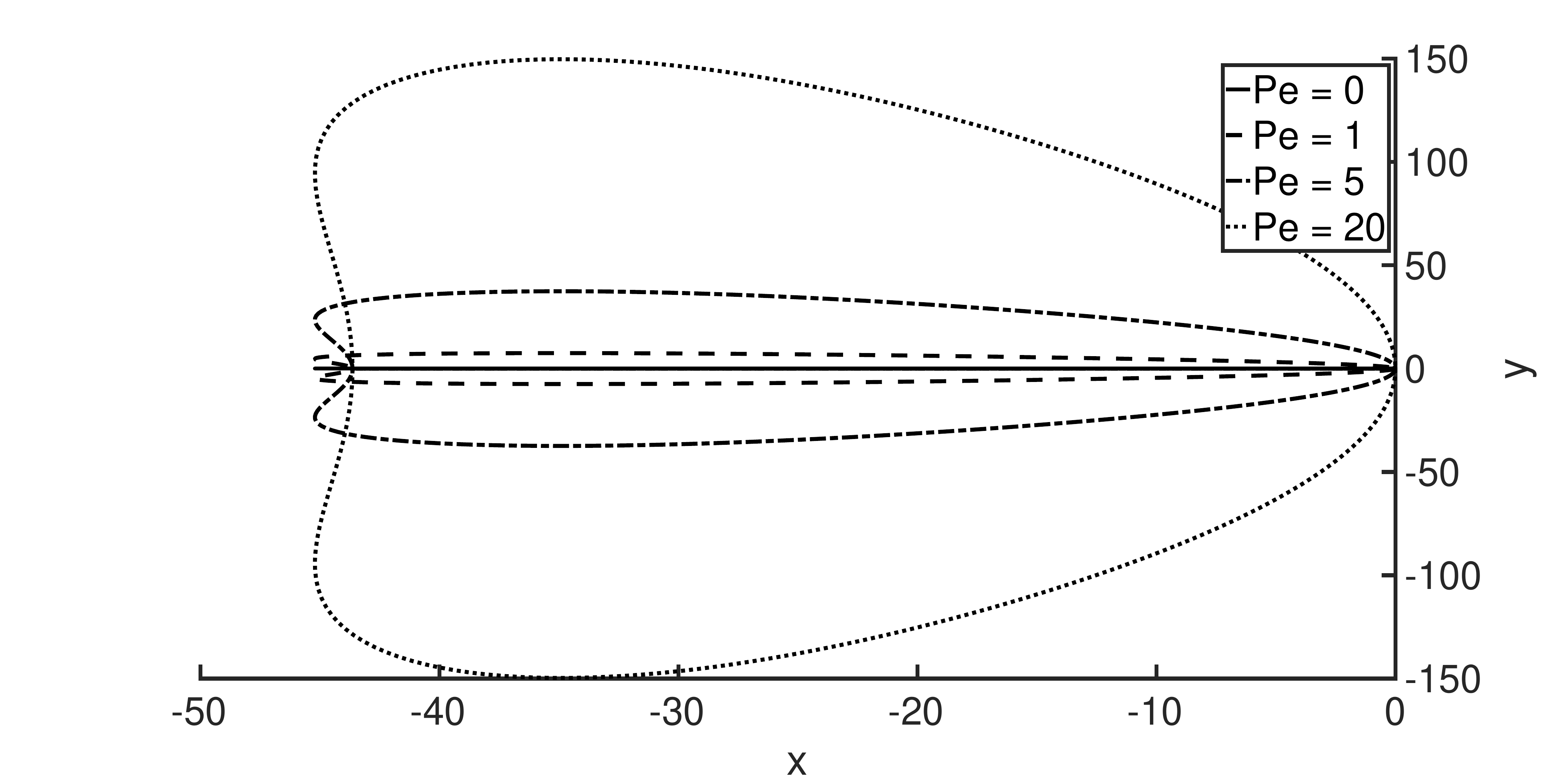} ~~
  \includegraphics[width=.48\textwidth]{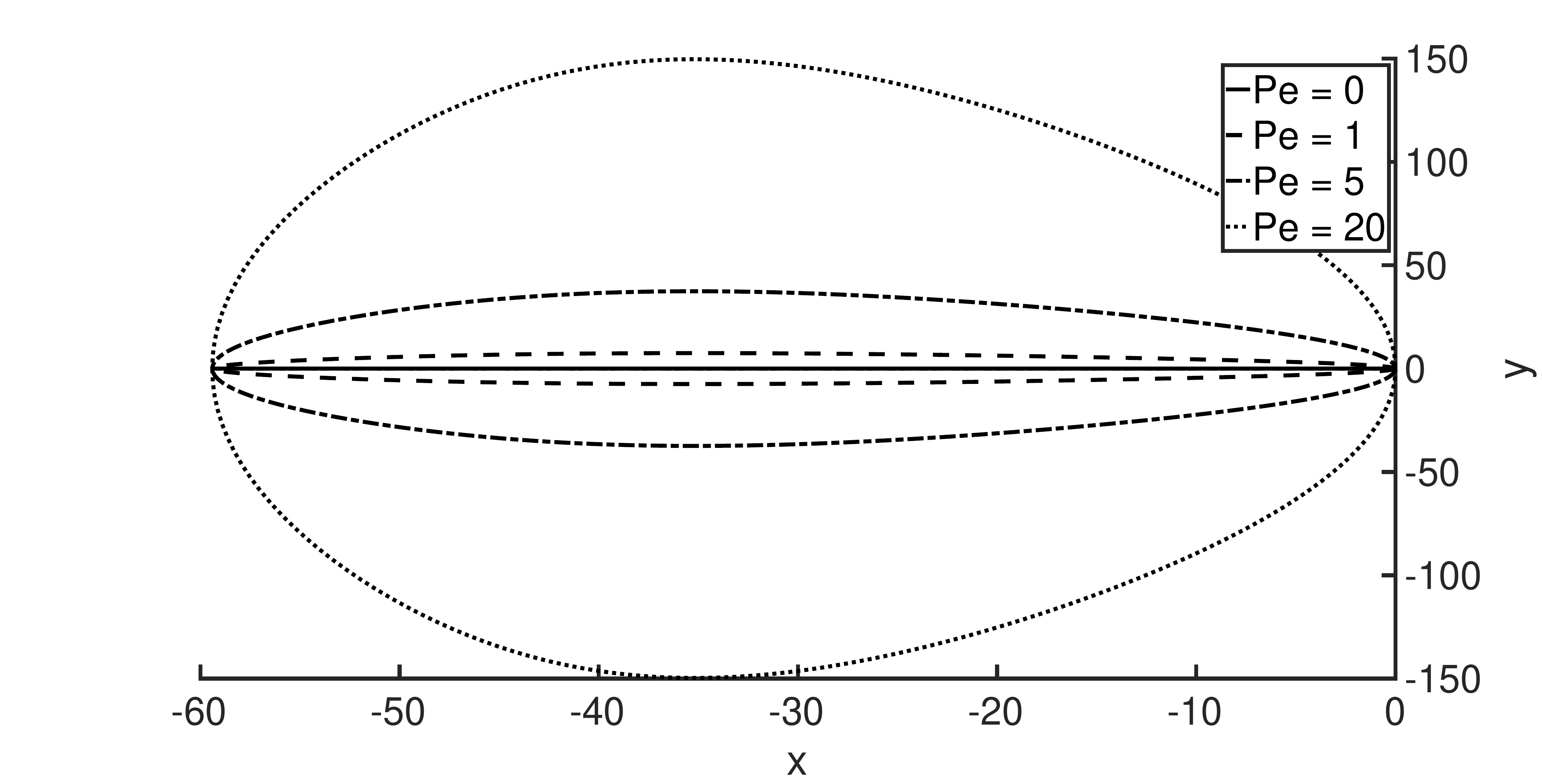} \\
  \hspace{.4cm} $[\mathcal{D}_x^{(c-10)}], [\mathcal{D}_x^{(c-40)}], [\mathcal{D}_{xx}^{(c-10)}]$ \hspace{2.4cm}
  $[\mathcal{D}_x^{(c-10)}], [\mathcal{D}_x^{(c-14)}], [\mathcal{D}_{xx}^{(c-36)}]$ 
  \caption{Trajectories $\mathcal{S}(\Pe)$ corresponding to selected unbalanced central HV schemes for $\Pe=0, 1, 5, 20$.
    Each central HV method is designated by the three operators in the order $[\mathcal{D}_x]$, $[\mathcal{D}_x^c]$, $[\mathcal{D}_{xx}^c]$.}
  \label{fg:num_cent_stab_unbal}
\end{figure}

\section{Conclusions}
\label{sec:concl}
We developed an accuracy and stability theory for a class of Hermite-type discretizations of one-dimensional linear advection-diffusion equations.
These methods evolve both nodal values and cell averages simultaneously and construct the required nodal spatial derivatives using a linear combination of both types of variables in a continuous stencil; hence they are also called the hybrid-variable (HV) methods and approximations to spatial derivatives are called hybrid-variable discrete differential operators (HV-DDO).

The accuracy analysis shows that the formal order of accuracy of individual HV-DDO does not translate directly into the accuracy of the method itself.
In particular, if the local truncation error of the HV-DDO for the advection term is $P_1$, that of the central HV-DDO for the advective flux is $P_2$, and the central HV-DDO for the diffusion term has order $P_3$, then the spatial order of the advection-diffusion discretization is $\min(P_1+2,P_2,P_3+2)$.
Furthermore, we proved that if central HV-DDOs are used to discretize both the advection term and the diffusion term, the resulting method is always stable at the semi-discretized level.
\bibliographystyle{plain}      % mathematics and physical sciences
\bibliography{pap}

% Non-BibTeX users please use
%   \begin{thebibliography}{}
%   %
%   % and use \bibitem to create references. Consult the Instructions
%   % for authors for reference list style.
%   %
%   \bibitem{ref.J}
%   % Format for Journal Reference
%   Author, Article title, Journal, Volume, page numbers (year)
%   % Format for books
%   \bibitem{ref.B}
%   Author, Book title, page numbers. Publisher, place (year)
%   \end{thebibliography}

\appendix
\crefalias{section}{appendix}

\section{Sequence verification in the proof of~\cref{thm:stab_asym}}
\label{app:seq}
In this appendix, we verify various conditions for the sequences involved in the proof of~\cref{thm:stab_asym}; %these conditions are designated by~\cref{thm:stab_viet} and~\cref{thm:stab_dec} to ensure the positivity of various trigonometric series.
they are organized according to the order in which these sequences appear in the proof.

\smallskip

{\bf Part 1}. 
The first special Vietories sequence to verify is $\{\mathcal{a}_0-\mathcal{a}_1,\;\mathcal{a}_1-\mathcal{a}_2,\;\cdots,\;\mathcal{a}_{q-1}-\mathcal{a}_q=\mathcal{a}_{q-1}\}$, for which we use~\cref{eq:prelim_ddo_dxx_c_a_pos} and distinguish between $q'=q$ and $q'=q-1$.

Suppose $q'=q$, then using the fact that $\zeta_k^{q,q},\zeta_k^{q',q'}>0$ when $k>0$, we obtain:
\begin{equation}\label{eq:app_viet_asym_a_bal}
  \mathcal{a}_{k-1}-\mathcal{a}_k = \frac{6(2+k(\zeta_k^{q,q}+\zeta_k^{q',q'}))}{k^3}C_k^{q,q}C_k^{q',q'}>0\;,\quad 1\le k\le q\;.
\end{equation}
Using the estimates $\frac{2k}{q-k+1} > \zeta_k^{q,q} > \frac{2k}{q+k+1}$ and $\frac{2k}{q'-k+1} > \zeta_k^{q',q'} > \frac{2k}{q'+k+1}$, for all $2\le k\le q-1$:
%To show $\mathcal{a}_{k-1}-\mathcal{a}_k$ is decreasing in $k$, we use the estimate:
%\begin{displaymath}
%  \frac{2k}{q-k+1} > \zeta_k^{q,q} = \frac{1}{q-k+1}+\frac{1}{q-k+2}+\cdots+\frac{1}{q+k} > \frac{2k}{q+k+1}
%\end{displaymath}
%and a similar one for $\zeta_k^{q',q'}$; then for $2\le k\le q-1$:
\begin{align}
  \notag
   \frac{\mathcal{a}_{k-2}-\mathcal{a}_{k-1}}{\mathcal{a}_{k-1}-\mathcal{a}_k} 
  %= \frac{2+(k-1)(\zeta_{k-1}^{q,q}+\zeta_{k-1}^{q',q'})}{2+k(\zeta_k^{q,q}+\zeta_k^{q',q'})}\cdot\frac{k^3}{(k-1)^3}\cdot\frac{(q+k)(q'+k)}{(q-k+1)(q'-k+1)} \\
  \notag
  &> \frac{2+(k-1)\left(\frac{2(k-1)}{q+k-1+1}+\frac{2(k-1)}{q'+k-1+1}\right)}{2+k\left(\frac{2k}{q-k+1}+\frac{2k}{q'-k+1}\right)}\cdot\frac{k^3}{(k-1)^3}\cdot\frac{(q+k)(q'+k)}{(q-k+1)(q'-k+1)} \\
  \label{eq:app_viet_asym_a_ratio}
  &= \frac{2k^3(q+k)(q'+k)+2k^3(k-1)^2(q+q'+2k)}{2(k-1)^3(q-k+1)(q'-k+1)+2k^2(k-1)^3(q+q'-2k+2)} \ge \frac{k}{k-1}\;.
\end{align}
Hence $\{\mathcal{a}_0-\mathcal{a}_1,\;\mathcal{a}_1-\mathcal{a}_2,\;\cdots,\;\mathcal{a}_{q-1}-\mathcal{a}_q=\mathcal{a}_{q-1}\}$ is a special Vietoris sequence. 
The other case $q'=q-1$ can be handled similarly.
%Suppose $q'=q-1$, then $\mathcal{a}_{k-1}-\mathcal{a}_k$ is still given by (\ref{eq:app_viet_asym_a_bal}) for $1\le k\le q-1$, whereas the last term is:
%\begin{displaymath}
%  \mathcal{a}_{q-1}-\mathcal{a}_q = \mathcal{a}_{q-1} = \frac{6}{q^3}C_q^{q,q}C_q^{q',q}.
%\end{displaymath}
%Noticing in the previous case the inequality (\ref{eq:app_viet_asym_a_ratio}) does not use the assumption $q'=q$, it remains true in the present situation.
%Hence we just need to establish the inequality for the last ratio as follows:
%\begin{align*}
%  \frac{\mathcal{a}_{q-2}-\mathcal{a}_{q-1}}{\mathcal{a}_{q-1}-\mathcal{a}_q} 
%  =&\ \frac{2+(q-1)(\zeta_{q-1}^{q,q}+\zeta_{q-1}^{q-1,q-1})}{1}\cdot\frac{q^3}{(q-1)^3}\cdot\frac{C_{q-1}^{q,q}C_{q-1}^{q-1,q-1}}{C_q^{q,q}C_q^{q-1,q}} \\
%  \ge&\ \frac{2q^3}{(q-1)^3}\cdot\frac{C_{q-1}^{q,q}C_{q-1}^{q-1,q-1}}{C_q^{q,q}C_q^{q-1,q}} = \frac{2q^3}{(q-1)^3}\cdot\frac{2(2q-1)}{1} > \frac{q}{q-1} > 1\;.
%\end{align*}

The second sequence to be considered is $\{-\beta_k^c\,:\;1\le k\le p'\}$ and by using~\cref{eq:prelim_ddo_dx_c_beta} we first check each number is positive:
\begin{displaymath}
  -\beta_k^c = \frac{2}{k}C_k^{p,p}C_k^{p',p'} = \frac{2}{k}\frac{p!p!}{(p+k)!(p-k)!}\frac{p'!p'!}{(p'+k)!(p'-k)!} > 0\;,\quad 1\le k\le p'\;.
\end{displaymath}
Then we compute for $2\le k\le p'$:
\begin{displaymath}
  \frac{-\beta_{k-1}^c}{-\beta_k^c} = \frac{k}{k-1}\cdot\frac{(p+k)(p'+k)}{(p-k+1)(p'-k+1)} > \frac{k}{k-1}\;;
\end{displaymath}
thus the sequence is a special Vietoris one.

\smallskip

{\bf Part 2}. 
By~\cref{eq:prelim_ddo_dxx_c_b}, the second sequence to check $\{-\mathcal{b}_0,\;-2\mathcal{b}_1,\;\cdots,\;-2\mathcal{b}_{q'}\}$ has positive elements and the ratio between consecutive numbers are: 
%\begin{displaymath}
%  -\mathcal{b}_0=\sum_{k=1}^q\frac{6}{k^2}+\sum_{k=1}^{q'}\frac{6}{k^2}>0\;;\quad
%  -2\mathcal{b}_k=\frac{12}{k^2}C_k^{q,q}C_k^{q',q'}>0\;,\ \ 1\le k\le q'\;.
%\end{displaymath}
%The ratios between consecutive numbers are:
\begin{displaymath}
  \frac{-\mathcal{b}_0}{-2\mathcal{b}_1} = \frac{\sum_{k=1}^q(6/k^2)+\sum_{k=1}^{q'}(6/k^2)}{12C_1^{q,q}C_1^{q',q'}}
  = \frac{\sum_{k=1}^q(1/k^2)+\sum_{k=1}^{q'}(1/k^2)}{2\frac{qq'}{(q+1)(q'+1)}} \ge 
    \frac{2}{2\frac{qq'}{(q+1)(q'+1)}} > 1\;;
\end{displaymath}
and for $2\le k\le q'$:
\begin{displaymath}
  \frac{-2\mathcal{b}_{k-1}}{-2\mathcal{b}_k} = \frac{k^2}{(k-1)^2}\cdot\frac{(q+k)(q'+k)}{(q-k+1)(q'-k+1)} > \frac{k}{k-1}\;,
\end{displaymath}
which meets all criteria for a special Vietoris sequence. 

\smallskip

{\bf Part 3}.
For this part we begin with~\cref{eq:stab_asym_eigval_dxx_seq}, %which has been used to prove~\cref{eq:stab_asym_eigval_dxx}.
and will first show $c_0>0$ and $c_k<0$, $1\le k\le q$, and then prove $\sum_{k=0}^q(2k+1)c_k = 0$.
For simplicity, we present the details when $q'=q$; and the proof for the case $q'=q-1$ follows the same procedure.
Recall that:
\begin{displaymath}
  c_k=\left\{\begin{array}{lcl}
    -\mathcal{a}_0+\mathcal{a}_1-2\mathcal{b}_0+2\mathcal{b}_1\;, & & k=0\;, \\
    -\mathcal{a}_{k-1}+\mathcal{a}_{k+1}-2\mathcal{b}_k+2\mathcal{b}_{k+1}\;, & & k\ge1\;.
  \end{array}\right.
\end{displaymath}
Then by~\cref{thm:prelim_ddo_dxx}, there is:
\begin{displaymath}
  c_0 = -6(2+2\zeta^{q,q}_1)\left[C_1^{q,q}\right]^2 + 12H_{q,2} - 6\left[C_1^{q,q}\right]
  = 12H_{q,2}-\frac{6q(q+2)(3q+1)}{(q+1)^3}\;,
\end{displaymath}
where $H_{m,2}=\sum_{k=1}^m(1/k^2)$ is the generalized Harmonic number of order $2$.
It is trivial to see $c_0>12H_{q,2}-18$ thus for $q\ge7$, $c_0\ge12H_{7,2}-18>12\times1.511-18>0$; and one can easily verify $c_0>0$ for $1\le q\le 6$; hence $c_0>0$ for all $q$.

For $1\le k\le q-1$, using $\zeta_k^{q,q}>\frac{2k}{q+1/2}$, $\zeta_{k+1}^{q,q}>\frac{2(k+1)}{q+1/2}$, and writing $q=z+k$:
\begin{align*}
  &\frac{c_k}{12\left[C_k^{q,q}\right]^2} = - \frac{1+k\zeta_k^{q,q}}{k^3} - \frac{1+(k+1)\zeta_{k+1}^{q,q}}{(k+1)^3}\frac{(q-k)^2}{(q+k+1)^2} + \frac{1}{k^2} - \frac{1}{(k+1)^2}\frac{(q-k)^2}{(q+k+1)^2} \\
  <& - \frac{1+\frac{2k^2}{z+k+1/2}}{k^3} - \frac{1+\frac{2(k+1)^2}{z+k+1/2}}{(k+1)^3}\frac{z^2}{(z+2k+1)^2} + \frac{1}{k^2} - \frac{1}{(k+1)^2}\frac{z^2}{(z+2k+1)^2} \\
  =& -\frac{(2k+1)\left[2z^3+(4k^2+10k+5)z^2+4(k+1)^4z+(k+1)^3(2k+1)(2k^2+k+1)\right]}{k^3(k+1)^3(2z+2k+1)(z+2k+1)^2} < 0\;.
\end{align*}
And lastly $c_q = -\mathcal{a}_{q-1} - 2\mathcal{b}_q = -\frac{12\left(1+q\left(H_{2q}-1\right)\right)}{q^3}\left[C_q^{q,q}\right]^2 < 0$.

\begin{remark}\label{rm:app_seq_h2}
  A more refined analysis on $c_0$ shows that for $m>50$, $H_{m,2}>\frac{13}{8}$, therefore:
  \begin{equation}\label{eq:app_seq_h2_low}
    2H_{m,2} - \frac{m(m+2)(3m+1)}{(m+1)^3} > \frac{13}{4} - 3 = \frac{1}{4}\;,
  \end{equation}
  and one can directly verify~\cref{eq:app_seq_h2_low} for $1\le m\le 50$.
  Hence it is actually true that $c_0>\frac{3}{2}$; and this improved lowerbound will be used latter in the proof of~\cref{eq:stab_asym_eigval_dxc_equiv}.
\end{remark}

Next, we show that $c_0 + 3c_1 + 5c_2 + \cdots+ (2q+1)c_q = 0$.
Writing $\mathcal{a}_k=\sum_{j=k+1}^qu_j, 0\le k\le q$, where $u_j=\frac{12(1+k\zeta^{q,q}_k)}{k^3}\left[C_k^{q,q}\right]^2$, one has:
\begin{displaymath}
  \sum_{k=0}^q(2k+1)c_k 
  = -4\sum_{k=1}^qku_k-2b_0-4\sum_{k=1}^q\mathcal{b}_k 
  = 24H_{q,2}-\sum_{k=1}^q\frac{24(1+2k\zeta^{q,q}_k)}{k^2}\left[C^{q,q}_k\right]^2\;,
\end{displaymath}
which is zero by the identity~\cref{eq:app_seq_identity}$_2$ that we will prove soon.

\medskip

To prove~\cref{eq:stab_asym_eigval_dxc_seq}, we show first that $\{d_k\}$ is positive and decreasing and next $\sum_{k=1}^pkd_k = 1$.
Again, we assume $p'=p$ and the other case ($p'=p-1$) follows the same procedure.
Recalling $d_k = \alpha_{k-1}^c+\alpha_k^c+2\beta_k^c$, one computes for $1\le k\le p-1$:
\begin{align*}
  & \frac{d_k-d_{k+1}}{2\left[C_k^{p,p}\right]^2} %= \alpha_{k-1}^c-\alpha_{k+1}^c+2\beta_k^c-2\beta_{k+1}^c \\
  %=& \frac{2(1+2k\zeta^{p,p}_k)}{k^2}\left[C^{p,p}_k\right]^2 + \frac{2(1+2(k+1)\zeta^{p,p}_{k+1})}{(k+1)^2}\left[C^{p,p}_{k+1}\right]^2-\frac{4}{k}\left[C^{p,p}_k\right]^2 + \frac{4}{k+1}\left[C^{p,p}_{k+1}\right]^2 \\
  = \frac{(1+2k(\zeta_k^{p,p}-1))}{k^2}+\frac{\left[1+2(k+1)(\zeta_{k+1}^{p,p}+1)\right]}{(k+1)^2}\frac{(p-k)^2}{(p+k+1)^2} \\
  >& \frac{
    2z^3 + (4k^2+10k+5)z^2 + 4(k+1)^2(2k+1)z + (k+1)^2(2k+1)^2(4k^2+1) 
  }{k^2(k+1)^2(2z+2k+1)(z+2k+1)^2} > 0\;,
\end{align*}
where $z=p-k>0$ and we used again the estimates $\zeta_k^{p,p}\ge\frac{2k}{p+1/2}$ and $\zeta_{k+1}^{p,p}>\frac{2(k+1)}{p+1/2}$.
Thus:
\begin{displaymath}
  d_1>d_2>\cdots>d_p=\alpha_{p-1}^c+2\beta_p^c = \frac{2\left[1+2p\left(H_{2p}-1\right)\right]}{p^2}\left[C_p^{p,p}\right]^2 > 0\;.
\end{displaymath}
Finally, we prove $\sum_{k=1}^pkd_k=1$ and write $\alpha_k^c = \sum_{j=k+1}^pt_j$, where $t_j=\frac{2\left(1+2j\zeta_j^{p,p}\right)}{j^2}\left[C_j^{p,p}\right]^2$:
\begin{equation}\label{eq:app_seq_dxc_mom}
  \sum_{k=1}^pkd_k = \sum_{k=1}^pk\left(\alpha_{k-1}^c+\alpha_k^c+2\beta_k^c\right)
  = \sum_{k=1}^p\left(k^2t_k+2k\beta_k^c\right)
  = \sum_{k=1}^p\left(4k\zeta_k^{p,p}-2\right)\left[C_k^{p,p}\right]^2\;.
\end{equation}
To compute the series sum in the right hand side, we compute the power series expansion of $R(x) = \prod_{k=1}^p\left(1-\frac{x^2}{k^2}\right)^{-2}$ in two different ways.
First, using $\left(1-\frac{x^2}{k^2}\right)^{-1} = 1 + \frac{x^2}{k^2} + O(x^4)$:
\begin{equation}\label{eq:app_seq_dxc_taylor_1}
  R(x) = 1 + 2H_{p,2}x^2 + O(x^4)\;.
\end{equation}
Then, we seek the partial fraction expansion:
\begin{align*}
  R(x) &= \sum_{k=1}^p\left[\frac{R_{k,2}}{(x-k)^2}+\frac{R_{k,2}}{(x+k)^2}+\frac{R_{k,1}}{x-k}-\frac{R_{k,1}}{x+k}\right]\;,
\end{align*}
and it is easy to compute:
\begin{align*}
  R_{k,2} &= \lim_{x\to k}\,(x-k)^2R(x) = \frac{k^2}{4}\prod_{j=1,j\ne k}^p\frac{j^4}{(j+k)^2(j-k)^2} = k^2\left[C_k^{p,p}\right]^2\;, \\
  R_{k,1} &= \left.\frac{d}{dx}\left[(x-k)^2R(x)\right]\right|_{x=k} = 2k\left(1-k\zeta^{p,p}_k\right)\left[C_k^{p,p}\right]^2\;.
\end{align*}
Therefore one has:
\begin{align}
  \notag
  R(x) &= \sum_{k=1}^p\frac{2R_{k,2}-2kR_{k,1}}{k^2} + x^2\sum_{k=1}^p\frac{6R_{k,2}-2kR_{k,1}}{k^4} + O(x^4) \\
  \label{eq:app_seq_dxc_taylor_2}
  &= \sum_{k=1}^p\left(4k\zeta^{p,p}_k-2\right)\left[C_k^{p,p}\right]^2 + x^2\sum_{k=1}^pt_k + O(x^4)\;.
\end{align}
Comparing~\cref{eq:app_seq_dxc_taylor_1} and~\cref{eq:app_seq_dxc_taylor_2}, one has:
\begin{equation}\label{eq:app_seq_identity}
  \sum_{k=1}^p\left(4k\zeta^{p,p}_k-2\right)\left[C_k^{p,p}\right]^2 = 1\;,\quad
  \sum_{k=1}^p\frac{2\left(1+2k\zeta^{p,p}_k\right)}{k^2}\left[C_k^{p,p}\right]^2 = 2H_{p,2}\;.
\end{equation}
The first is precisely $\sum_{k=1}^pkd_k=1$, whereas the lastter gives $\alpha_0^c = 2H_{p,2}$, which allows us to derive the last inequality of~\cref{eq:stab_asym_eigval_dxc_seq}:
\begin{align*}
  d_1 = 4H_{p,2}-2(1+2\zeta^{p,p}_1)\left[C_1^{p,p}\right]^2-4\left[C_1^{p,p}\right]^2
      = 4H_{p,2}-\frac{2p(p+2)(3p+1)}{(p+1)^3} > \frac{1}{2}
\end{align*}
by the remark below~\cref{eq:app_seq_h2_low}.

\section{Proof of lemmas in~\cref{sec:stab_ade}}
\label{app:ade}
In this appendix, we prove the lemmas that are used in the proof of~\cref{thm:stab_ade}. %particularly~\cref{lm:stab_ade_dx} and~\cref{lm:stab_ade_dxx}. %, which give bounds on the ratios of characteristic polynomials associated with central HV discretizations $[\mathcal{D}_x^c]$ and $[\mathcal{D}_{xx}^c]$, respectively.
First we consider~\cref{lm:stab_ade_dx} and recall:
\begin{displaymath}
  f^c(\theta) = 2\alpha^c_0+2\sum_{k=1}^p(\alpha^c_k-\alpha^c_{k-1})\cos k\theta\;,\quad
  h^c(\theta) = -2\sum_{k=1}^{p'}\beta^c_k\sin k\theta\;;
\end{displaymath}
and we intend to show:
\begin{equation}\label{eq:app_ade_dx}
  \theta h^c < f^c < \theta(\theta+h^c)\;,\quad\textrm{ where }\quad
  f^c-2\sin\frac{\theta}{2}\,h^c\ge4\sin^2\frac{\theta}{2}\;,
  %f^c > \theta h^c\;,\quad
  %\theta(\theta+h^c) > f^c\;,\quad
  %\omega(\omega+h^c) < f^c\;,\quad\textrm{ where }\quad
  %\omega=\theta\left(1-\frac{\theta^2}{50}\right)\;.
\end{equation}
and if $p+p'\ge2$, there is in addition:
\begin{equation}\label{eq:app_ade_dx_omega}
  f^c > \omega(\omega+h^c)\quad\textrm{ where}\quad \omega=\theta\left(1-\frac{\theta^2}{50}\right)\;.
\end{equation}
%with $\omega=\theta\left(1-\frac{\theta^2}{18\pi^2}\right)$.
%with $\omega=\theta\left(1-\frac{\theta^2}{50}\right)$.
We assume $p'=p$ in the proof and the case of $p'=p-1$ is parallel. 
Note that we switch the stencil from $r$ and $r'$ to $p$ and $p'$ to make the notations consistent with the previous appendix.%, in comparison with the lemma; the purpose is to make the notations consistent with those in the previous appendix.
By definition, when $p'=p$ the $\alpha$-coefficients and $\beta$-coefficients are:
\begin{align*}
  \alpha^c_k = \sum_{j=k+1}^p\frac{2(1+2j\zeta^{p,p}_j)}{j^2}\left[C^{p,p}_j\right]^2\;,\quad 0\le k\le p\;;\qquad
  \beta^c_k = -\frac{2}{k}\left[C^{p,p}_k\right]^2\;,\quad 1\le k\le p\;.
\end{align*}
The plan is to study the remainder function:
\begin{equation}\label{eq:app_ade_dx_rem}
  R(\theta) = \theta(\theta+h^c(\theta))-f^c(\theta)\;,
\end{equation}
which satisfies:
\begin{displaymath}
  R(0) = -f^c(0) = -\alpha^c_p = 0\;,\quad
  R'(\theta) = 2\theta+h^c+\theta (h^c)' - (f^c)'\;.
\end{displaymath}
In fact, we have the following integral representation of $R(\theta)$ for $0<\theta\le\pi$:
\begin{lemma}\label{lm:app_ade_dx_rem}
  The functions $h^c(\theta)$ and $R(\theta)$ can be written:
  \begin{align}
    \label{eq:app_ade_dx_hc_int}
    %h^c(\theta) &= \frac{s_p^2}{\pi}\int_0^\theta\left[\int_{-\pi}^{\pi}\cos^{2p}\frac{t}{2}\,\cos^{2p}\frac{\theta'-t}{2}\,dt\right]d\theta'-2\theta\;, \\
    h^c(\theta) &= 2\pi-4\pi\hat{\mathscr{B}}\left(\frac{\theta}{2\pi}\right)+4\pi\hat{\mathscr{B}}\left(1-\frac{\theta}{2\pi}\right)-2\theta\;, \\
    \label{eq:app_ade_dx_rem_int}
    R(\theta) &= 8\pi^2\hat{\mathscr{B}}\left(1-\frac{\theta}{2\pi}\right)\;, %2\left[\frac{2^{2p}p!p!}{(2p)!}\right]^2\int_0^{\theta}\left[\int_0^{\theta'}\sin^{2p}\frac{t}{2}\;\sin^{2p}\frac{\theta'-t}{2}\,dt\right]\,d\theta'\;.
  \end{align}
  where
  \begin{equation}\label{eq:app_ade_dx_b}
    \hat{\mathscr{B}}\left(\frac{\theta}{2\pi}\right) = \frac{1}{4\pi^2}\left[\frac{2^{2p}p!p!}{(2p)!}\right]^2\int_0^{2\pi-\theta}\left[\int_0^{\theta'}\sin^{2p}\frac{t}{2}\;\sin^{2p}\frac{\theta'-t}{2}\,dt\right]\,d\theta'\;,\ \forall 0\le\theta\le2\pi\;.
  \end{equation}
\end{lemma}
%\noindent
%Note: All inequalities in~\cref{eq:app_ade_dx} can be derived by estimating the double integral properly.
\begin{proof}
  We define the auxiliary entire function:
  \begin{equation}\label{eq:app_ade_dx_a}
    \mathscr{C}(x) = \left[\frac{[\Gamma(p+1)]^2}{\Gamma(p+1+x)\Gamma(p+1-x)}\right]^2\;,
  \end{equation}
  where $\Gamma(x)$ is the gamma function.
  It is easy to verify that for $1\le k\le p$:
  \begin{displaymath}
    \mathscr{C}(k) = \left[\frac{\Gamma(p+1)}{\Gamma(p+1+k)}\frac{\Gamma(p+1)}{\Gamma(p+1-k)}\right]^2
    = \left[\frac{p!}{(p+k)!}\frac{p!}{(p-k)!}\right]^2 = \left[C^{p,p}_k\right]^2\;.
  \end{displaymath}
  and one can compute by logarithmic derivative that $\mathscr{C}'(k) = -2\zeta_k^{p,p}\mathscr{C}(k)$.
  %The logarithmic derivative of $\mathscr{C}$ is:
  %\begin{displaymath}
  %  \frac{\mathscr{C}'(x)}{\mathscr{C}(x)} = -2\gamma(p+1+x)+2\gamma(p+1-x)\;,
  %\end{displaymath}
  %where $\gamma(x) = \frac{\Gamma'(x)}{\Gamma(x)}$ is the digamma function; therefore:
  %\begin{displaymath}
  %  \frac{\mathscr{C}'(k)}{\mathscr{C}(k)} = -2\gamma(p+1+k)+2\gamma(p+1-k) = -2H_{p+k}+2H_{p-k} = -2\zeta^{p,p}_k\;.
  %\end{displaymath}
  Hence we rewrite:
  \begin{displaymath}
    h^c(\theta) = \sum_{k=1}^p\frac{4}{k}\mathscr{C}(k)\sin k\theta\;,\quad
    f^c(\theta) %= 2\sum_{k=1}^p\frac{2(1+2k\zeta^{p,p}_k)}{k^2}\left[C_k^{p,p}\right]^2\left[1-\cos(k\theta)\right]
    = 2\sum_{k=1}^p\frac{2}{k^2}\left[\mathscr{C}(k)-k\mathscr{C}'(k)\right]\left[1-\cos k\theta\right]\;,%\quad
  \end{displaymath}
  and using $\mathscr{C}(0)=1$, $\mathscr{C}(x)$ is even, and $\mathscr{C}'(x)$ is odd, we have:
  \begin{displaymath}
    {h^c}'(\theta) = 2\sum_{k=-p}^p\mathscr{C}(k)e^{ik\theta}-2\;,\quad
    {f^c}'(\theta) %&= 2\sum_{k=1}^p\frac{2}{k}[\mathscr{C}(k)-k\mathscr{C}'(k)]\sin{k\theta}
    = h^c(\theta) + 2i \sum_{k=-p}^p\mathscr{C}'(k)e^{ik\theta}\;;
  \end{displaymath}
  it follows from~\cref{eq:app_ade_dx_rem}:
  \begin{equation}\label{eq:app_ade_dx_rem_der}
    R'(\theta) %= 2\theta + 4\theta\sum_{k=1}^p\mathscr{C}(k)\cos k\theta+4\sum_{k=1}^p\mathscr{C}'(k)\sin k\theta
    = 2\theta\sum_{k=-p}^p\mathscr{C}(k)e^{ik\theta}-2i\sum_{k=-p}^p\mathscr{C}'(k)e^{ik\theta}\;.
  \end{equation}
  %\begin{align*}
  %  R'(\theta) &= 2\theta + \sum_{k=1}^p\frac{4}{k}\mathscr{C}(k)\sin k\theta+\theta\sum_{k=1}^p4\mathscr{C}(k)\cos k\theta-\sum_{k=1}^p\frac{4}{k}\left[\mathscr{C}(k)-k\mathscr{C}'(k)\right]\sin k\theta \\
  %    &= 2\theta + 4\theta\sum_{k=1}^p\mathscr{C}(k)\cos k\theta+4\sum_{k=1}^p\mathscr{C}'(k)\sin k\theta\;.
  %\end{align*}
  %As $\mathscr{C}(0)=1$, $\mathscr{C}(x)$ is even, and $\mathscr{C}'(x)$ is odd, one immediately sees:
  %\begin{equation}\label{eq:app_ade_dx_rem_der}
  %  R'(\theta) = 2\theta\sum_{k=-p}^p\mathscr{C}(k)e^{ik\theta}-2i\sum_{k=-p}^p\mathscr{C}'(k)e^{ik\theta}\;.
  %\end{equation}
  To proceed, we notice that $\mathscr{S}(x)=\sqrt{\mathscr{C}(x)}$ is a reciprocal beta function~\cite[5.12.5]{FWJOlver:2010a}:
  \begin{displaymath}
    \mathscr{S}(x) = \frac{B(p+1,p+1)}{B(p+1+x,p+1-x)} 
     = s_p\int_{-\hf}^{\hf}\cos^{2p}(\pi t)\;e^{-2\pi ixt}\,dt\;,
  \end{displaymath}
  where $s_p=(2p+1)2^{2p}B(p+1,p+1)=\frac{2^{2p}p!p!}{(2p)!}$.
  Then $\mathscr{S}(x)$ is the Fourier transform of $\hat{\mathscr{S}}$, where $\hat{\mathscr{S}}(t)=s_p\cos^{2p}(\pi t)\,\chi_{[-1/2,1/2]}(t)$ and $\chi$ is the indicator function, 
  Therefore $\mathscr{C}(x)$ is the Fourier transform of the convolution $\hat{\mathscr{C}}\eqdef\hat{\mathscr{S}}\ast\hat{\mathscr{S}}$:
  \begin{equation}\label{eq:app_ade_dx_fourier}
    \mathscr{C}(x) = \int_{-1}^1\hat{\mathscr{C}}(t)e^{-2\pi ixt}\,dt\;.
  \end{equation}
  Because $\hat{\mathscr{C}}$ is supported on $[-1,1]$, the following $1$-period function
  \begin{displaymath}
    \tilde{\mathscr{C}}(t) = \sum_{k\in\mathbb{Z}}\hat{\mathscr{C}}(t+k)
  \end{displaymath}
  is well-defined.
  The complex Fourier coefficient for $k\in\mathbb{Z}$ of $\tilde{\mathscr{C}}$ is:
  \begin{displaymath}
    \int_{-\hf}^{\hf}\tilde{\mathscr{C}}(t')e^{-2\pi ikt'}dt'  
    %\int_{-\hf}^0\left(\hat{\mathscr{C}}(t')\!+\!\hat{\mathscr{C}}(t'\!+\!1)\right)e^{-2\pi ikt'}dt' +
    %\int_0^{\hf}\left(\hat{\mathscr{C}}(t')\!+\!\hat{\mathscr{C}}(t'\!-\!1)\right)e^{-2\pi ikt'}dt' \\
    = \int_{-1}^1\hat{\mathscr{C}}(t')e^{-2\pi ikt'}dt' = \mathscr{C}(k)\;.
    %\sum_{k\in\mathbb{Z}} e^{2\pi i kt}\times\int_{-\hf}^{\hf}\tilde{\mathscr{C}}(t')e^{-2\pi ikt'}dt'
  \end{displaymath}
  Therefore:
  \begin{displaymath}
    \tilde{\mathscr{C}}(t) = \sum_{k\in\mathbb{Z}}\mathscr{C}(k)e^{2\pi ikt}
    = \sum_{k=-p}^p\mathscr{C}(k)e^{2\pi ikt}\;.
  \end{displaymath}
  Here we used a few facts that are straightforward to check: (1) $\tilde{\mathscr{C}}$ is continuously differentiable on $\mathbb{R}$, (2) $\mathscr{C}(x)$ is the square of a reciprocal beta function and thus entire in $x$, and (3) $\mathscr{C}(k)=0$ for all integer $k$ such that $\abs{k}>p$.
  Let $t=\theta/(2\pi)$, one gets for $0<\theta\le\pi$:
  \begin{equation}\label{eq:app_ade_dx_cos}
    1+\frac{1}{2}{h^c}'(\theta) = \sum_{k=-p}^p\mathscr{C}(k)e^{ik\theta} = \tilde{\mathscr{C}}\left(\frac{\theta}{2\pi}\right)
    = \hat{\mathscr{C}}\left(\frac{\theta}{2\pi}\right)+\hat{\mathscr{C}}\left(\frac{\theta}{2\pi}-1\right)\;.
  \end{equation}
  Let us define:
  \begin{displaymath}
    \hat{\mathscr{B}}(t) = \int_{\min(\abs{t},1)}^1\hat{\mathscr{C}}(t')\,dt'\;,
  \end{displaymath}
  then it is easy to verify that $\hat{\mathscr{B}}$ is continuously differentiable on $\mathbb{R}$, compactly supported on $[-1, 1]$, and even; furthermore it is precisely~\cref{eq:app_ade_dx_b} with $t=\theta/(2\pi)$.
  Then using $h^c(0)=0$ and integrating ${h^c}'(\theta)$, one gets:
  \begin{displaymath}
    h^c(\theta) %&= 2s_p^2\int_0^{\theta}\left[\hat{\mathscr{C}}\left(\frac{\theta'}{2\pi}\right)+\hat{\mathscr{C}}\left(\frac{\theta'}{2\pi}-1\right)\right]\,d\theta'-2\theta \\
    = 4\pi\left[\hat{\mathscr{B}}(0)-\hat{\mathscr{B}}\left(\frac{\theta}{2\pi}\right)+\hat{\mathscr{B}}\left(1-\frac{\theta}{2\pi}\right)\right]-2\theta\;,
  \end{displaymath}
  then \cref{eq:app_ade_dx_hc_int} comes from:
  \begin{displaymath}
    \hat{\mathscr{B}}(0) = \frac{1}{2}\int_{\mathbb{R}}\hat{\mathscr{C}}(t)\,dt = \frac{1}{2}\left[\int_{\mathbb{R}}\hat{\mathscr{S}}(t)\,dt\right]^2 = \frac{1}{2}\;.
  \end{displaymath}

  Next we derive an integral form of $R(\theta)$ by using~\cref{eq:app_ade_dx_rem_der}, whose first term is already representedy by $\hat{\mathscr{C}}$ using~\cref{eq:app_ade_dx_cos}; next we focus on the second term and to this end take the derivative of~\cref{eq:app_ade_dx_fourier} with respect to $x$ and obtain:
  \begin{displaymath}
    \mathscr{C}'(x) = \int_{-1}^1(-2\pi i t)\hat{\mathscr{C}}(t)e^{-2\pi ixt}\,dt\;.
  \end{displaymath}
  Define the $1$-periodization of the integrand as $\mathscr{M}(t) = \sum_{k\in\mathbb{Z}}(t+k)\hat{\mathscr{C}}(t+k)$, then $\mathscr{M}$ is continuously differentiable on $\mathbb{R}$ and it has the complex Fourier series:
  \begin{align*}
    \mathscr{M}(t) = \sum_{k\in\mathbb{Z}}e^{2\pi ikt}\int_{-\hf}^{\hf}\mathscr{M}(t)e^{-2\pi ikt'}dt' 
    = \sum_{k\in\mathbb{Z}}\frac{\mathscr{C}'(k)}{-2\pi i}e^{2\pi ikt}
    = \sum_{k=-p}^p\frac{\mathscr{C}'(k)}{-2\pi i}e^{2\pi ikt}\;. 
  \end{align*}
  Here we used the fact that $\mathscr{C}'(x)$ is entire and $\mathscr{C}'(k)=0$ for all integers $\abs{k}>p$.
  %thanks to the fact that it is the square of a reciprocal beta function.
  Substituting $t=\theta/(2\pi)$ gives for $0<\theta\le\pi$:
  \begin{equation}\label{eq:app_ade_dx_cos_der}
    i\sum_{k=-p}^p\mathscr{C}'(k)e^{ik\theta} = 2\pi\mathscr{M}\left(\frac{\theta}{2\pi}\right)
    = 2\pi\left[\frac{\theta}{2\pi}\hat{\mathscr{C}}\left(\frac{\theta}{2\pi}\right)+\left(\frac{\theta}{2\pi}-1\right)\hat{\mathscr{C}}\left(\frac{\theta}{2\pi}-1\right)\right]
  \end{equation}
  Using~\cref{eq:app_ade_dx_cos} and~\cref{eq:app_ade_dx_cos_der} in~\cref{eq:app_ade_dx_rem_der}, there is:
  \begin{displaymath}
    R'(\theta) = 4\pi\hat{\mathscr{C}}\left(\frac{\theta}{2\pi}-1\right)\quad\Rightarrow\quad
    R(\theta) = 8\pi^2\hat{\mathscr{B}}\left(1-\frac{\theta}{2\pi}\right)\;,
  \end{displaymath}
  which completes the proof.
  %Plugging~\cref{eq:app_ade_dx_cos} and~\cref{eq:app_ade_dx_cos_der} to~\cref{eq:app_ade_dx_rem_der} yields~\cref{eq:app_ade_dx_rem_int}.

  %At last we can write~\cref{eq:app_ade_dx_rem_der} for $0<\theta\le\pi$ as:
  %\begin{align*}
  %  R'(\theta) &= 2\theta s_p^2\left[\hat{\mathscr{C}}\left(\frac{\theta}{2\pi}\right)+\hat{\mathscr{C}}\left(\frac{\theta}{2\pi}-1\right)\right] - 2s_p^2\left[\theta\hat{\mathscr{C}}\left(\frac{\theta}{2\pi}\right)+\left(\theta-2\pi\right)\hat{\mathscr{C}}\left(\frac{\theta}{2\pi}-1\right)\right] \\
  %  &= 4\pi s_p^2\hat{\mathscr{C}}\left(\frac{\theta}{2\pi}-1\right) 
  %  = 4\pi s_p^2\int_{-\hf}^{\frac{\theta}{2\pi}-\hf}\cos^{2p}(\pi t)\,\cos^{2p}\left(\frac{\theta}{2}-\pi t\right)\,dt \\
  %  &= 2s_p^2\int_0^\theta\sin^{2p}\frac{t}{2}\,\sin^{2p}\frac{\theta-t}{2}\,dt\;,
  %\end{align*}
  %and~\cref{eq:app_ade_dx_rem_int} follows directly.
\end{proof}

Let us continue to prove~\cref{eq:app_ade_dx}.
The right half of the first inequality comes immediately from the fact that $\mathscr{B}>0$, which gives:
\begin{displaymath}
  R(\theta)>0\quad\Rightarrow\quad
  \theta(\theta+h^c)>f^c\;.
\end{displaymath}
To show the other side of the inequality~\cref{eq:app_ade_dx}$_1$, combining~\cref{eq:app_ade_dx_rem_int} and~\cref{eq:app_ade_dx_b} we get:
\begin{equation}\label{eq:app_ade_dx_rem_int_exp}
  R(\theta) = 
    2\left[\frac{2^{2p}p!p!}{(2p)!}\right]^2\int_0^{\theta}\left[\int_0^{\theta'}\sin^{2p}\frac{t}{2}\;\sin^{2p}\frac{\theta'-t}{2}\,dt\right]\,d\theta'\;,
\end{equation} 
whose inner integral satisfies:
\begin{displaymath}
  \frac{1}{\theta'}\int_0^{\theta'}\sin^{2p}\frac{t}{2}\;\sin^{2p}\frac{\theta'-t}{2}\,dt
  = \int_0^1\sin^{2p}\frac{\theta't}{2}\;\sin^{2p}\frac{\theta'(1-t)}{2}\,dt\;,
\end{displaymath}
which is clearly a decreasing function in $\theta'$ on $(0,\pi]$.
Therefore:
\begin{displaymath}
  \int_0^{\theta'}\sin^{2p}\frac{t}{2}\;\sin^{2p}\frac{\theta'-t}{2}\,dt \le \frac{\theta'}{\pi}\int_0^\pi\sin^{2p}\frac{t}{2}\;\cos^{2p}\frac{t}{2}\,dt
  = \theta'B\left(p+\hf,p+\hf\right) = \frac{\theta'}{2^{2p}s_p}\;.
\end{displaymath}
Substituting it into~\cref{eq:app_ade_dx_rem_int}, one has:
\begin{displaymath}
  R(\theta) < \frac{s_p}{2^{2p}}\theta^2 \le \frac{1}{2}\theta^2 < \theta^2\;,
\end{displaymath}
and consequently:
\begin{displaymath}
  \theta(\theta+h^c) - f^c < \theta^2\quad\Rightarrow\quad
  \theta h^c < f^c\;.
\end{displaymath}

To show~\cref{eq:app_ade_dx_omega}, we compute:
\begin{displaymath}
  f^c-\omega(\omega+h^c) %&= \theta(\theta+h^c)-R(\theta)-\omega(\omega+h^c) 
  = \frac{\theta^3}{50}(\theta+\omega+h^c)-R(\theta) %\\
  > \frac{\theta^3}{50}\frac{100\theta-\theta^3}{50}-R(\theta)
  > \frac{(100-\pi^2)\theta^4}{50^2} - R(\theta)\;;
\end{displaymath}
hence we need an estimate for $R(\theta)$ sharper than $\sim\theta^2$.
Consider the inner integrand of~\cref{eq:app_ade_dx_rem_int_exp}:
\begin{displaymath}
  \int_0^{\theta'}\sin^{2p}\frac{t}{2}\;\sin^{2p}\frac{\theta'-t}{2}\,dt < 
  \frac{1}{2^{4p}}\int_0^{\theta'}t^{2p}(\theta'-t)^{2p}\,dt %= \frac{(\theta')^{4p+1}}{2^{4p}}\int_0^1t^{2p}(1-t)^{2p}\,dt \\
  %&= \frac{(\theta')^{4p+1}}{2^{4p}}B(2p+1,2p+1) 
  = \frac{(\theta')^{4p+1}}{2^{4p}}\frac{(2p)!(2p)!}{(4p+1)!}\;,
\end{displaymath}
therefore:
\begin{displaymath}
  R(\theta) < \frac{2s_p^2}{4p+2}\frac{(2p)!(2p)!}{2^{4p}(4p+1)!}\theta^{4p+2}
  = \frac{2(p!)^4}{(4p+2)!}\theta^{4p+2}\;.
\end{displaymath}
Using induction and $0<\theta\le\pi$ one can show $R(\theta) < \frac{2\pi^2}{6!}\theta^4$ for $p\ge1$, which completes the proof as it is easily verified $\frac{1}{25}-\frac{\pi^2}{50^2} > \frac{2\pi^2}{6!}$.

The last piece in~\cref{lm:stab_ade_dx} to prove is~\cref{eq:app_ade_dx}$_2$; and to this end we compute:
\begin{align*}
  & f^c(\theta)-2\sin\frac{\theta}{2}\,h^c(\theta) = \theta^2+\left(\theta-2\sin\frac{\theta}{2}\right)h^c(\theta)-R(\theta) \\
  =& \theta^2+2(\pi-\theta)\left(\theta-2\sin\frac{\theta}{2}\right)-4\pi\left(\theta-2\sin\frac{\theta}{2}\right)\hat{\mathscr{B}}\left(\frac{\theta}{2\pi}\right)-4\pi\left(2\pi+2\sin\frac{\theta}{2}-\theta\right)\hat{\mathscr{B}}\left(1-\frac{\theta}{2\pi}\right)
\end{align*}
It is not difficult to see that $\hat{\mathscr{B}}\left(\frac{\theta}{2\pi}\right)$ decreases in $p$ fixing any $0\le\theta\le2\pi$; therefore $f^c-2\sin\frac{\theta}{2}h^c$ increases in $p$ and when $p=1$, one has $f^c=8\sin^2\frac{\theta}{2}$ and $h^c=\sin\theta$, therefore for general $p$:
\begin{displaymath}
  f^c-2\sin\frac{\theta}{2}\,h^c \ge 8\sin^2\frac{\theta}{2}-2\sin\frac{\theta}{2}\,\sin\theta \ge 4\sin^2\frac{\theta}{2}\left(2-\cos\frac{\theta}{2}\right) \ge 4\sin^2\frac{\theta}{2}\;.
\end{displaymath}

\medskip

Then we consider~\cref{lm:stab_ade_dxx}, where:
\begin{displaymath}
  c(\theta) = 2\sum_{k=1}^q(\mathcal{a}_{k-1}-\mathcal{a}_k)\sin k\theta\;,\quad
  b(\theta) = -\mathcal{b}_0-2\sum_{k=1}^{q'}\mathcal{b}_k\cos k\theta\;.
\end{displaymath}
Assuming again for simplicity $q'=q$, one has:
\begin{displaymath}
  c(\theta) = \sum_{k=1}^q\frac{24(1+k\zeta^{q,q}_k)}{k^3}\left[C^{q,q}_k\right]^2\sin k\theta\;,\quad
  b(\theta) = 12H_{q,2} + \sum_{k=1}^q\frac{12}{k^2}\left[C^{q,q}_k\right]^2\cos k\theta\;;
\end{displaymath}
and the purpose is to prove:
\begin{equation}\label{eq:app_ade_dxx}
  %c < \omega b = \theta\left(1-\frac{\theta^2}{18\pi^2}\right)b\;.
  c < \omega b = \theta\left(1-\frac{\theta^2}{50}\right)b\;,\quad
  \theta b - c \ge 6(\theta-\sin\theta)
\end{equation}
and a bound on $b(\theta)$:
\begin{equation}\label{eq:app_ade_dxx_b}
  b(\theta) < 4\pi^2-6\pi\theta+3\theta^2\;.
\end{equation}
Abusing the notations slightly, we define the remainder function:
\begin{displaymath}
  R(\theta) = \theta b(\theta) - c(\theta)\;,
\end{displaymath}
then one computes:
\begin{displaymath}
  R(0) = -c(0) = 0\;,\quad
  R'(\theta) = b+\theta b'-c'\;,\quad
  R''(\theta) = 2b'+\theta b''-c''\;.
\end{displaymath}
Following~\cref{eq:acry_semi_fun_dxx_c} and~\cref{eq:acry_semi_approx_dxx_c}, we see:
\begin{displaymath}
  \mathcal{a}(\theta) + \mathcal{b}(\theta) = \frac{ic}{(i\theta)^3} + \frac{-b}{(i\theta)^2} = 1+O(\theta)\ \Rightarrow\ 
  \theta b - c = \theta^3 + O(\theta^4)\ \Rightarrow\ 
  R'(0) = 0\;.
\end{displaymath}
With the same definition of $\mathscr{C}$ as before, but replacing $p$ by $q$:
\begin{displaymath}
  \mathscr{C}(x) = \left[\frac{[\Gamma(q+1)]^2}{\Gamma(q+1+x)\Gamma(q+1-x)}\right]^2\;,
\end{displaymath}
we immediately get:
\begin{displaymath}
  b(\theta) = 12H_{q,2}+\sum_{k=1}^q\frac{12}{k^2}\mathscr{C}(k)\cos k\theta\;,\quad
  c(\theta) = \sum_{k=1}^q\frac{12}{k^3}(2\mathscr{C}(k)-k\mathscr{C}'(k))\sin k\theta\;;
\end{displaymath}
taking the derivative twice gives:
\begin{align*}
  &b''(\theta) = 6-6\sum_{k=-q}^q\mathscr{C}(k)e^{ik\theta}\;,\quad
  c''(\theta) = 2b'(\theta)-6i\sum_{k=-q}^q\mathscr{C}'(k)e^{ik\theta}\;, \\
  \textrm{and}\quad&R''(\theta) = 6\theta\left(1-\sum_{k=-q}^q\mathscr{C}(k)e^{ik\theta}\right)+6i\sum_{k=-q}^q\mathscr{C}'(k)e^{ik\theta}\;.
\end{align*}
Defining $\hat{\mathscr{C}}$ and $\hat{\mathscr{B}}()$ as before but with $p$ replaced by $q$, there is:
\begin{displaymath}
  b''(\theta) = 6-6\hat{\mathscr{C}}\left(\frac{\theta}{2\pi}\right)-6\hat{\mathscr{C}}\left(\frac{\theta}{2\pi}-1\right)\;,\quad
  R''(\theta) = 6\theta-12\pi\hat{\mathscr{C}}\left(\frac{\theta}{2\pi}-1\right)\;,
\end{displaymath}
and integrating them twice leads to:
\begin{align}
  \label{eq:app_ade_dxx_b_int}
  b(\theta) &= b(0) + 3\theta^2-6\theta\pi+12\pi\int_0^{\theta}\left[\hat{\mathscr{B}}\left(\frac{\theta'}{2\pi}\right)-\hat{\mathscr{B}}\left(1-\frac{\theta'}{2\pi}\right)\right]d\theta'\;, \\
  \notag
  R(\theta) &= \theta^3-24\pi^2\int_0^\theta\hat{\mathscr{B}}\left(1-\frac{\theta'}{2\pi}\right)\,d\theta' \\
  \label{eq:app_ade_dxx_rem_int}
  &= \theta^3-6\left[\frac{2^{2q}q!q!}{(2q)!}\right]^2\int_0^\theta\int_0^{\theta'}\int_0^{\theta''}\sin^{2q}\frac{t}{2}\,\sin^{2q}\frac{\theta''-t}{2}\,dt\,d\theta''\,d\theta'
\end{align}

%\begin{displaymath}
%  R''(\theta) %&= -\sum_{k=1}^q\frac{24}{k}\mathscr{C}(k)\sin k\theta - \theta \sum_{k=1}^q12\mathscr{C}(k)\cos k\theta + \sum_{k=1}^q\frac{12}{k}(2\mathscr{C}(k)-k\mathscr{C}'(k))\sin k\theta \\
%  = -\theta\sum_{k=1}^q12\mathscr{C}(k)\cos k\theta - \sum_{k=1}^q12\mathscr{C}'(k)\sin k\theta %\\
%  = 6\theta - 6s_q^2\int_0^\theta\sin^{2q}\frac{t}{2}\;\sin^{2q}\frac{\theta-t}{2}\,dt\;,
%\end{displaymath}
%where we used the calculation in the proof of~\cref{lm:app_ade_dx_rem}.
%As $R(0)=R'(0)=0$, we obtain:
%\begin{equation}\label{eq:app_ade_dxx_rem_int}
%  R(\theta) = \int_0^\theta\left\{\int_0^{\theta_1}\left[6\theta_2-6s_q^2\int_0^{\theta_2}\sin^{2q}\frac{t}{2}\;\sin^{2q}\frac{\theta_2-t}{2}\,dt\right]d\theta_2\right\}d\theta_1\;.
%\end{equation}
Both inequalities in~\cref{eq:app_ade_dxx} can be derived from~\cref{eq:app_ade_dxx_rem_int}.
Given $0<\theta\le\pi$, there is:
\begin{align*}
  &\int_0^{\theta''}\sin^{2q}\frac{t}{2}\;\sin^{2q}\frac{\theta''-t}{2}\,dt < \frac{(\theta'')^{4q+1}}{2^{4q}}\frac{(2q)!(2q)!}{(4q+1)!}\quad 
  \Rightarrow\quad
  %R''(\theta_2) > %6\theta_2-6s_q^2\frac{\theta_2^{4q+1}}{2^{4q}}\frac{(2q)!(2q)!}{(4q+1)!}
  %6\theta_2 - \frac{6(q!)^4}{(4q+1)!}\theta_2^{4q+1} \\
  %&\Rightarrow\quad
  %R'(\theta_1) > 3\theta_1^2 - \frac{6(q!)^4}{(4q+2)!}\theta_1^{4q+2} \quad\Rightarrow\quad
  R(\theta) > \theta^3-\frac{6(q!)^4}{(4q+3)!}\theta^{4q+3}\;.
\end{align*}
It is easy to see that the last lower bound is smallest when $q=1$; therefore for $0<\theta\le\pi$:
\begin{displaymath}
  R(\theta) > \left(1-\frac{6\pi^4}{7!}\right)\theta^3 = \left(1-\frac{\pi^4}{840}\right)\theta^3\ \Rightarrow\ 
  \omega b - c = R(\theta) - \frac{\theta^3b}{50} > \left(1-\frac{\pi^4}{840}-\frac{b}{50}\right)\theta^3\;.
\end{displaymath}
Since $C^{q,q}_k<1$ for all $q$ and $1\le k\le q$, $b(\theta)$ is uniformly bounded:
\begin{displaymath}
  b(\theta) < 12H_{q,2}+\sum_{k=1}^q\frac{12}{k^2} = 24H_{q,2} < 24\sum_{k=1}^\infty\frac{1}{k^2} = 4\pi^2\;.
\end{displaymath}
Hence $\omega b - c > \left(1-\frac{\pi^4}{840}-\frac{4\pi^2}{50}\right)\theta^3 > 0$, which completes the proof of~\cref{eq:app_ade_dxx}$_1$.
To show~\cref{eq:app_ade_dxx}$_2$, we note that $s_q^2\sin^{2q}\frac{t}{2}\,\sin^{2q}\frac{\theta_2-t}{2}$ is a decreasing function of $q$, thus:
\begin{align*}
  \theta b-c &=R(\theta) \ge \theta^3-24\int_0^\theta\int_0^{\theta'}\int_0^{\theta''}\sin^2\frac{t}{2}\;\sin^2\frac{\theta''-t}{2}\,dt\,d\theta''\,d\theta' \\
  &= 6(\theta-\sin\theta)+3(2\theta+\theta\cos\theta-3\sin\theta) > 6(\theta-\sin\theta)\;.
\end{align*}
Lastly, we define $D(\theta)=4\pi^2-6\pi\theta+3\theta^2-b(\theta)$, then~\cref{eq:app_ade_dxx_b} is equivalent to show $D(\theta)>0$ for $0<\theta<\pi$.
By~\cref{eq:app_ade_dxx_rem_int}, we have $R(\pi)<\pi^3$, or equivalently:
\begin{displaymath}
  \pi b(\pi) = R(\pi) + c(\pi) = R(\pi) < \pi^3\ \Rightarrow\ b(\pi) < \pi^2\ \Rightarrow\ D(\pi) = \pi^2-b(\pi) > 0\;.
\end{displaymath}
Additionally, computing $D'(\theta)$ gives:
\begin{displaymath}
  D'(\theta) = -6\pi + 6\theta - b'(\theta) = -12\pi\hat{\mathscr{B}}\left(\frac{\theta}{2\pi}\right) + 12\pi\hat{\mathscr{B}}\left(1-\frac{\theta}{2\pi}\right) < 0\;,
\end{displaymath}
where we used the fact that $\hat{\mathscr{B}}$ decreases strictly on $[0, 1]$ and $0<\theta<\pi$.
It follows naturally that $D(\theta)>0$ and completes the proof.

\end{document}